\documentclass[a4paper,american,reqno]{amsart}

\usepackage[utf8]{inputenc}
\usepackage[T1]{fontenc}
\usepackage{babel}
\usepackage[binary-units=true]{siunitx}
\usepackage[ruled,linesnumbered]{algorithm2e}
\usepackage{csquotes}
\usepackage{todonotes}
\usepackage{booktabs}
\usepackage{longtable}
\usepackage{comment}
\usepackage{subcaption}
\usepackage[style = authoryear-comp,
            maxbibnames = 100,
            maxcitenames = 2,
            giveninits = true,
            uniquename = init,
            isbn = false,
            backend = bibtex]{biblatex}
\usepackage[colorlinks,
            citecolor=blue,
            urlcolor=blue,
            linkcolor=blue]{hyperref} 

\makeatletter
\patchcmd{\@settitle}{\uppercasenonmath\@title}{\scshape\large}{}{}
\patchcmd{\@setauthors}{\MakeUppercase}{\scshape\normalsize}{}{}
\makeatother

\vfuzz \hfuzz
\makeatletter
\@namedef{subjclassname@2020}{%
  \textup{2020} Mathematics Subject Classification}
\makeatother

\newtheorem{prop}{Proposition}

\newtheorem{thm}{Theorem}
\newtheorem{lem}[thm]{Lemma}

\SetKwInOut{Input}{Input}
\SetKwInOut{Output}{Output}

\newcommand{\st}{\text{s.t.}}
\newcommand{\define}{\mathrel{{\mathop:}{=}}}

\makeatletter
\newenvironment{varsubequations}[1]
{%
  \addtocounter{equation}{-1}%
  \begin{subequations}
    \renewcommand{\theparentequation}{#1}%
    \def\@currentlabel{#1}%

  }
  {%
  \end{subequations}\ignorespacesafterend
}
\makeatother

\newcommand\MyBox[2]{
  \fbox{\lower0.75cm
    \vbox to 1.7cm{\vfil
      \hbox to 1.7cm{\hfil\parbox{1.4cm}{#1\\#2}\hfil}
      \vfil}%
  }%
}

\newcommand{\defset}[3][\defsep]{\set{#2#1#3}}
\newcommand{\Defset}[3][\defsep]{\Set{#2#1#3}}
\newcommand{\set}[1]{\{#1\}}

\newcommand{\field}{\mathbb}

\newcommand{\reals}{\field{R}}

\newcommand{\R}{\reals}

\DeclareMathOperator*{\argmin}{arg\,min}
\newcommand{\ot}{\leftarrow}

\newcommand{\AC}{\text{AC}}
\newcommand{\TP}{\text{TP}}
\newcommand{\TN}{\text{TN}}
\newcommand{\FP}{\text{FP}}
\newcommand{\FN}{\text{FN}}
\newcommand{\PR}{\text{PR}}
\newcommand{\RE}{\text{RE}}
\newcommand{\FPR}{\text{FPR}}
\newcommand{\SVM}{\text{SVM}}
\newcommand{\true}{\text{true}}

\newcommand{\codename}[1]{\textsf{#1}}
\newcommand{\sign}{{\rm sign}}

\newcommand{\rev}[1]{#1}

\bibliography{constrained-svm}

\begin{document}

\title[MIQP and Iterative Clustering for Semi-Supervised SVMs]%
{Mixed-Integer Quadratic Optimization\\and Iterative Clustering
  Techniques for\\Semi-Supervised Support Vector Machines}

\author[J. P. Burgard, M. E. Pinheiro, M. Schmidt]%
{Jan Pablo Burgard, Maria Eduarda Pinheiro, Martin Schmidt}

\address[J. P. Burgard]{%
  Trier University,
  Department of Economic and Social Statistics,
  Universitätsring 15,
  54296 Trier,
  Germany}
\email{burgardj@uni-trier.de}

\address[M. E. Pinheiro, M. Schmidt]{%
  Trier University,
  Department of Mathematics,
  Universitätsring 15,
  54296 Trier,
  Germany}
\email{pinheiro@uni-trier.de}
\email{martin.schmidt@uni-trier.de}

\date{\today}

\begin{abstract}
  Among the most famous algorithms for solving classification problems are
support vector machines (SVMs), which find a separating hyperplane for
a set of labeled data points.
In some applications, however, labels are only available for a subset
of points.
Furthermore, this subset can be non-representative, e.g., due to
self-selection in a survey.
Semi-supervised SVMs tackle the setting of labeled and
unlabeled data and can often improve the reliability of the
results. Moreover, additional information about the size of the
classes can be available from undisclosed sources. We propose a
mixed-integer quadratic optimization (MIQP) model that covers the
setting of labeled and unlabeled data points as well as the overall
number of points in each class. Since the MIQP's solution time
rapidly grows as the number of variables increases, we introduce
an iterative clustering approach to reduce the model's size.
Moreover, we present an update rule for the required big-$M$
values, prove the correctness of the iterative clustering method as
well as derive tailored dimension-reduction and warm-starting
techniques.
Our numerical results show that our approach leads to a similar
accuracy and precision than the MIQP formulation but at much lower
computational cost.
Thus, we can solve larger problems.
With respect to the original SVM formulation, we observe that our
approach has even better accuracy and precision for biased samples.


\end{abstract}

\keywords{Semi-Supervised Learning,
Support Vector Machines,
Clustering,
Mixed-Integer Quadratic Optimization.%
%
%
}
\subjclass[2020]{90C11,
90C90,
90-08,
68T99%
%
%
}

\maketitle

\section{Introduction}
\label{sec:introduction}

Support vector machines (SVMs) are a standard approach for supervised
binary classification \parencite{boser1992,cortes1995support}. The
core idea is to find a separating hyperplane that optimally splits
the feature space in a positive and a negative side according to the
positive and negative labels of the data.

Obtaining labels for all units of interest can be costly. This is
especially the case if one has to do a classic survey to obtain the
labels. In this case, it would be favorable to train the SVM on only
partly labeled data. This yields a semi-supervised learning
setting. \textcite{Bennett1998SemiSupervisedSV} formulate and solve
the semi-supervised SVM (S$^3$VM) as a mixed-integer linear problem
(MILP). Many strategies for solving S$^3$VM have been proposed in the
following decades such as the transductive approach (TSVM) by
\textcite{TSVM,TSVM2} or manifold regularization (LapSVM) by
\textcite{LAPSVM,LAPSVM2}. Some researchers also consider a balancing
constraint as done in meanS3VM by \textcite{KONTONATSIOS201767} and in
c$^3$SVM by \textcite{Chapelle}.
Moreover, the balancing constraint proposed by \textcite{Chapelle2}
enforces that the proportion of unlabeled and labeled data on both
sides is similar to the proportion given by the labeled data.

In many cases, however, the aggregated information about the number of
positive and negative cases in a population is known from an external
source. For example, in population surveys, there are population
figures from official statistics agencies.
\rev{This setting is studied, e.g., by \textcite{Burgard2021CSDA}, who
  develop a cardinality-constrained multinomial logit model and apply
  it in the context of micro-simulations.} As another example, in
some businesses, the total amount of positive labels could be known
but not which customer has a positive or a negative label.
An intuitive example is a supermarket for which the amount of cash
payments is known. However, this information
is not ex-post attributable to the individual customers. We propose to
add this aggregated additional information to the optimization
model by imposing a cardinality constraint on the predicted labels
for the unlabeled data. As will be shown in our numerical experiments,
this improves the accuracy of the classification of the unlabeled
data. Furthermore, the inclusion of such a cardinality constraint is
very useful in the case in which the labeled data
is not a representative sample from the population. When obtaining the
labels from process data or from online surveys, the inclusion
process of the labeled data is generally not known. This is subsumed
under the non-probability sample. In this case, inverse inclusion
probability weighting, as typically done in survey sampling, is not
applicable.
By not controlling the inclusion process, strong over- or
under-coverage of relevant information in the data set is possible and
should be taken into account in the analysis.
Not accounting for possible biases in the data generally leads to
biased results.

We propose a big-$M$-based MIQP to solve the semi-supervised SVM
problem with a cardinality constraint for the unlabeled data.
\rev{Here, we restrict ourselves to the linear kernel. Other kernels
  such as Gaussian and polynomial ones can, in principle, be used as
  well.
  However, this would lead to additional nonlinear constraints in a
  our mixed-integer model and would thus significantly increase the
  computational challenge of solving the problem.
  Although we strongly suspect that the problem is NP-hard, we have no
  proof for it since we focus here on solution techniques and not on a
  formal complexity analysis of the problem.}
The cardinality constraint helps to account for biased samples since
the number of positive predictions on the population is bounded by the
constraint.
The computation time for this MIQP grows rapidly with
the number of variables---especially for an increasing number of
integer variables. We develop an algorithm that uses a
clustering-based model reduction to reduce the computation
time. Similar reduction approaches can be found for the classic
SVM using, e.g., fuzzy clustering \parencite{Fuzzy1,Fuzzy2},
clustering-based convex hulls \parencite{CBCH}, and $k$-means
clustering \parencite{Kmeans1,kmeans2}.
We prove the correctness of our iterative clustering method and
further show that it computes feasible points for the original
problem.
Hence, it also delivers proper upper bounds.
Within our iterative approach, we additionally derive a scheme
for updating the required big-$M$ values and present tailored
dimension-reduction as well as warm-starting techniques.

The paper is organized as follows.
In Section~\ref{sec:miqp-formulation}, we describe our optimization
problem and the big-$M$-based MIQP formulation.
Afterward, the clustering-based model reduction technique is presented
in Section~\ref{sec:cluster-unlabeled-points}. There, we also present
our algorithm that combines the model reduction and the MIQP
formulation. In Section~\ref{Some Improvements}, we discuss some
algorithmic improvements such as the handling of data points that are
far away from the hyperplane and the choice of $M$ in the big-$M$
formulation. In Section~\ref{section c2svm3+}, we present how to use
the solution of our algorithm to obtain the solution of the initial
MIQP formulation by fixing some points on the correct side of the hyperplane.
Finally, in Section~\ref{sec:numerical-results}, numerical results are
reported and discussed and we conclude in Section~\ref{sec:conclusion}.


\section{An MIQP Formulation for a Cardinality-Constrained
  Semi-Supervised SVM}
\label{sec:miqp-formulation}

Let $X\in \R^{d \times N}$ be the data matrix with $X_l =
[x^1, \dotsc, x^n]$ being the labeled data and $X_u = [x^{n+1},
  \dotsc, x^N]$ being the unlabeled data. Hence, we have $x^i \in
  \mathbb{R}^d$ for all $i \in [1,N] \define \set{1,\dots, N}$.
We set $m \define N - n$ and $y \in \set{-1,1}^n$ is the vector of
class labels for the labeled data.
When the data is linearly separable, the SVM provides a
hyperplane $(\omega, b)$ that separates the positively and negatively
labeled data.
In the case that the data is not linearly separable, the standard
approach is to use the $\ell_2$-SVM by \textcite{cortes1995support}
given by
\begin{varsubequations}{P1}
  \label{l2svm}
  \begin{align}
    \min_{\omega,b,\xi} \quad
    & \frac{\Vert\omega \Vert^2}{2} + C_1 \sum_{i=1}^n \xi_i
    \\
    \st \quad
    & y_i (\omega^\top x^i -b) \geq 1 - \xi_i, \quad i \in [1,n],
    \\
    & \xi_i \geq 0, \quad i \in [1, n].
  \end{align}
\end{varsubequations}
Here and in what follows, $\Vert \cdot \Vert$ denotes the Euclidean
norm.
\rev{However, other norms such as the $1$- or the max-norm could be
  used as well.}
For being able to include unlabeled data in the optimization process,
\textcite{Bennett1998SemiSupervisedSV} propose the semi-supervised SVM
(S$^3$VM).
In many applications, the aggregated information on the labels is
available, e.g., from census data.
In the following, we know the total number~$\tau$ of positive labels
for the unlabeled data from an external source.
We adapt the idea of the S$^3$VM such that we can use $\tau$ as
an additional information in the optimization model.
Our goal is to find optimal parameters $\omega^* \in \R^d$,
$b^* \in \R$, $\xi^* \in \R^n$, and $\eta^* \in \R^2$ that solve the
optimization problem
\begin{varsubequations}{P2}
  \label{EQproblem1}
  \begin{align}
    \min_{\omega,b,\xi,\eta} \quad
    & \frac{\Vert\omega \Vert^2}{2} + C_1 \sum_{i=1}^n
      \xi_i +C_2 (\eta_1 + \eta_2) \label{svmfunction}
    \\
    \st \quad
    & y_i (\omega^\top x^i -b) \geq 1 - \xi_i, \quad i \in
      [1,n], \label{labeledpart}
    \\
    & \tau - \eta_ 1 \leq  \sum_{i=n+1}^N h_{\omega,b}(x^i) \leq \tau +
      \eta_2,
      \label{unlabeled part} \\
    & \xi_i  \geq 0, \quad i \in [1, n] \label{xi1}, \\
    & \eta_1, \eta_2 \geq 0, \label{eta1}
  \end{align}
\end{varsubequations}
with
\begin{equation*}
  h_{\omega,b}(x) =
  \begin{cases}
    1, & \text{if } \omega^\top x + b \geq 0,\\
    0, & \text{otherwise}.
  \end{cases}
\end{equation*}

Note that the objective function in \eqref{svmfunction} is a
compromise between maximizing the distance between the two classes as
well as minimizing the classification error for the label and the
unlabeled data.
The penalty parameters $C_1 > 0 $ and $C_2>0$ aim to control the
importance of the slack variables $\xi$ and $\eta$, respectively.
Constraint~\eqref{labeledpart} enforces on which side of the
hyperplane the labeled data $x^i$ should lie.
Constraint~\eqref{unlabeled part} ensures that we have $\tau$
unlabeled data on the positive side.
\rev{If $\eta^*_1>0$ holds for a solution $(\omega^*, b^*,\xi^*, \eta^*)$,
  then less than $\tau$ unlabeled points are classified as positive.
  On the other hand, if $\eta^*_2 > 0$ holds, more than $\tau$
  unlabeled points are classified as positive.
  If $\eta^*_1 = \eta^*_2 =0$ holds, exactly $\tau$
  unlabeled points are classified in the positive class.}
Note that, having assigned a very high value to $C_1$ or $C_2$,
the objective function value is dominated by these slack variables.

The function $h_{\omega,b}(\cdot)$ in Constraint~\eqref{unlabeled
  part} is not continuous, which means that Problem~\eqref{EQproblem1}
cannot be easily solved by standard solvers.
A typical way to overcome this problem is to add binary variables to
turn on or off the enforcement of a constraint.
By introducing binary variables $z_i \in \set{0,1}$, $i \in [n+1, N]$,
we can reformulate the optimization Problem~\eqref{EQproblem1} using
the following big-$M$ formulation:
\begin{varsubequations}{P3}
  \label{equation2}
  \begin{align}
    \min_{\omega,b,\xi,\eta, z} \quad
    & \frac{\Vert\omega \Vert^2}{2} + C_1 \sum_{i=1}^n
      \xi_i +C_2 (\eta_1 + \eta_2) \label{2}
    \\
    \st \quad
    & y_i (\omega^\top x^i +b) \geq 1 - \xi_i, \quad i \in
      [1,n],
      \label{3} \\
    &  \omega^\top x^i +b \leq z_i M, \quad i \in [n+1,N],
      \label{4}\\
    &  \omega^\top x^i +b \geq -(1-z_i)M, \quad i \in [n+1,N],
      \label{5}\\
    & \tau - \eta_1 \leq  \sum_{i=n+1}^N z_i \leq \tau + \eta_2,
      \label{6}
    \\
    & \xi_i  \geq 0, \quad i \in [1, n],
      \label{8} \\
    & \eta_1, \eta_2 \geq 0,
      \label{sumeta}\\
    & z_i \in \{0,1\}, \quad i \in [n+1, N],
      \label{10}
  \end{align}
\end{varsubequations}
where $M$ needs to be chosen sufficiently large.
As $z_i$ is binary, Constraints~\eqref{4} and~\eqref{5} lead to
\begin{align*}
  \omega^\top x^i +b > 0 \implies z_i = 1, \quad i \in [n+1,N],\\
  \omega^\top x^i +b < 0 \implies z_i = 0,\quad i \in [n+1,N].
\end{align*}
If $x^i$ lies on the hyperplane, i.e., $\omega^\top x^i +b =0 $,
Constraints~\eqref{4} and~\eqref{5} hold for $z_i = 1$ and $z_i = 0$.
In this case, it can be counted either on
the positive or on the negative side.
\rev{For this reason, Problem~\eqref{equation2} is not formally
  equivalent to Problem~\eqref{EQproblem1}.}
Reformulation \eqref{equation2} is a mixed-integer quadratic problem
(MIQP) in which all constraints are linear but the objective function
is quadratic.
We  refer to this problem as CS$^3$VM.

\rev{Since we now stated our first model, let us shed some light on
  the results depending on whether the standard SVM or CS$^3$VM is
  used.
  Figure~\ref{fig:draws} shows a $2$-dimensional example data set and
  the corresponding hyperplanes for SVM and CS$^3$VM.
  In this case, $\tau = 11$, i.e., $11$ unlabeled points belong the
  the positive class. Note that SVM only classifies $6$ unlabeled
  points as positive, while CS$^3$VM classifies $11$ as such.
  The point that lies on the CS$^3$VM hyperplane is classified as
  positive because the binary variable regarding this point is $1$.
  This example shows that using $\tau$ as additional information can
  improve the classification of unlabeled points.}
\begin{figure}
  \centering
  \includegraphics[width=0.485\textwidth]{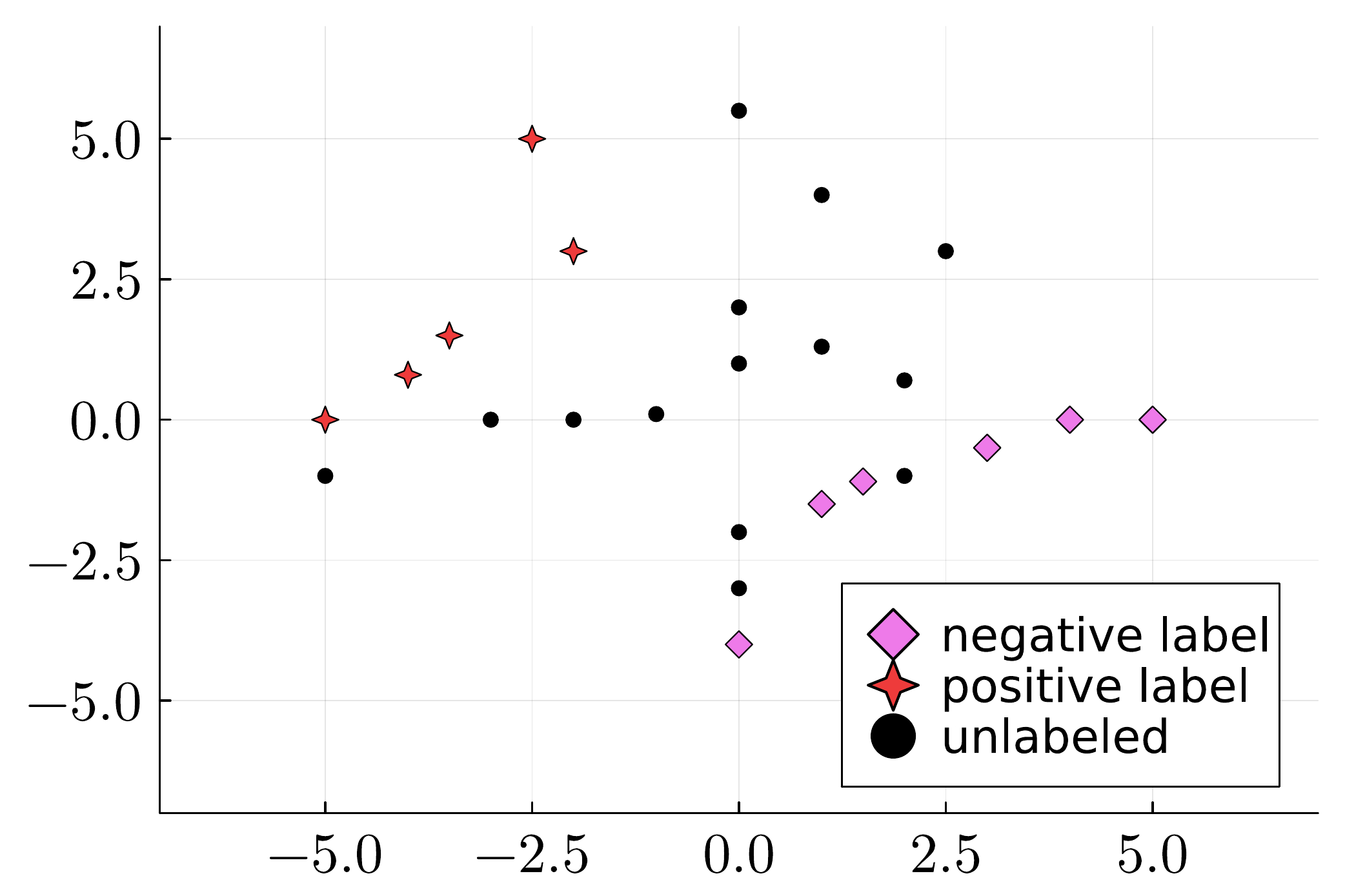}\quad%
  \includegraphics[width=0.485\textwidth]{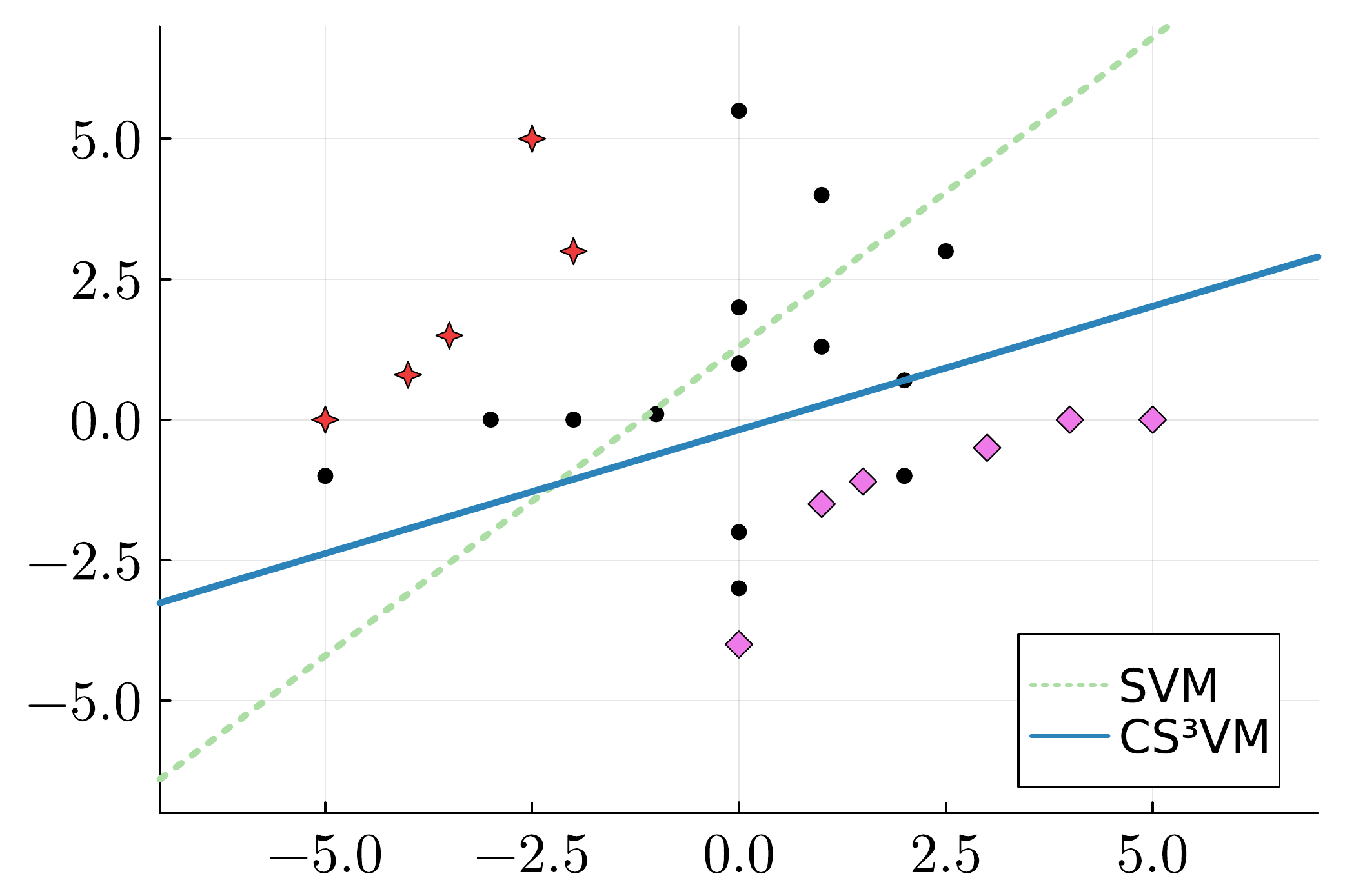}
  \caption{\rev{A 2-dimensional example (left) and the hyperplanes
      resulting from the SVM and the CS$^3$VM (right).}}
  \label{fig:draws}
\end{figure}

In the big-$M$ formulation, the choice of $M$ is crucial.
If $M$ is too small, the problem can become infeasible or optimal
solutions could be cut off.
If $M$ is chosen too large, the respective continuous relaxations
usually lead to bad lower bounds and solvers may encounter numerical
troubles.
The choice of~$M$ is discussed in the following lemma and theorem.
In Lemma~\ref{theob} we show how~$M$ is related to the objective
function and the given data.
This is then used in Theorem~\ref{theob2} to derive a provably correct
big-$M$.
\begin{lem}
  \label{theob}
  Given a feasible point for Problem~\eqref{equation2} with an
  objective function value~$f$, an optimal solution $(\omega^*, b^*,
  \xi^*, \eta^*, z^*)$ of \eqref{equation2} satisfies
  \begin{equation*}
    \label{boundomega}
    \Vert \omega^* \Vert \leq \sqrt{2f}
    \quad\text{and}\quad
    \vert b^* \vert \leq  \Vert \omega^* \Vert  \max_{i \in [1,N]} \Vert
    x^i \Vert + 1
  \end{equation*}
  and, consequently, every optimal solution satisfies \eqref{4} and
  \eqref{5} for
  \begin{equation*}
    \label{HOWM}
    M = 2 \sqrt{2f} \max_{i \in [1,N]} \Vert x^i \Vert + 1.
  \end{equation*}
\end{lem}
\begin{proof}
  Due to optimality, we get
  \begin{equation*}
    \frac{\Vert \omega^* \Vert^2}{2} \leq   \frac{\Vert
      \omega^* \Vert^2}{2} + C_1  \sum_{i = 1}^n
    \xi^*_i + C_2(\eta^*_1 + \eta^*_2)\leq f
    \implies \Vert \omega^* \Vert \leq \sqrt{2f}.
  \end{equation*}
  The second inequality is shown by contradiction.
  To this end, we w.l.o.g.\ assume that $\tilde{b} =
  \Vert \omega^* \Vert  \max_{i \in [1,N]} \Vert x^i \Vert +1 +
  \delta$ is part of an optimal solution for some $\delta > 0$.
  Using the inequality of Cauchy--Schwarz then yields
  \begin{align*}
    (\omega^{*})^\top x^i + \tilde{b}
    & =  (\omega^{*})^\top x^i + \Vert \omega^* \Vert  \max_{j \in
      [1,N]} \Vert x^j \Vert  +1 + \delta
    \\
    & \geq  - \Vert \omega^* \Vert \Vert x^i \Vert + \Vert \omega^*
      \Vert \max_{j  \in [1,N]} \Vert x^j \Vert +1 + \delta
    \\
    & >  1
  \end{align*}
  for all $i \in [1,N]$.
  Hence, for all $i \in [1,n]$ with $y_i = 1$, we get $\tilde{\xi}_i =
  0$ from Constraint~\eqref{3} and the objective function.
  Moreover, for $i \in [1,n]$ with $y_i = -1$, the same reasoning
  implies
  \begin{equation*}
    - (\omega^{*})^\top x^i - \tilde{b}  = 1 - \tilde{\xi}_i \implies
    \tilde{\xi}_i = 2 + (\omega^{*})^\top x^i +  \Vert \omega^* \Vert
    \max_{j  \in [1,N]} \Vert x^j \Vert + \delta.
  \end{equation*}
  Besides that, for the unlabeled data $i\in [n+1,N] $, since $
  (\omega^{*})^\top x^i + \tilde{b} >1$, we get $\tilde{z}_i = 1,$
  which leads to
  \begin{equation*}
    \sum_{i=n+1}^N \tilde{z}_i = m \implies
    \tilde{\eta}_1 = 0, \ \tilde{\eta}_2 = m -\tau.
  \end{equation*}
  This means that the objective function value for the point
  $(\omega^*,\tilde{b}, \tilde{\xi}, \tilde{\eta}, \tilde{z})$ is given by
  \begin{equation*}
    \tilde{f} \define \frac{\Vert \omega^* \Vert^2}{2} + C_1
    \sum_{i : y_i =-1} \left(2 + (\omega^{*})^\top x^i +
      \Vert \omega^* \Vert \max_{j  \in [1,N]} \Vert x^j \Vert + \delta \right)
    + C_2( m -\tau).
  \end{equation*}
  However, if we set  $\bar{b} \define \Vert \omega^* \Vert \max_{i  \in
    [1,N]} \Vert x^i \Vert +1, $ we get
  \begin{equation*}
    (\omega^{*})^\top x^i + \bar{b} \geq 1, \quad i \in [1,N],
  \end{equation*}
  i.e., $z_i = 1$ for all $i \in [n+1,N] $, $\bar{\eta}_1 = 0,$
  $\bar{\eta}_2 = m -\tau$,  and $\bar{\xi}_i = 0 $ for $i$ with $y_i
  = 1$.
  Moreover, for $i \in [1,n]$ with $y_i = -1$, from Constraint
  \eqref{3} we obtain
  \begin{equation*}
    - (\omega^{*})^\top x^i - \tilde{b}  = 1-\bar{\xi_i} \implies
    \bar{\xi_i} = 2 + (\omega^{*})^\top x^i +  \Vert \omega^* \Vert
    \max_{i  \in [1,N]} \Vert x^i \Vert.
  \end{equation*}
  All this implies that the objective function value $\bar{f}$ for the
  point  $(\omega^*,\bar{b}, \bar{\xi}, \bar{\eta}, \bar{z})$
  satisfies
  \begin{equation*}
    \bar{f} \define  \frac{\Vert \omega^* \Vert^2}{2} + C_1
    \sum_{i : y_i =-1}(2 + (\omega^{*})^\top x^i +
    \Vert \omega^* \Vert \max_{j  \in [1,N]} \Vert x^j \Vert ) \ +
    C_2( m -\tau) < \tilde{f},
  \end{equation*}
  which contradicts the assumption that $\tilde{f}$ is optimal.
  Hence,
  \begin{equation*}
    \vert b^* \vert
    \leq \Vert \omega^* \Vert \max_{i  \in [1,N]} \Vert x^i \Vert +1
  \end{equation*}
  holds, which proves the second inequality.
  Note further that
  \begin{equation*}
    (\omega^{*})^\top x^i + b^* \leq \Vert\omega^* \Vert
    \Vert x^i \Vert + \vert b^* \vert \leq  2\sqrt{2f} \max_{j
      \in [1,N]} \Vert x^j \Vert + 1 = M
  \end{equation*}
  and
  \begin{equation*}
    (\omega^{*})^\top x^i + b^* \geq - \Vert\omega^* \Vert
    \Vert x^i \Vert - \vert b^* \vert \geq  - 2\sqrt{2f}
    \max_{j \in [1,N]} \Vert x^j \Vert - 1 = -M
  \end{equation*}
  holds for all $i \in [n+1,N]$.
\end{proof}

We now use the result from the last technical lemma to obtain a
provably correct big-$M$.

\begin{thm}
  \label{theob2}
  A valid big-$M$ for Problem~\eqref{equation2} is given by
  \begin{equation}
    \label{validM}
    M = 2\sqrt{2(2C_1\bar{n} + C_2(m-\tau))}\max_{i \in [1,N]}\Vert x^
    i \Vert +1
  \end{equation}
  with $\bar{n} \define \vert \defset{i \in [1,n]}{ y_i = -1}\vert$.
\end{thm}
\begin{proof}
  Consider the feasible point of \eqref{equation2} given by $\omega = 0 \in \R^d$
  and $b=1$.
  Since \mbox{$\omega^\top x^i + b = 1$} holds for all $i\in[1,N]$,
  Constraint~\eqref{3} implies
  \begin{equation*}
    \xi_i =
    \begin{cases}
      2, & \text{if } y_i = -1 ,\\
      0, & \text{otherwise}.
    \end{cases}
  \end{equation*}
 Moreover, using Constraints~\eqref{4}--\eqref{6} leads to
  \begin{equation*}
    z_i = 1, \ i \in [n+1,N], \quad \eta_1 = 0, \quad \eta_2 = m - \tau,
  \end{equation*}
  which implies that the objective function for the point
  $(\omega, b, \xi, \eta, z)$ is given by
  \begin{equation*}
    f = 0 + 2C_1\bar{n} + C_2(m-\tau).
  \end{equation*}
  Finally, from Lemma~\ref{theob}, we get
  \begin{equation*}
    M = 2\sqrt{2(2C_1\bar{n} + C_2(m-\tau))}\max_{i \in [1,N]}\Vert x^
    i \Vert + 1. \qedhere
  \end{equation*}
\end{proof}


\section{A Re-Clustering Method for solving CS$^3$VM}
\label{sec:cluster-unlabeled-points}

In Model~\eqref{equation2} of the last section, each binary variable
is related to an unlabeled point. The larger the number of unlabeled
data, the larger the number of binary variables and, hence, the
larger the computational burden to solve Problem~\eqref{equation2}.
To reduce this computational burden, we propose to cluster the
unlabeled data. This way, only one binary variable per cluster is
needed. For every cluster, we use its centroid as its representative
point.
To obtain clusterings, we use minimum sum-of-squares
clustering (MSSC). The MSSC problem is NP-hard; see, e.g.,
\textcite{KNP,KNP2,KNP3}. However, we do not need a globally
optimal solution for the MSSC problem as will be shown below.
Given a number~$k$ of clusters and a matrix $S =[s^1, \dotsc, s^p] \in
\R^{d\times p}$ of given points, the goal of the MSSC is to find
mean vectors $c^j \in \R^d$, $j \in [1, k]$, that solve the problem
\begin{equation*}
  c^* = \argmin_c \ \ell(S,c), \quad c = (c^j)_{j=1,\dots,k},
\end{equation*}
where the loss function $\ell$ is the sum of the squared Euclidean
distances, i.e.,
\begin{equation*}
  \ell(S,c) = \sum_{j=1}^k \sum_{s^i \in \mathcal{C}_j} \Vert s^i - c^j \Vert^2
\end{equation*}
with $\mathcal{C}_j \subset \R^{d}$ being the set of data points that
are assigned to cluster $j$.

We solve this problem heuristically using the $k$-means
algorithm \parencite{Kmeans1a, Kmeans3} for $S = X_u$, i.e.,
we cluster the unlabeled data.
Then, instead of using all unlabeled data as in the last section, we
only use the clusters' centroids $c^1, \dotsc, c^k$ and the numbers
$e_1, \dots, e_k$ of data points in each cluster to obtain the problem
\begin{varsubequations}{P4}
  \label{equation3}
  \begin{align}
    \min_{\omega,b,\xi,\eta, z} \quad
    & \frac{\Vert\omega \Vert ^ 2  }{2} + C_1 \sum_{i=1}^n
      \xi_i +C_2 (\eta_1 + \eta_2)
      \label{svmfunction3} \\
    \st \quad
    & y_i (\omega^\top x^i+b)  \geq 1 - \xi_i,  \quad i \in [1,n],
      \label{constraintx} \\
    & \omega^\top c^j +b \leq z_j M,  \quad   j\in [1,k],
      \label{constraintc1}\\
    & \omega^\top c^j +b \geq -(1-z_j) M, \quad  j \in [1,k],
      \label{constraintc2}\\
    & \tau - \eta_1 \leq   \sum_{j=1}^k e_jz_j \leq \tau + \eta_2,
      \label{constraints1} \\
    & \xi_i   \geq 0, \quad i \in [1, n],
      \label{constraintxi} \\
    & \eta_1, \eta_2 \geq 0,
      \label{constrainteta} \\
    & z_j \in \{0,1\}, \quad j \in [1,k].
      \label{constraintz}
  \end{align}
\end{varsubequations}
A valid big-$M$ is still given by~\eqref{validM} as shown in the next
proposition.
\begin{prop}
  \label{lemma3}
  If $e_j \geq 1$ for all $j \in [1,k]$, a valid big-$M$ for
  Problem~\eqref{equation3} is given by \eqref{validM}.
\end{prop}
\begin{proof}
  The proof follows the same lines as the proofs of Lemma~\ref{theob}
  and Theorem~\ref{theob2} with the additional observation that for all
  $j \in [1,k]$, it holds
  \begin{equation*}
    \Vert c^j \Vert = \frac{ 1 }{e_j} \left\Vert  \sum_{i : x^i\in
        \mathcal{C}_j}x^i\right \Vert \leq \frac{e_j \max_{i \in
        [n+1,N]}\Vert x^i \Vert}{e_j}
    = \max_{i \in [n+1,N]}\Vert x^i \Vert.
    \qedhere
  \end{equation*}
\end{proof}

It can happen that the hyperplane given by $(\omega^*,b^*)$ that
results from the solution of Problem~\eqref{equation3} cuts through
some cluster.
This means that not all data points of the cluster actually lie on the
same side of the hyperplane.
If this happens, the solution of Problem~\eqref{equation3}
does not satisfy the cardinality constraint \eqref{6} of
Problem~\eqref{equation2}. To fix this, we propose an iterative
method that is formally  listed  in Algorithm~\ref{First version}.
\rev{Note that the use the $k$-means algorithm is helpful here as it
  automatically provides the convex hulls of the clusters.
  Hence, it is easy to check if the hyperplane cuts through some
  cluster or not.}

\begin{algorithm}
  \caption{Re-Clustering Method (RCM)}
  \label{First version}
  \SetKwInput{Input}{Input~}
  \Input{$X\in \mathbb{R}^{d\times N}, y \in \set{-1,1}^n, k^1 \in
    \mathbb{N}, C_1 >0, C_2 > 0$, and $\tau \in \mathbb{N}$.}
  Set $t \ot 1$, compute $M^t$ as in~\eqref{validM}, compute a
  clustering of~$X_u$ in $k^1$ many clusters using the $k$-means
  algorithm, and obtain the centroids $c^1, \dotsc, c^{k^1}$ as well as the
  numbers $e_1, \dotsc, e_{k^1}$ of data points in each cluster.\label{first-version:cluster}
  \\
  Solve Problem~\eqref{equation3} to compute the hyperplane
  $(\omega^t,b^t)$ as well as $\xi^t,\eta^t,$ $z^t$.
  \label{first-version:solve-problem}
  \\
  \uIf{the hyperplane $(\omega^t,b^t)$ cuts a cluster}{
    Set $k^{t+1} \ot k^t$. \\
    \For{each cluster that is cut by the hyperplane $(\omega^t,b^t)$}{
      Split the cluster into two new clusters so that neither of the
      two new clusters is cut by the hyperplane~$(\omega^t,b^t)$.
      \label{First version:split-step}
      \\
      Update the centroids of the newly created clusters.
      \\
      Set $k^{t+1} \ot k^{t+1} + 1$.
    }
    Update $t\ot t +1 $ and go to Step~\ref{first-version:solve-problem}.
  }
  \Else{
    Return the hyperplane~$(\omega^t,b^t)$ as well as $\xi^t,\eta^t,$ $z^t$.
  }
\end{algorithm}

If Algorithm~\ref{First version} terminates it holds that all points
in a cluster are on the same side of the final hyperplane.
This implies the cardinality constraint~\eqref{6} is satisfied.
Note that the $k$-means algorithm is only called once to initialize
the clustering.
For all other iterations, we manually split clusters if they are cut
by the hyperplane of the respective iteration and compute the new
centroids directly.

The next theorem establishes that Algorithm~\ref{First version} always
terminates after finitely many iterations.

\begin{thm}
  \label{finite}
  Suppose that $e_j \geq 1$ for all $j\in[1,k^1]$ after
  Step~\ref{first-version:cluster} of Algorithm~\ref{First
    version}. Then, Algorithm~\ref{First version} terminates after at
  most $m-k^1$ iterations, where $m$ is the number of the unlabeled data
  points and $k^1$ is the number of initial clusters.
\end{thm}
\begin{proof}
  Observe that since we cluster $m$ unlabeled points, the maximum
  number of clusters we can obtain is $m$. Besides that, if in an
  iteration~$t$, Algorithm~\ref{First version} does not terminate, at
  least one cluster is split Step~\ref{First version:split-step}.
  Because we start with $k^1$ clusters and since in each iteration, we
  increase the number of clusters at least by one, the maximum
  number of iterations is $m-k^1$.
\end{proof}

Note that the point obtained by Algorithm~\ref{First version} is not
necessarily a minimizer of Problem~\eqref{equation2}.
However, the objective function value of the point obtained by
Algorithm~\ref{First version} is an upper bound for the objective
function value of Problem~\eqref{equation2}.
\begin{thm}
  \label{theoupper}
  Let $(\bar{\omega}, \bar{b}, \bar{\xi}, \bar{\eta}, \bar{z})$ be the
  point returned by Algorithm~\ref{First version}.
  Then, \rev{$(\bar{\omega}, \bar{b}, \bar{\xi}, \bar{\eta}, \bar{z})$
    is feasible for} Problem~\eqref{equation2} with
  \begin{equation*}
    M = 2\sqrt{2\bar{f}}\max_{i~\in[1,N]} \Vert x^i \Vert~+~1
  \end{equation*}
  \rev{and, consequently,
    \begin{equation*}
      \bar{f} \define \frac{\Vert\bar{\omega} \Vert ^ 2
      }{2} + C_1\sum_{i=1}^n\bar{ \xi_i} +C_2 (\bar{\eta}_1 +
      \bar{\eta}_2)
    \end{equation*}
    is an upper bound of Problem~\eqref{equation2}.
  }
\end{thm}
\begin{proof}
  For all clusters $\mathcal{C}_j$, $j\in \set{1, \dotsc, k^t}$, where $t$
  is the final iteration of Algorithm~\ref{First version}, we set
  $\tilde{z}_i = \bar{z}_j$ for all $i$ with $x^i \in
  \mathcal{C}_j$. We now show that $(\bar{\omega}, \bar{b}, \bar{\xi},
  \bar{\eta}, \tilde{z})$ is a feasible point for
  Problem~\eqref{equation2}. Indeed, Constraints~\eqref{3}, \eqref{8},
  \eqref{sumeta}, and~\eqref{10} are clearly fulfilled.
  Furthermore, since
  \begin{equation*}
    \sum_{i \in \mathcal{C}_j} \tilde{z}_i = e_j\bar{z}_j
  \end{equation*}
  for all $j \in [1, k^t]$, using \eqref{constraints1} we get
  \begin{equation*}
    \sum_{i=n+1}^N \tilde{z}_i = \sum_{j=1}^{k^{t}} e_j
    \bar{z}_j \implies  \tau - \bar{\eta}_1 \leq
    \sum_{i=n+1}^N \tilde{z}_i \leq \tau + \bar{\eta}_2
  \end{equation*}
  and Constraint~\eqref{6} is satisfied.
  Besides that,
  \begin{equation}
    \label{normw}
    \frac{\Vert \bar{\omega} \Vert^2}{2} \leq \bar{f}
    \implies
    \Vert \bar{\omega} \Vert \leq \sqrt{2\bar{f}}
  \end{equation}
  holds and as in Lemma~\ref{theob}, we get
  \begin{equation}
    \label{normb}
    \vert \bar{b} \vert
    \leq
    \Vert \bar{\omega} \Vert  \max_{i \in [1,N]} \Vert x^i \Vert +1.
  \end{equation}

  Moreover, by construction, for all $i \in \{n+1, \dots N\}$ with
  $\tilde{z}_i = 1$, $x^i$ belongs to a cluster $\mathcal{C}_j$ such
  that $\bar{\omega}^\top c^j +\bar{b}\geq 0 $. Using the fact that
  all points in $\mathcal{C}_j$ are on the same side of the  hyperplane,
  this side must be the positive one. This fact together with
  \eqref{normw} and \eqref{normb} implies
  \begin{align*}
    -(1- \tilde{z}_i) M = 0\leq \bar{\omega}^\top x^i + \bar{b}
    & \leq \Vert \bar{\omega}\Vert \max_{i \in [1,N]}  \Vert x^i \Vert +
      \vert \bar{b} \vert
    \\
    & \leq 2\sqrt{2\bar{f}}\max_{i \in [1,N]}  \Vert x^i \Vert + 1 = M =
      \tilde{z}_i M.
  \end{align*}
  Similarly, for all $i \in \{n+1, \dots N\}$ with   $\tilde{z}_i = 0$, we get
  \begin{equation*}
    -M = -(1-  \tilde{z}_i) M
    \leq
    \bar{\omega}^\top x^i + \bar{b}
    \leq 0
    = \tilde{z}_i M
  \end{equation*}
  and \eqref{4} as well as \eqref{5} are fulfilled.
  Because $(\bar{\omega}, \bar{b}, \bar{\xi}, \bar{\eta}, \tilde{z})$ is
  a feasible point for Problem~\eqref{equation2}, $\bar{f}$ is an
  upper bound to the Problem~\eqref{equation2}.
\end{proof}
\rev{Note, finally, that since the point obtained from
  Algorithm~\ref{First version} is feasible for
  Problem~\eqref{equation2}, we can use it for warm starting.}


\section{Further Algorithmic Enhancements}
\label{Some Improvements}

In order to reduce computational costs, we propose two additional
enhancements.
The first one (see Section~\ref{section41}) makes use of the fact that
the SVM is mostly influenced by data points that are close to the separating
hyperplane.
The second one (see Section~\ref{section42}) introduces a rule for
updating $M$ in each iteration of Algorithm~\ref{First version}.

\subsection{Handling Points far From the Hyperplane}
\label{section41}

In Algorithm~\ref{First version}, the number of clusters increases in
each iteration. Hence, the time to solve Problem~\eqref{equation3}
increases from iteration to iteration in general.
Like in the original SVM, the  points closest to the hyperplane
influence the resulting hyperplane more than the other points.
Obviously, eliminating points that do not strongly influence the
hyperplane decreases the size of the problem. Some approaches
to eliminate these points have also been proposed for the original
SVM. For a survey, see, e.g., \textcite{reducing}. However,
most of these approaches are heuristics and do not necessarily yield a
feasible point of the problem.

The idea for our setting is the following.
Clusters that are far away from the hyperplane could be omitted as
this  will not change the solution.
The farther a cluster is from the hyperplane in an iteration, the
less likely it is that the cluster will be split or change sides
completely in a future iteration.
Hence, the clusters farthest from the current hyperplane mainly
add information about their side and capacity.
However, in a later iteration, the cluster may become relevant again.
Thus, we need to find a way to discard detailed information on certain
clusters but also a way to reactivate the discarded clusters if necessary.

We propose the following procedure to reduce the amount of clusters
that have to be considered in the current iteration of the algorithm.
If the number of clusters exceeds a
fixed value~$k^+$, we first fix the cluster with the centroid farthest
from the hyperplane as a kind of residual cluster on a side if this
side has points far from the hyperplane.
Second, we discard all clusters in which all points are
farther from the hyperplane than some~$\Delta^t$ and assign them to
the residual cluster on their side of the hyperplane. This way the
cardinality constraint remains valid. Moreover, all formerly discarded
clusters are checked for re-consideration. If a discarded cluster has a
point with a distance to the hyperplane less than $\Delta^t$ or if any
point in the cluster changed the side, the cluster is reactivated.

Let $\bar{S} = (s_{\alpha(1)}, \dotsc, s_{\alpha(d)})^\top$ be the vector of
increasingly sorted values of $S = \{s_1, \dots, s_d\}$ and let
$a \in (0,1)$.
The $a$-quantile of $S$, as proposed by \textcite{Quantile},
is given by
\begin{equation*}
  P_S(a) \define s_{\alpha(q)} +  \frac{s_{\alpha(q)} -
    s_{\alpha(r)}}{q - r} \left((d-1)a - q +1 \right)
\end{equation*}
with
\begin{equation*}
  q \define \max_{i \in [1,d]} \left\{i:\frac{i-1}{d-1} \leq a
  \right\},
  \quad r \define \min_{i \in [1,d]} \left\{i:\frac{i-1}{d-1} \geq a
  \right\}.
\end{equation*}
Given a parameter $\hat{\Delta}^t \in (0,1)$, we choose
$\Delta^t$ in each iteration $t$ according to
\begin{equation}
  \label{delta}
  \Delta^t = P_{D^{t}}(\hat{\Delta}^t)
  \quad \text{with} \quad
  D^{t}_j = \left\vert (\omega^t)^\top c_{j} + b^t \right \vert
  \quad \text{for all } j \in \left[1, k^t\right].
\end{equation}
Note that if in an iteration~$t$, a point in some discarded
cluster changed the side, the vector~$z$ as part of the current
solution does not fit to this change.
This happens when, e.g., $(\omega^{t-1})^\top x^i +b^{t-1} > 0$ and
$(\omega^{t})^\top x^i +b^{t} < 0$ but $z^t_j>0$ with $\mathcal{C}_j$
being the cluster with centroid farthest from the hyperplane on the
positive side.
To avoid that this happens too often, $\hat{\Delta}^{t+1}$ is
increased by a fixed value~$\tilde{\Delta} \in (0,1)$
when there is some point in some discarded cluster that has changed
sides.

Motivated by the above discussions, we add new steps in the
Algorithm~\ref{First version} that can be seen in
Algorithm~\ref{Second version}.
In Step \ref{Secondversion:removed}, if the number of clusters
exceeds $k^+$, clusters far from the hyperplane are discarded.
In Steps~\ref{Secondversion:change1} and~\ref{Secondversion:change2},
clusters discarded with a point that changed sides or that
is closer to the hyperplane than $\Delta^t$ are reactivated.
In Step~\ref{Second version:update-delta}, $\hat{\Delta}^t$ is
updated.

\begin{algorithm}
 \caption{Improved Re-Clustering Method (IRCM)}
  \label{Second version}
  \SetAlgoLined
  \SetKwInput{Input}{Input}\SetKwInOut{Output}{Output}
  \Input{$X\in~\mathbb{R}^{d\times N}$, $y \in \set{-1,1}^n$, $k^1 \in
    \mathbb{N}$, $C_1 >0$, $C_2>0$, $\tau\in\mathbb{N}$,
    $\hat{\Delta}^1 \in (0,1)$, $\tilde{\Delta}\in (0,1)$,
    $\mathcal{G}^1 =\emptyset$, $k^+ \in \mathbb{N}$.}
  Set $t=1$, compute $M^t$ as in \eqref{validM}, cluster $X_u$ in
  $k^1$ clusters using $k$-means, leading to centroids  $c^1, \dots,
  c^{k^1}$ and the numbers $e_1,\dots, e_{k^1}$ of  data points in each
  cluster.
  \\
  Solve  Problem~\eqref{equation3} to compute the hyperplane
  $(\omega^t,b^t)$ as well as $\xi^t, \eta^t, z^t$.\label{step2-2}
  \\  Compute $\Delta^t$ as in \eqref{delta}. \\
  \uIf{$k^t> k^+$}{
    update $\mathcal{G}^{t+1} \ot\mathcal{G}^{t} \cup
    \{\mathcal{C}_j:   \vert (\omega^t){^\top }x^\ell + b^t \vert >
    \Delta^t \; \forall  x^\ell \in \mathcal{C}_j \}$.
    \label{Secondversion:removed}
  }
  \Else{
    set $\mathcal{G}^{t+1} \ot \mathcal{G}^{t}$.
  }
  Set $\mathcal{J}^t\define\{\mathcal{C}_j \in \mathcal{G}^{t}:
  \exists x^\ell \in \mathcal{C}_j :  \sign(  (\omega^t)^\top
  x^\ell + b^t ) \neq \sign( (\omega^{t+1})^\top x^\ell + b^{t+1} ) \}.$ \label{Secondversion:change1}
  \\
  Update $\mathcal{G}^{t+1} \ot \mathcal{G}^{t+1}  \backslash
  (\{\mathcal{C}_j \in \mathcal{G}^{t} : \exists  x^\ell \in
  \mathcal{G}^{t}_j \text{ with }  \vert (\omega^t)^\top x^\ell + b^t
  \vert \leq \Delta^t \}\cup \mathcal{J}^t )$. \label{Secondversion:change2}
  \\
  \uIf{$\mathcal{J}^t \neq \emptyset$}{
    update $\hat{\Delta}^{t+1} \ot \min\{\hat{\Delta}^{t}+
    \tilde{\Delta},1\}$.
    \label{Second version:update-delta}
  }
  \Else{
    set $\hat{\Delta}^{t+1} \ot \hat{\Delta}^{t}$
  }
  Compute $ M^{t+1}$ as in \eqref{BigM 2}.\label{secondversion:updateM}
  \\
  \uIf{$\mathcal{J}^t \neq \emptyset$ or the hyperplane $(\omega^t,b^t)$ cuts a
    cluster}{
  Set $k^{t+1} \ot k^t$. \\
    \For{each cluster that is cut by the hyperplane $(\omega^t,b^t)$}{
      Split the cluster into two new clusters so that neither of the
      two new clusters is cut by the hyperplane~$(\omega^t,b^t)$.      \\
      Update the centroids of the newly created clusters.
      \\
      Set $k^{t+1} \ot k^{t+1} + 1$.}
    Update $t\ot t +1 $ and back to Step \ref{step2-2}.
  }
  \Else{
   Return the hyperplane~$(\omega^t,b^t)$ as well as $\xi^t,\eta^t, z^t$.
  }
\end{algorithm}

\subsection{Updating the Big-$M$}\label{section42}
As discussed in Section \ref{sec:miqp-formulation}, $M$ needs to be
sufficiently large.
However, the bigger the $M$, the more likely we face numerical
issues.
As shown in Section~\ref{sec:miqp-formulation}, the smaller the
objective function provided by a feasible point, the smaller
the value of $M$ can be chosen.
Based on that, we update $M$ in each iteration with the aim of
decreasing it.
We do this by adding Step~\ref{secondversion:updateM} in
Algorithm~\ref{Second version} and the next theorem justifies this.

\begin{thm}
  \label{isaupper}
  Consider $X, y, C_1, C_2, \tau$, as well as $c^1,\dots, c^{k^t}$ and
  $e_1, \dots e_{k^t}$ in an iteration~$t$ of Algorithm~\ref{First
    version}.
  Then, the optimal solution $(\bar{\omega}^t,\bar{b}^t, \bar{\xi}^t,
  \bar{\eta}^t, \bar{z}^t)$ of Problem~\eqref{equation3} provides an
  upper bound
  \begin{equation}
    \label{upperf}
    \tilde{f}_t \define \frac{\Vert \bar{\omega}^t \Vert^2}{2} +
    C_1 \sum_{i=1}^n \bar{\xi_i} + C_2(\tilde{\eta}_1
    + \tilde{\eta}_2 ),
  \end{equation}
  with
  \begin{equation}
    \label{ez}
    \tilde{z}_j =
    \begin{cases}
      1, & \text{ if }  \left(\bar{\omega}^{t}\right)^\top \tilde{c}_j +
      \bar{b}^t \geq 0,
      \\
      0, & \text{ otherwise},
    \end{cases}
  \end{equation}
  and
  \begin{equation}
    \label{eeta}
    \tilde{\eta}_1 = \max\left\{0, \tau -  \sum_{j=1}^s
      e_j\tilde{z}_j\right \},
    \quad
    \tilde{\eta}_2 = \max\left\{0,   \sum_{j=1}^s
      e_j\tilde{z}_j - \tau\right \},
  \end{equation}
  for Problem~\eqref{equation3} with $c^1,\dots, c^{k^{t+1}}$ and
  $e_1, \dots e_{k^{t+1}}$ as updated in iteration $t$ with
  \begin{equation}  \label{BigM 2}
    M =  2 \sqrt{2\tilde{f_t}} \max_{i \in [1,N]} \Vert x^i \Vert + 1.
  \end{equation}
\end{thm}
\begin{proof}
  Consider $\tilde{z}$ as given in \eqref{ez} and $\tilde{\eta}_1,
  \tilde{\eta}_2$ as given in \eqref{eeta}.
  We now show  that $(\bar{\omega}^t,\bar{b}^t, \bar{\xi}^t, \tilde{z},
  \tilde{\eta} )$ is a feasible point for Problem
  \eqref{equation3}. Indeed, Constraints~\eqref{constraintx} and
  \eqref{constraints1}--\eqref{constraintz}  are clearly
  satisfied.
  Moreover, $(\bar{\omega}^t,\bar{b}^t, \bar{\xi}^t, \bar{\eta}^t,
  \bar{z}^t)$ provides the objective function value given
  by \eqref{upperf} and
  \begin{equation*}
    \Vert \bar{\omega}^t \Vert \leq \sqrt{2\tilde{f}_t },
    \quad
    \vert \bar{b}^t \vert \leq  \Vert \bar{\omega}^t \Vert  \max_{i \in
      [1,N]} \Vert x^i \Vert +1,
  \end{equation*}
  see the proof of Lemma \ref{theob}.
  This together with $\Vert c^j \Vert \leq \max_{i \in [n+1,N]}\Vert
  x^i \Vert$ implies
  \begin{equation*}
    \left(\bar{\omega}^t\right)^\top  c^j + \bar{b}^t \leq \Vert
    \bar{\omega}^t \Vert \max_{i \in [n+1,N]}\Vert x^i \Vert + \vert
    \bar{b}^t \vert \leq 2 \sqrt{2\tilde{f}_t} \max_{i \in [1,N]}
    \Vert x^i \Vert +1 = M
  \end{equation*}
  and
  \begin{equation*}
  \left(\bar{\omega}^t\right)^\top  c^j + \bar{b}^t \geq -M.
  \end{equation*}
  Hence, Constraints \eqref{constraintc1} and \eqref{constraintc2} are
  satisfied. Since $(\bar{\omega}^t, \bar{b}^t, \bar{\xi}^t, \tilde{z},
  \tilde{\eta})$ is a feasible point for Problem~\eqref{equation3},
  $\tilde{f}_t$ is an upper bound for Problem~\eqref{equation3}.
\end{proof}

Using Theorem~\ref{isaupper}, we can update $M$ in each
iteration of Algorithm~\ref{Second version} as in~\eqref{BigM 2}.
The following theorem establishes that as Algorithm~\ref{First
  version}, Algorithm~\ref{Second version} always terminates after
finitely many iterations.

\begin{thm}
  \label{algo2end}
  The Algorithm~\ref{Second version} terminates after at most $$2m-k^1 +
  \frac{(1-\hat{\Delta}^1)}{\tilde{\Delta}} $$ iterations, where $m$
  is the number of unlabeled data points, $k^1$ is the number of
  initial clusters, and $\hat{\Delta}^1, \tilde{\Delta}$ are inputs of
  Algorithm~\ref{Second version}.
\end{thm}
\begin{proof}
  In Algorithm~\ref{Second version}, the number of iterations can only
  be greater as in Algorithm~\ref{First version} if there is some iteration
  $t$ for which $\mathcal{J}^t \neq \emptyset$ holds but the hyperplane
  does not cut any cluster. At each iteration in which this happens,
  $\hat{\Delta}^t$ is increased and, in the worst case, i.e.,
  \begin{equation*}
    \hat{t} \define  m-k^1 + \frac{(1-\hat{\Delta}^1)}{\tilde{\Delta}},
  \end{equation*}
  we get $\hat{\Delta}^{\hat{t}} = 1$. This implies that for all
  further iterations $t$,
  \begin{equation*}
    \Delta^t = \max_{j \in [1,k^t]} \vert (\omega^t)^\top c^j + b^t\vert
  \end{equation*}
  holds.
  Thus, no cluster is added to the set $\mathcal{G}^t$.
  Since $\vert \mathcal{G}^{\hat{t}} \vert \leq m$ and $\mathcal{J}^t
  \subset \mathcal{G}^{\hat{t}}$, Algorithm~\ref{Second version} can
  only have $m$ more iterations with $\mathcal{J}^t \neq \emptyset$.
  This means that the maximum number of iterations is $2m - k^1 +
  (1-\hat{\Delta}^1)/\tilde{\Delta}$.
\end{proof}

\rev{Although Theorem~\ref{algo2end} shows that, in the worst case,
  Algorithm~\ref{Second version} can take more iterations than
  Algorithm~\ref{First version} to terminate, Algorithm~\ref{Second
    version} solves problems with less binary variable in every
  iteration, which means that the time per iteration will be lower
  compared to Algorithm~\ref{First version}.}

Note that the objective function value obtained by
Algorithm~\ref{Second version} is an upper bound for the objective
function value of Problem~\eqref{equation2}.

\begin{thm}
  \label{algo2up}
  Let $(\bar{\omega}, \bar{b}, \bar{\xi}, \bar{\eta}, \bar{z})$ be the
  point returned by Algorithm~\ref{Second version}.
  Then, \rev{$(\bar{\omega}, \bar{b}, \bar{\xi}, \bar{\eta}, \bar{z})$ is feasible}
  for Problem~\eqref{equation2} with
  \begin{equation*}
    M = 2\sqrt{2\bar{f}}\max_{i \in[1,N]} \Vert x^i \Vert + 1
  \end{equation*}
  \rev{and, consequently,
    \begin{equation*}
      \bar{f} \define \frac{\Vert\bar{\omega} \Vert ^ 2
      }{2} + C_1\sum_{i=1}^n\bar{ \xi_i} +C_2 (\bar{\eta}_1 +
      \bar{\eta}_2)
    \end{equation*}
    is an upper bound of Problem~\eqref{equation2}.
  }
\end{thm}
\begin{proof}
  Since Algorithm~\ref{Second version} terminates when no cluster
  changes the side and no cluster is cut by the hyperplane, the proof
  is the same as for Theorem~\ref{isaupper}.
\end{proof}

\rev{As before, we can use the point obtained from
  Algorithm~\ref{Second version} to warm start
  Problem~\eqref{equation2}.}


\section{Using IRCM for Warm-Starting}
\label{section c2svm3+}

As stated in Theorem \ref{algo2up}, the solution found by
Algorithm~\ref{Second version} is feasible for
Problem~\eqref{equation2}.
Hence, we can use it for warm-starting the solution process of
Problem~\eqref{equation2}.
The next lemma establishes that unlabeled points can be fixed to be in
one side of the hyperplane.
\begin{lem}
  \label{lemma1}
  Let $(\bar{\omega}, \bar{b}, \bar{\xi}, \bar{\eta}, \bar{z})$ be a
  feasible point of Problem~\eqref{equation2} with objective function
  value~$\bar{f}$.
  Furthermore, let $(\omega^*,b^*, \xi^ *, \eta^*, z^*)$ be an optimal
  solution of Problem~\eqref{equation2} with objective function
  value~$f^*$.
  Set
  \begin{align*}
    P_u & \define \Defset{i \in [n+1,N]}{(\omega^*)^{\top}  x^i + b^* >0 },
    \\
    N_u & \define \Defset{i \in [n+1,N]}{(\omega^*)^{\top}  x^i + b^* < 0},
  \end{align*}
  and let $S_p \subseteq P_u$, $S_n \subseteq N_u$ be arbitrarily chosen
  subsets and let $x^s \notin S_n$ be an unlabeled point with
  $\bar{\omega}^{\top} x^s + \bar{b} <0$.
  Then, the objective function value $\tilde{f}$ given by any feasible
  point of the problem
  \begin{varsubequations}{P5}
    \label{equation5}
    \begin{align}
      \min_{\omega,b,\xi,\eta, z} \quad
      & \frac{\Vert\omega \Vert ^ 2  }{2} + C_1
        \sum_{i=1}^n \xi_i +C_2 (\eta_1 + \eta_2)
        \label{svmfunction4} \\
      \text{s.t} \quad
      & y_i (\omega^{\top} x^i +b)  \geq 1 - \xi_i, \quad   i \in
        [1,n],
        \label{constraintx1} \\
      & \omega^{\top}  x^i +b \leq z_iM, \quad  i\in [n+1,N]
        \setminus (\{s\} \cup S_p \cup S_n),
        \label{constraintc11} \\
      & \omega^{\top}  x^i +b \geq -(1-z_i)M, \quad i\in
        [n+1,N] \setminus (\{s\} \cup S_p \cup S_n),
        \label{constraintc21} \\
      & \omega^{\top}  x^i +b \geq 0, \quad i\in S_p,
        \label{constraintc2a1}\\
      & \omega^{\top}  x^i +b \leq 0, \quad i\in S_n,
        \label{constraintc2b1}\\
      &  0 \leq   \omega^{\top}  x^s +b \leq z_sM,
        \label{constraintc2d1}\\
      &   \tau - \eta_1  \leq \vert S_p\vert +
        \sum_{i\in[n+1,N]\setminus( S_p \cup S_n)   } z_i
        \leq \tau + \eta_2,
        \label{constraints11} \\
      & \xi_i   \geq 0, \quad  i \in [1, n],
        \label{constraintxi1} \\
      & \eta_1, \eta_2 \geq 0,
        \label{constrainteta1} \\
      & z_i \in \{0,1\}, \quad i\in[n+1,N]\setminus( S_p \cup S_n),
        \label{constraintz1}
    \end{align}
  \end{varsubequations}
  with $M$ as defined in \eqref{BigM 2}, satisfies the following properties:
  \begin{enumerate}
  \item[(a)] $\tilde{f}$ is an upper bound for $f^*$,
  \item[(b)] if $\tilde{f}$ is the optimal objective function value of
    Problem~\eqref{equation5} and $\bar{f} < \tilde{f}$ is satisfied, it
    holds $(\omega^*)^{\top} x^s + b^* <0$, i.e., $x^s \in N_u.$
  \end{enumerate}
\end{lem}
\begin{proof}
  \begin{enumerate}
  \item[(a)] The points that satisfy Constraints
    \eqref{constraintx1}--\eqref{constraintz1} are feasible for
    Problem~\eqref{equation2} and provide an objective function value
    $\tilde{f}$.
    Since $f^*$ is the optimal objective function value of
    Problem~\eqref{equation2}, $f^* \leq \tilde{f}$ holds.
  \item[(b)] Consider by contradiction that $(\omega^*)^{\top} x^s +
    b^* \geq 0$ holds.
    This means that $(\omega^*,b^*, \xi^ *, \eta^*, z^*)$ satisfies
    \eqref{constraintx1}--\eqref{constraintz1}.
    Moreover, since $\tilde{f}$ is the objective function for
    Problem~\eqref{equation5}, we get $f^* = \tilde{f}$.
    However, $f^* \leq \bar{f}$ holds.
    Thus,
    \begin{equation*}
      f^* \leq \bar{f} < \tilde{f } = f^*
    \end{equation*}
    yields a contradiction.
    \qedhere
  \end{enumerate}
\end{proof}

Note that the last lemma can be adapted for the case
$\bar{\omega}^{\top} x^s + \bar{b}> 0$.
In this case, the constraints~\eqref{constraintc2d1} need to be
replaced with
\begin{equation}
  \label{subs1}
  -(1-z_s)M \leq  \omega^{\top}  x^s +b \leq 0
\end{equation}
and (b) needs to be replaced with $(\omega^*)^{\top} x^s +
b^* >0$, i.e., $x^s \in P_u$.
Note that the more points we have fixed on one side, the solution
of Problem~\eqref{equation2} tends to be faster as there are fewer binary
variables.

Moreover, the solution of Problem~\eqref{equation2}  can be found by
solving the problem
\begin{varsubequations}{P6}
  \label{equation7}
  \begin{align}
    \min_{\omega,b,\xi,\eta, z} \quad
    & \frac{\Vert\omega \Vert ^ 2  }{2} + C_1  \sum_{i=1}^n
      \xi_i +C_2 (\eta_1 + \eta_2)
      \label{svmfunction5} \\
    \st \quad
    & y_i (\omega^{\top} x^i -b)  \geq 1 - \xi_i, \quad  i \in
      [1,n],
      \label{constraintx1a} \\
    & \omega^{\top}  x^i +b \leq z_iM, \quad  i\in [n+1,N]
      \setminus ( S_p \cup S_n),
      \label{constraintc11a}\\
    & \omega^{\top}  x^i +b \geq -(1-z_i)M, \quad  i\in [n+1,N]
      \setminus ( S_p \cup S_n),
      \label{constraintc21a} \\
    & \omega^{\top}  x^i +b \geq 0, \quad  i\in S_p,
      \label{constraintc2a1a}\\
    & \omega^{\top}  x^i +b \leq 0,  \quad  i\in S_n,
      \label{constraintc2b1a}\\
    & \tau - \eta_1  \leq \vert S_p\vert + \sum_{i\in[n+1,N]\setminus(
      S_p \cup S_n) } z_i \leq \tau + \eta_2,
      \label{constraints11a} \\
    & \xi_i   \geq 0, \quad i \in [1, n],
      \label{constraintxi1a} \\
    & \eta_1, \eta_2 \geq 0
      \label{constrainteta1a} \\
    & z_i \in \{0,1\}, \quad  i\in [n+1,N]
      \setminus ( S_p \cup S_n),
      \label{constraintz1a}
  \end{align}
\end{varsubequations}
where $S_p$ and $S_n$ are subsets of $P_u$ and $N_u$,
respectively.

Based on these results, we propose the following.
We compute the point $(\bar{\omega},\bar{b}, \bar{\xi},\bar{\eta},
\bar{z})$ using Algorithm~\ref{Second version}, leading to an
objective function value~$\bar{f}$ for Problem~\eqref{equation2}.
Afterward, we sort the indices $i \in \set{n+1, N}$, indicated by the
permutation~$\alpha : \set{n+1, N} \to \set{n+1, N}$, so that $\vert
\bar{\omega}^{\top}x^{\alpha(i)} + \bar{b} \vert \geq \vert
\bar{\omega}^{\top}x^{\alpha(i)+1} + \bar{b} \vert$ holds.

Consider now a given and fixed parameter~$B_{\max}$, a factor~$\gamma
\in (1, m / B_{\max}]$, and let~$\beta$ be $\gamma B_{\max}$
rounded to the next integer.
While the number of fixed points is smaller than~$B_{\max}$, we do the
following.
For $i \in \set{1,\dotsc, \beta}$, if
$\bar{\omega}^{\top}x^{\alpha(i)} + \bar{b} < 0$ holds, we try to solve
Problem~\eqref{equation5} using the limit time of~$T_{\max}$ and the
upper bound~$\bar{f}$.
If there is a feasible point of this problem,
we set $(\bar{\omega},\bar{b}, \bar{\xi},\bar{\eta}, \bar{z})$ to this
point and update the objective function value $\bar{f}$ accordingly.
If no feasible point could be computed and if the limit time was not
reached, we fix $x^s$ to be in the negative side.

Similarly, we do the same if $\bar{\omega}^{\top}x^{d_\ell} +
\bar{b} > 0$ holds with \eqref{constraintc2d1} replaced
by~\eqref{subs1}.
The method is formally described in Algorithm~\ref{Scheme version}.
\rev{Finally note that, although Problem~\eqref{equation5} is an MIQP,
  it is a feasibility problem, which is often easier to solve than an
  optimization problem in practice.
  Besides that, if the point obtained from Algorithm~\ref{Second
    version} is close to the optimum of Problem~\eqref{equation2},
  many unlabeled points will be fixed and Problem~\eqref{equation2}
  will be faster to solve.}

\begin{algorithm}
  \caption{Improved \& Warm-Started Re-Clustering Method (WIRCM)}
  \label{Scheme version}
  \SetAlgoLined
  \SetKwInOut{Input}{Input}\SetKwInOut{Output}{Output}
  \SetAlgoLined
  \SetKwInput{Input}{Input}\SetKwInOut{Output}{Output}
  \Input{$X~\in~\mathbb{R}^{d\times N}$ , $y \in \set{-1,1}^n$, $k^1
    \in \mathbb{N}$, $C_1 > 0$, $C_2 > 0$, $\tau \in \mathbb{N}$,
    $\hat{\Delta}^1 \in (0,1)$, $\tilde{\Delta} \in (0,1)$,
    $\mathcal{G}^1 = \emptyset$, $k^+ \in \mathbb{N}$, $T_{\max} > 0$,
    $B_{\max} \in \mathbb{N}$, and $\gamma \in (1, m / B_{\max}]$.}
  Compute the hyperplane $(\bar{\omega}, \bar{b})$ and $\bar{\xi},
  \bar{\eta}, \bar{z}$ using Algorithm~\ref{Second
    version}, leading to the objective function value~$\bar{f}$.
  Let $M$ be the last $M^t$ of Algorithm~\ref{Second version}.
  \label{Scheme version:call-other-alg}\\
  Sort the indices $i \in \{n+1,N\}$ such that
  $\vert \bar{\omega}^{\top}x^{\alpha(i)} + \bar{b} \vert  \geq \vert
  \bar{\omega}^{\top}x^{\alpha(i) + 1} + \bar{b} \vert$ holds
  and set $\beta$ to be $\gamma B_{\max}$ rounded to the next integer.
  \\
  \For{$i \in \set{1, \dotsc, \beta}$}{
    \If{$\vert S_p \vert + \vert S_n \vert \leq  B_{\max}$ }{
      \uIf{$\bar{\omega}^{\top} x^s + \bar{b} <0$}{
        Solve Problem~\eqref{equation5} with upper bound~$\bar{f}$ and
        a time limit $T_{\max}$.
        \\
        \uIf{there is a feasible point}{
          update $(\bar{\omega}, \bar{b}, \bar{\xi},
          \bar{\eta}, \bar{z})$, and $\bar{f}$
        }
        \ElseIf{$T_{\max}$ was not reached}{
          $S_n \leftarrow S_n \cup \{s\}$}
      }
      \ElseIf{$\bar{\omega}^{\top} x^s + \bar{b} >0$}{
        Solve the problem~\eqref{equation5} with
        \eqref{constraintc2d1} replaced by~\eqref{subs1}, using
        $\bar{f}$ as an upper bound and a time limit of $T_{\max}$.
        \\
        \uIf{there is a feasible point}{
          update $(\bar{\omega}, \bar{b}, \bar{\xi},
          \bar{\eta}, \bar{z})$, and $\bar{f}$
        }
        \ElseIf{$T_{\max}$ was not reached}{
          $S_p \leftarrow S_p \cup \{s\}$}
      }
    }
  }
  Compute the solution $(\omega^*, b^*, \xi^*, \eta^*, z^*)$ of
  Problem~\eqref{equation7} with $(\bar{\omega}, \bar{b}, \bar{\xi},
  \bar{\eta}, \bar{z})$ value and $\bar{f}$ as an upper bound.
\end{algorithm}


\section{Numerical Results}
\label{sec:numerical-results}

In this section, we present and discuss our computational results that
illustrate the benefits of knowing the total amount of each class of
unlabeled data and of using our approaches to speed up the solution
process. We evaluate this on different test sets from the
literature. The test sets  are described in Section
\ref{subsection-test-sets}, while the computational setup is depicted
in Section~\ref{subsection-comp-setup}. The evaluation criteria are
described in Section~\ref{comparsions-SVM} and the numerical results
are discussed in Section~\ref{numericalresults}.

\subsection{Test Sets}
\label{subsection-test-sets}

For the computational analysis of the proposed approaches, we consider
the subset of instances presented by \textcite{Olson2017PMLB} that are
suitable for classification problems and that have at most three
classes. We restrict ourselves to instances of at most three classes
to obtain an overall test set of manageable size.
Repeated instances are removed and instances with missing information
are reduced to the observations without missing information.
If three classes are given in an instance, we transform them into two
classes such that the class with label~1 represents the
positive class, and the other two classes represent the negative
class.
\rev{This results in a final test set of 97~instances; see
  Table~\ref{table1} in Appendix~\ref{sec:deta-inform-inst}.}

To avoid numerical instabilities, we re-scale all data sets as
follows.
For each coordinate $j \in [1,d]$, we compute
\begin{equation*}
  l_j = \min_{i \in [1,N]}\{x^i_j\},
  \quad
  u_j = \max_{i \in [1,N]}\{x^i_j\},
  \quad
  m_j = 0.5 \left(l_j + u_j \right)
\end{equation*}
and shift each coordinate~$j$ of all data points~$x^i$ via $\bar{x}^i_j
= x^i_j - m_j$. If we do this for all data points, they get
centered around the origin. Moreover, if a coordinate~$j$ of the
re-scaled points is still large, i.e.,  if $\tilde{l}_j = l_j - m_j <-
10^{2}$ or $\tilde{u}_j = u_j - m_j > 10^{2}$ holds, it is re-scaled
via
\begin{equation*}
  \tilde{x}^i_j = (\overline{v} - \underline{v} ) \frac{\bar{x}^i_j -
    \tilde{l}_j}{\tilde{u}_j-\tilde{l}_j} +  \overline{v},
\end{equation*}
with $\overline{v} = 10^2$ and $ \underline{v} = -10^{2}$.
The corresponding 29 instances that we re-scaled are marked with an
asterisk in Table~\ref{table1}.
\rev{Note that we use a linear transformation to scale the
  datasets. Hence, after computing the hyperplane for the
  scaled data, the respective hyperplane for the original data can
  also be computed ex post by applying another suitably chosen linear
  transformation as well.}

In our computational study, we want to highlight the importance of
cardinality constraints, especially for the case of non-representative
biased samples. Biased samples occur frequently in non-probability
surveys, which are surveys for which the inclusion process is not
monitored and, hence, the inclusion probabilities are unknown as
well.
Correction methods like inverse inclusion probability weighting are
therefore not applicable. For an insight into inverse inclusion
probability weighting, see \textcite{Skinner2011} and references
therein.

To mimic this situation, we create 5~biased samples with
$\SI{10}{\percent}$ of the data being labeled for each instance.
Different from a simple random sample in which each point has an equal
probability of being chosen as labeled data, in the biased sample, the
labeled data is chosen with probability $\SI{85}{\percent}$ for being
on the positive side of the hyperplane. Then, for each instance, with
a time limit of \SI{3600}{\second}, we apply the approaches listed
in Section~\ref{subsection-comp-setup}.
In Appendix~\ref{sec:num-results-simple-sample}, we also provide
the results under simple random sampling, which produces unbiased
samples. We see that the results form the proposed methods are similar
to the plain SVM in that setting. Hence, besides the additional
computational burden, there is no downside to use the proposed method in
case of an unknown sampling process.

\subsection{Computational Setup}
\label{subsection-comp-setup}

Our algorithm has been implemented in \codename{Julia}~1.8.5 and we use
\codename{Gurobi}~9.5.2 and
\codename{JuMP} \parencite{DunningHuchetteLubin2017} to solve
Problem~\eqref{l2svm}, \eqref{equation2}, and \eqref{equation3}.  All
computations were executed on the high-performance cluster
``Elwetritsch'', which is part of the ``Alliance of High-Performance
Computing Rheinland-Pfalz'' (AHRP).
We used a single Intel XEON SP 6126 core with \SI{2.6}{\giga\hertz} and
\SI{64}{\giga\byte}~RAM.

For each one of the 485~instances described in
Section~\ref{subsection-test-sets}, the following approaches
are compared:
\begin{enumerate}
\item[(a)] SVM as given in Problem~\eqref{l2svm}, where only labeled
  data are considered;
\item[(b)] CS$^3$VM as given in Problem~\eqref{equation2} with $M$ as given
  in~\eqref{validM};
\item[(c)] IRCM as described in Algorithm~\ref{Second version};
\item[(d)] WIRCM as described in Algorithm~\ref{Scheme version}.
\end{enumerate}
Based on our preliminary experiments, we set the penalty parameters
$C_1 = C_2 = 1$.
For WIRCM, we impose a time limit for solving
Problem~\eqref{equation5} of $T_{\max} = \SI{40}{\second}$.
Moreover, we choose $\gamma = 1.2$ and the maximum number~$B_{\max}$
of unlabeled points that can be fixed as
\begin{equation*}
  B_{\max} =
  \begin{cases}
    0.2m,  & \text{if } m \in  [1,100],\\
    0.25m,  & \text{if } m \in  (100,500],\\
    0.35m, & \text{if } m \in  (500,1000],\\
    0.45m,  & \text{otherwise}.
  \end{cases}
\end{equation*}
Finally, for IRCM and WIRCM, we set $\hat{\Delta}\rev{^1} = 0.8$,
$\tilde{\Delta} = 0.1$, $k^+ = 50$, and the initial number of clusters
is set to
\begin{equation*}
  k^1 =
  \begin{cases}
    10, & \text{if } m \in  [1,500],\\
    20, & \text{if } m \in  (500,1000],\\
    50, & \text{otherwise}.
  \end{cases}
\end{equation*}
\rev{A more detailed discussion of the choice of hyperparameters is
  given in Appendix~\ref{sec:sensibilty-parameters}.}

\subsection{Evaluation Criteria}
\label{comparsions-SVM}

The first evaluation criterion is the run time of \rev{SVM, }CS$^3$VM,
IRCM, and WIRCM.
The results will help to contextualize other evaluation criteria such
as accuracy and precision.
To compare run times, we use empirical cumulative distribution
functions (ECDFs).
Specifically, for $S$ being a set of solvers (or approaches as above)
and for $P$ being a set of problems, we denote denote by $t_{p,s} \geq
0$ the run time of approach~$s \in S$ applied to problem~$p \in P$
in seconds. If $t_{p,s} > 3600$, we consider problem~$p$ as not being
solved by approach~$s$.
With these notations, the performance profile of approach~$s$ is the
graph of the function~$\gamma_s : [0, \infty) \to [0,1]$ given by
\rev{
\begin{equation}
  \label{perf}
  \gamma_s(\sigma) = \frac{1}{\vert P \vert}\big\vert\left\{p \in P:
    t_{p,s} \leq \sigma \right\}\big \vert.
\end{equation}}%
The second evaluation criterion is based on Theorem \ref{upperf},
where we show that the objective function value of the point obtained
by IRCM is an upper bound for CS$^3$VM, and consequently for
Problem~\eqref{equation3} that is solved with WIRCM.
\rev{Note that SVM also provides a feasible point for CS$^3$VM and,
  consequently, provides an upper bound as well.
  Consider $(\omega,b, \xi)$ the solution of SVM, we compute the
  binary variables $z_i$, $i\in [n+1,N]$ as follows:}
\begin{equation*}
\rev{  z_i =
  \begin{cases}
    1, & \text{if } \omega^\top x^i + b >0,\\
   0, & \text{if } \omega^\top x^i + b <0. \\
  \end{cases}}
\end{equation*}
\rev{If $\omega^\top x^i + b =0$ for some $x^i$, we set}
\begin{equation*}
  \rev{  z_i =
    \begin{cases}
      1, & \text{if } \sum_{j \in [n+1,N] :  \omega^\top x^j + b \neq 0 } z_i \leq \tau ,\\
      0, & \text{otherwise.}
    \end{cases}}
\end{equation*}
\rev{Finally, we set}
$$
\rev{
  \eta_1 = \max\left\{0, \tau -  \sum_{i=n+1}^N
    z_i\right \},
  \quad
  \eta_2 = \max\left\{0,    \sum_{i=n+1}^N
    z_i - \tau\right \},}
$$
\rev{and the objective function value can be computed as}
$$
\rev{ \frac{\Vert\omega \Vert ^ 2  }{2} + C_1  \sum_{i=1}^n
  \xi_i +C_2 (\eta_1 + \eta_2).
}
$$

Based on that, we compare how close the objective function values
obtained from \rev{SVM,} CS$^3$VM, IRCM, and WIRCM are to the optimal
solution.
To this end, we use ECDFs, \rev{for which we replace $t_{p,s}$ by
  $f_{p,s}$ in Equation~\eqref{perf}} with
\begin{equation}
  \label{gap}
  \rev{f}_{p,s} \define \frac{b_{p,s}-f^*_{p}}{f^*_{p}},
\end{equation}
where $f^*_{p}$ is the optimal objective function value of problem~$p$
and $b_{p,s}$ is the objective function value obtained by approach~$s$.

Besides that, for each instance and for each approach described in
Section~\ref{subsection-comp-setup}, after computing the hyperplane
$(\omega, b)$, we classify all points $x^i$ as being on the positive
side if $\omega^\top x^i + b > 0$ and as being on the negative side if
$\omega^\top x^i + b <0$ holds.
For CS$^3$VM and WIRCM, if the hyperplane $(\omega, b)$  satisfies
$\omega^\top x^i + b = 0$ for some unlabeled point $x^i$, we classify
this point as positive or negative depending on the respective binary
variable~$z_i$.
On the other hand, for IRCM, if $\omega^\top x^i + b = 0$  for some
unlabeled point $x^i$, we classify this point as positive or negative
depending on $z_j$ with $j$ so that $x^i \in \mathcal{C}_j$.
For the labeled points in these three approaches and for all points in
the SVM, if  $\omega^\top x^i + b = 0$ holds, we classify the point on the
correct side.
Note that for the cases in which the IRCMs take more than
\SI{3600}{\second} to solve the instance, we use the last hyperplane
found by the algorithm.
If we hit the time limit in \codename{Gurobi} when solving CS$^3$VM
(either standalone or in the final phase of the WIRCM),
we take the best solution found so far.

Knowing the true label of all points, we then distinguish all points in
four categories: true positive (TP) or true negative (TN) if the point is
classified correctly in the positive or negative class, respectively,
as well as false positive (FP) if the point is misclassified in the
positive class and as false negative (FN) if the point is misclassified
in the negative class.
Based on that we compute two classification metrics, for which a
higher value indicates a better classification.
The first one is accuracy ($\AC$).
It measures the proportion of correctly classified points and is given
by
\begin{equation}
  \label{AC}
  \AC \define \frac{\TP + \TN}{\TP + \TN + \FP + \FN} \in [0,1].
\end{equation}
The second metric is precision ($\PR$).
It measures the proportion of correctly classified points among all
positively classified points and is computed by
\begin{equation}
  \label{prec}
  \PR \define \frac{\TP}{\TP + \FP} \in [0,1].
\end{equation}

The main comparison in terms of accuracy and precision is w.r.t.\ the
``true hyperplane'', i.e., the solution of Problem~\eqref{l2svm} on the
complete data with all $N$ points and all labels available. The main
question is how close the accuracy and precision is to the one of the
true hyperplane. Hence, we compute the ratios of the accuracy and
precision according to
\begin{equation}\label{compartrue}
  \widehat{\AC} \define \frac{\AC}{\AC_{\true}},
  \quad
  \widehat{\PR} \define \frac{\PR}{\PR_{\true}},
\end{equation}
where $\AC_{\true}$ and $\PR_{\true}$ are computed as in Equations
\eqref{AC} and \eqref{prec} for the true hyperplane.

We also compare the measures with the SVM method, which only considers
the information of the labeled data. For this purpose, we compute
\begin{equation}
  \label{comparSVM}
  \overline{\AC} \define \frac{\AC-\AC_{\SVM}}{\AC_{\SVM}},
  \quad
  \overline{\PR} \define \frac{\PR-\PR_{\SVM}}{\PR_{\SVM}},
\end{equation}
where $\AC_{\SVM}$ and $\PR_{\SVM}$ are computed as in~\eqref{AC}
and~\eqref{prec} for the SVM hyperplane. To keep the numerical results
section concise, we report on recall and the false positive rate in
Appendix~\ref{sec:furth-numer-results}.

\subsection{Numerical Results}
\label{numericalresults}

\subsubsection{Run Time}
\label{Computation time}

Figure~\ref{perfomancetime} shows the ECDFs for the measured run
times.
\rev{Clearly, SVM is the fastest algorithm.
  This is expected as the SVM does not include any binary variables
  related to the unlabeled points, which is in contrast to other
  approaches.}
It can be seen that the IRCM outperforms both CS$^3$VM and WIRCM.
This shows that the idea to cluster unlabeled data points
significantly decreases the run time.
However, we need to be careful with the interpretation of these run
times since termination of \rev{SVM and}  IRCM does not imply that a globally optimal
point is found, whereas this is guaranteed CS$^3$VM and the WIRCM.
The quality of the points found by \rev{ SVM and} IRCM will be discussed in the
next section.
The figure also clearly indicates that Problem~\eqref{EQproblem1} is
rather challenging:
Even IRCM, which terminates for the most instances within the time
limit (indicated by the gray and dashed vertical line) only does so
for $\SI{57}{\percent}$ of the instances.
Note that the WIRCM has the worst efficiency.
This obviously needs to be the case since due to Step~\ref{Scheme
  version:call-other-alg} of Algorithm~\ref{Scheme version}, its run
time always includes the run time of the IRCM.
\rev{To shed some light on the scalability of the approaches,
  we also present a brief analysis of the run times in dependence of
  the number of samples in Appendix~\ref{sec:performancepersize}.}

\begin{figure}
  \centering
  \includegraphics[width=0.6\textwidth]{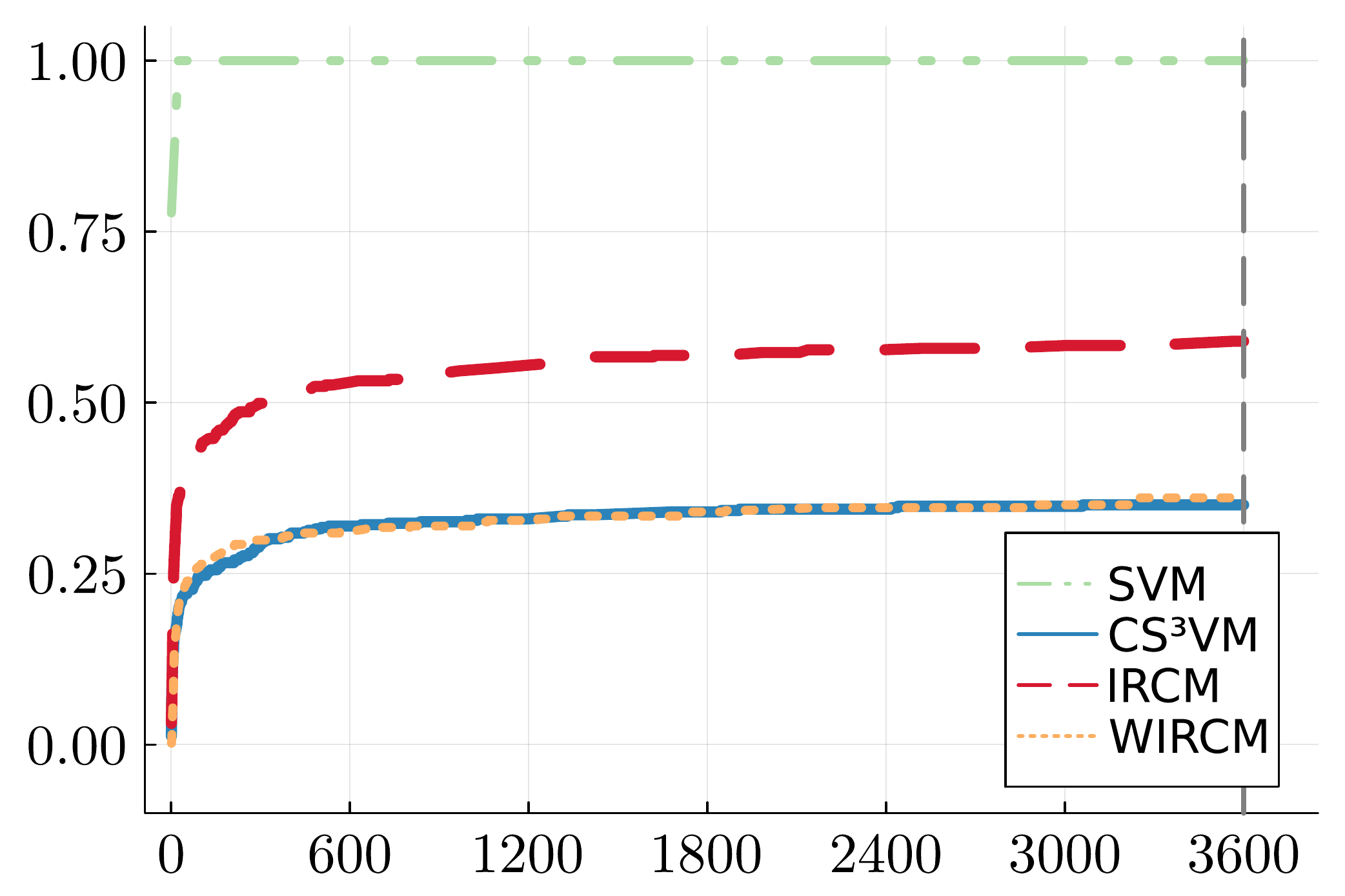}
  \caption{\rev{ECDFs for run time (in seconds).}}
  \label{perfomancetime}
\end{figure}

\subsubsection{Quality of the Obtained Upper Bounds}

\begin{figure}
  \centering
  \includegraphics[width=0.6\textwidth]{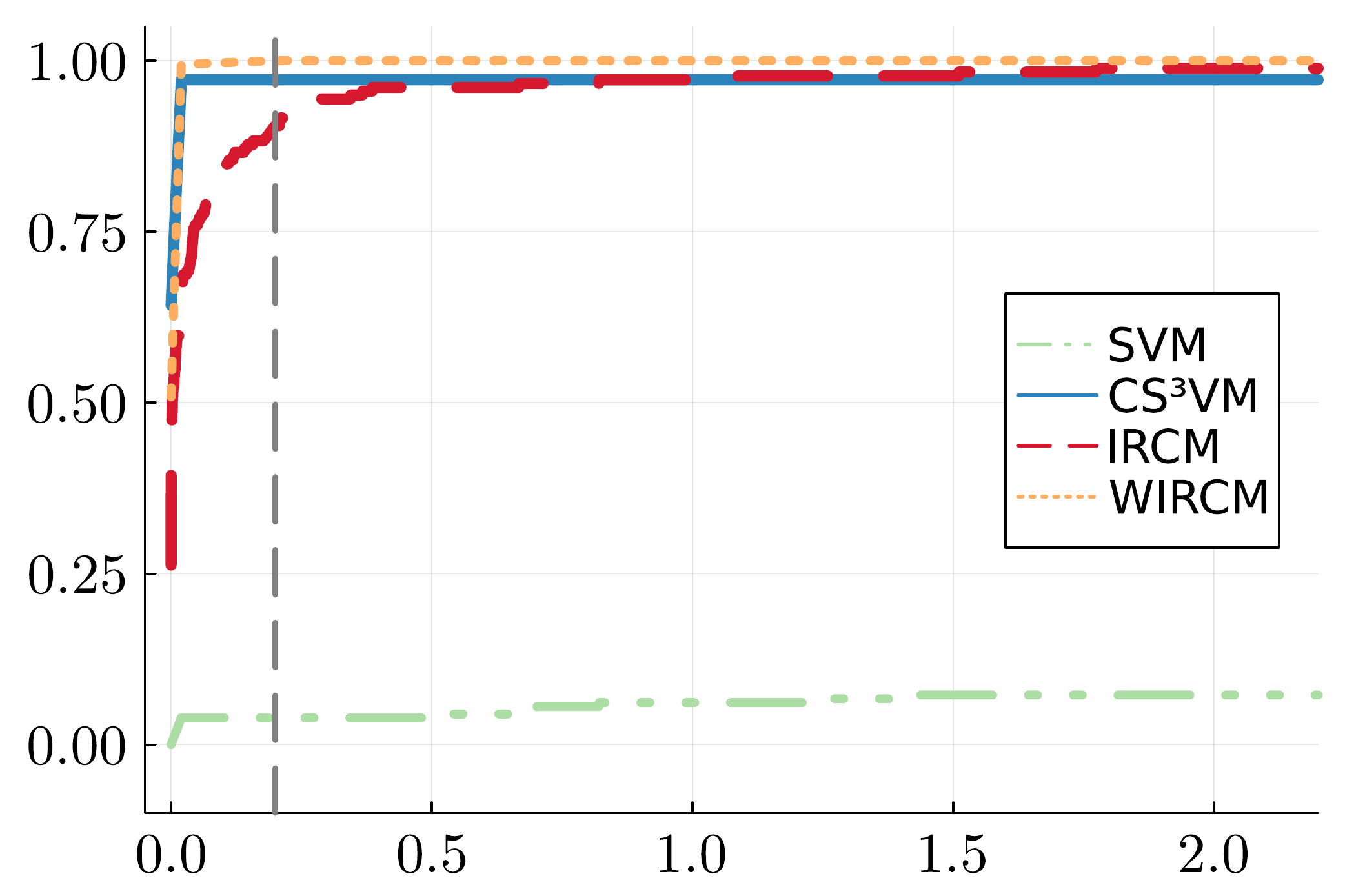}
  \caption{\rev{ECDFs for the quality of the obtained upper bounds.}}
  \label{perfomanceobjectivefunction}
\end{figure}

As discussed in the last section, for some instances none of the three
approaches \rev{that actually consider the unlabeled data} terminate
within the given time limit. This means we do not
obtain the optimal objective function value for these instances, which
we, moreover, can only provably obtain by CS$^3$VM and the WIRCM.
In fact, we have the optimal solution for 179~instances. These are
the baseline instances for Figure~\ref{perfomanceobjectivefunction},
which shows the ECDFs for the upper bound quality, as defined
in~\eqref{gap}.
\rev{Note that the objective function value obtained by SVM is very
  far from the optimal value, while} the IRCM finds an objective
function value rather close to the optimal value (with $f_{ps} \leq
0.2$, see the gray dashed vertical line) in $\SI{90}{\percent}$ of
these instances.
Besides that, the WIRCM outperforms CS$^3$VM in this comparison,
which means using the IRCM as a warm start improves the result.

The consequences of the results so far are the following.
If one is interested in getting a rather good feasible point as
quickly as possible, one should use the IRCM.
If one is able to spend some more run time, one should use the WIRCM.
Hence, both novel methods derived in this paper have their advantage
over just solving the CS$^3$VM with a standard MIQP solver.

\subsubsection{Accuracy}
\label{secaccuracy}

\rev{For some instances, none of the three approaches that
  actually tackle the unlabeled data terminate within the given time
  limit. Hence, our first comparison only considers instances for which
  CS$^3$VM terminates within the time limit.}

\begin{figure}
  \centering
  \includegraphics[width=0.495\textwidth]{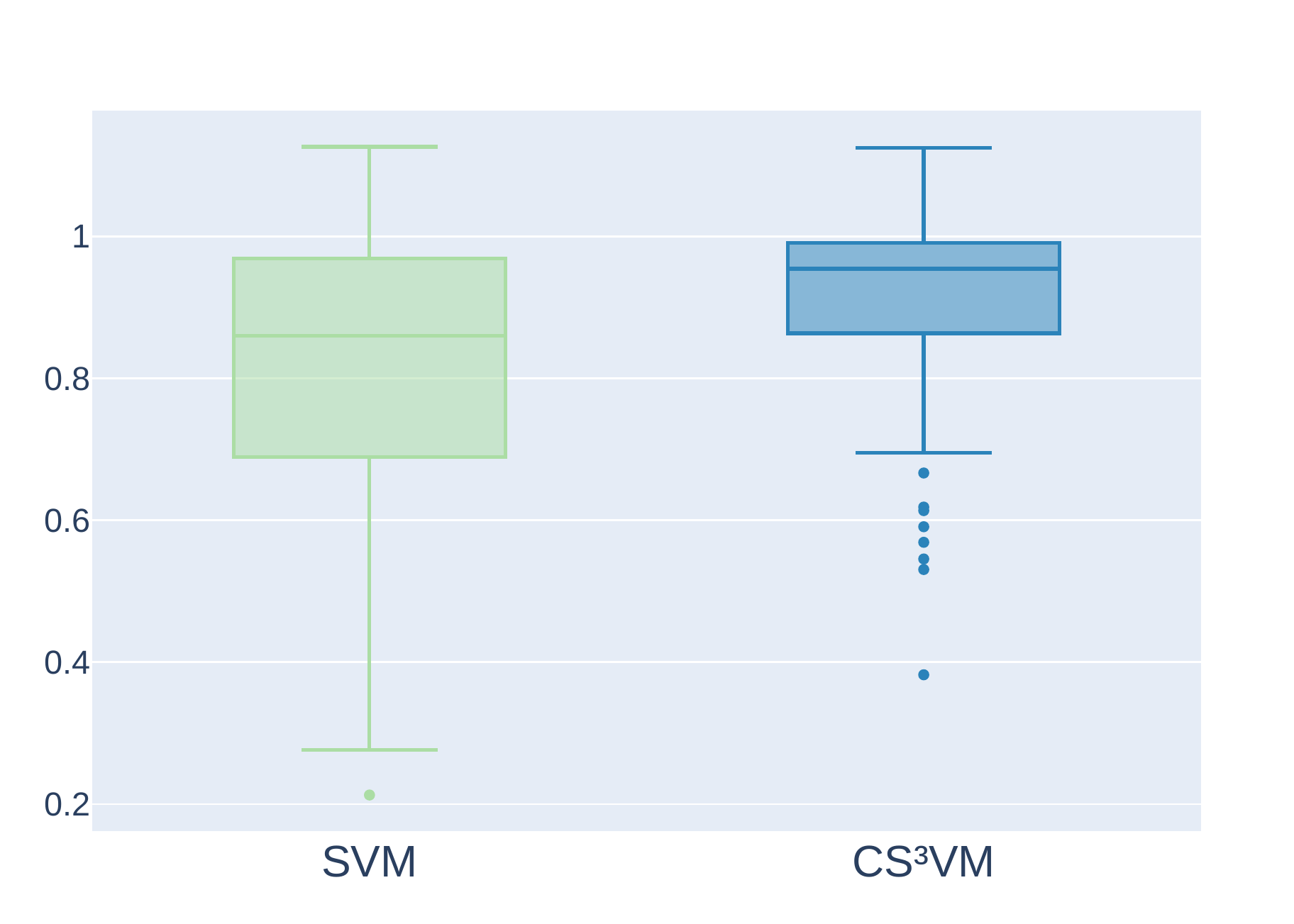}
  \includegraphics[width=0.495\textwidth]{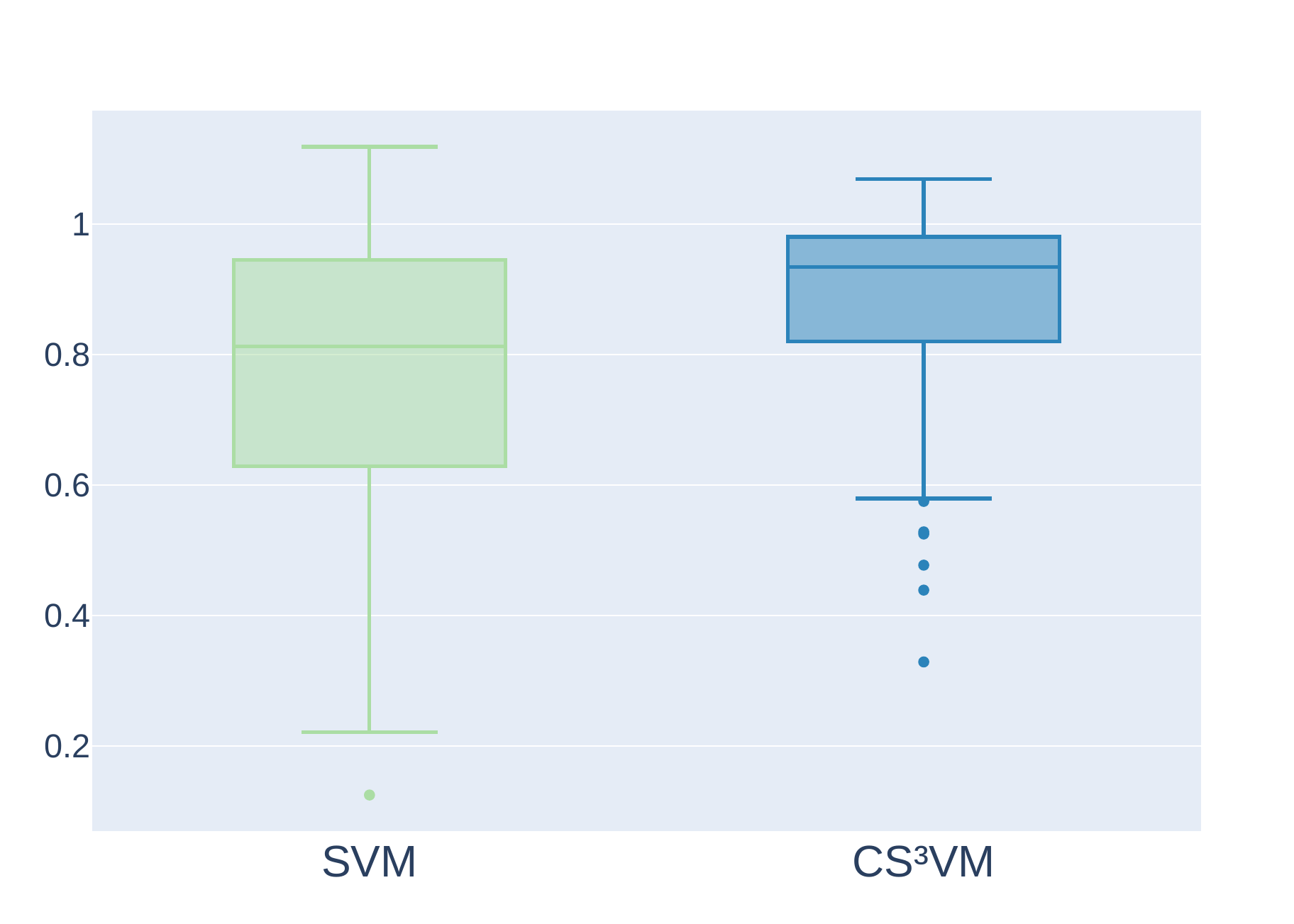}
  \caption{\rev{Relative accuracy $\widehat{\AC}$ w.r.t.\ the true
      hyperplane; see~\eqref{compartrue}.
      Only those instances are considered for which CS$^3$VM terminated.
    Left: Comparison for all data points.
    Right: Comparison only for unlabeled data points.}}
  \label{ACtruCSVM}
\end{figure}
\rev{As can be seen in Figure~\ref{ACtruCSVM}, the relative
accuracy~$\widehat{\AC}$ (w.r.t.\ the true hyperplane) of CS$^3$VM,
is closer to 1 than the relative accuracy of SVM---especially for
the unlabeled data.
This means that using the unlabeled points as well as the cardinality
constraint allows to re-produce the classification of the true
hyperplane with higher accuracy than the standard SVM does.
Besides that, the relative accuracy of the SVM is more spread than the
one of the other approaches, indicating that there is comparable more
variation in the results as compared the results of CS$^3$VM. The box
in the boxplot depicts the range of the medium \SI{50}{\percent}
of the values; \SI{25}{\percent} of the values are
below and \SI{25}{\percent} are above the box.}

\begin{figure}
  \centering
  \includegraphics[width=0.495\textwidth]{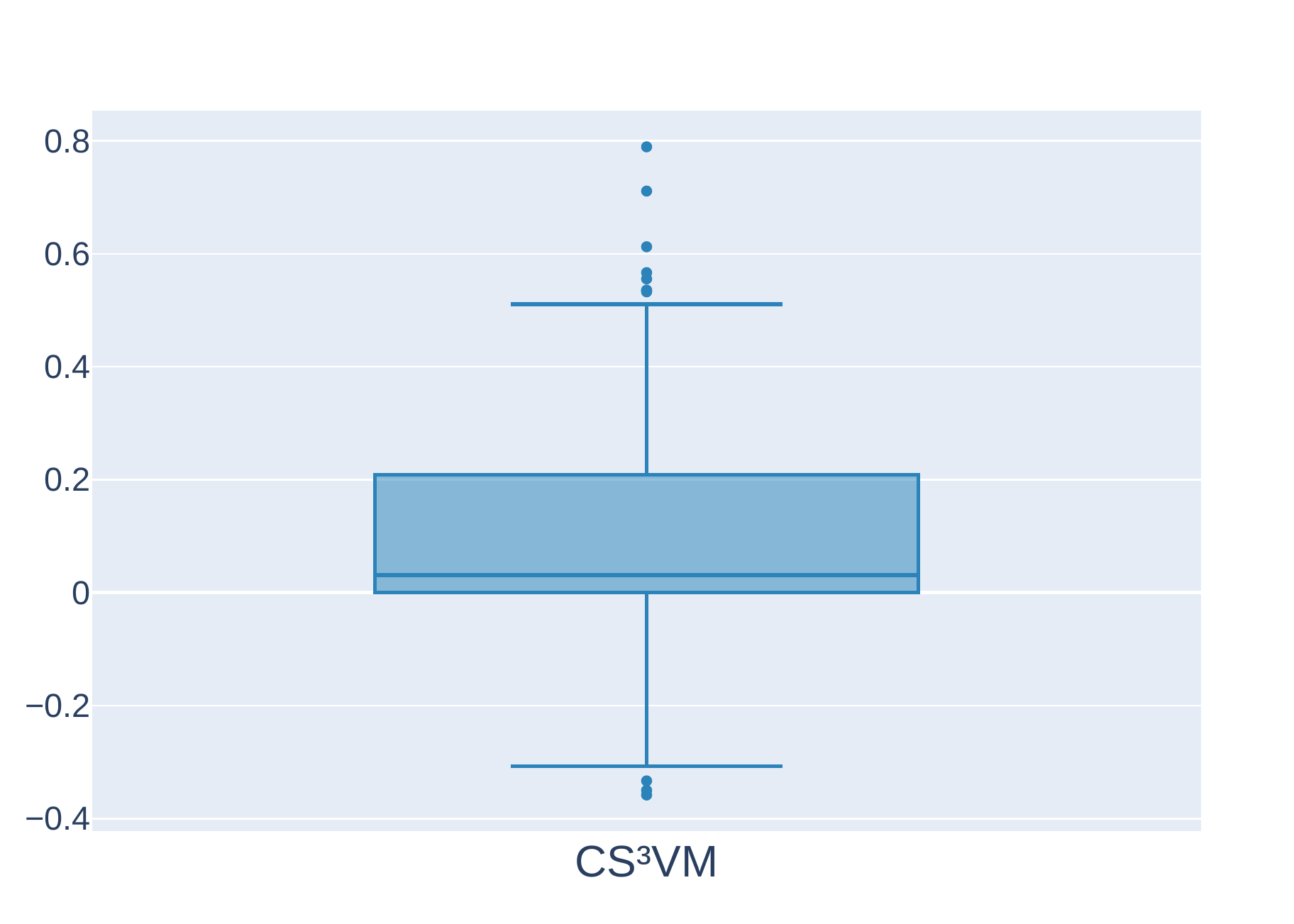}
  \includegraphics[width=0.495\textwidth]{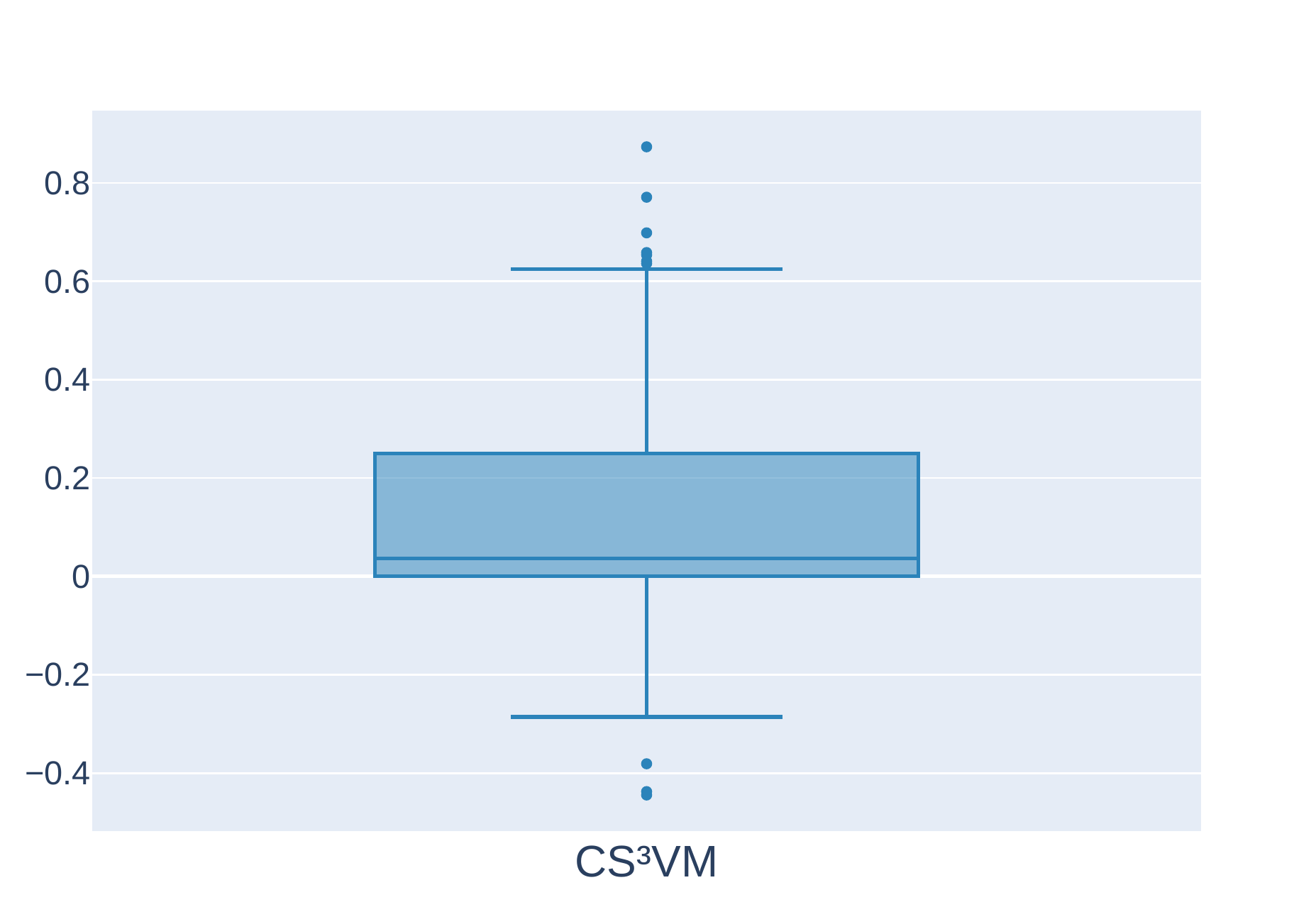}
  \caption{\rev{Accuracy values $\overline{\AC}$ w.r.t.\ the SVM; see
    \eqref{comparSVM} only consider the instances that CS$^3$VM terminated.
    Left: Comparison for all data points.
    Right: Comparison only for unlabeled data points.}}
  \label{ACSVMCSVM}
\end{figure}
\rev{Figure~\ref{ACSVMCSVM} shows that, in almost \SI{75}{\percent} of
  the cases, CS$^3$VM, has $\overline{\AC}$ values larger
than zero, where zero means the same accuracy as the SVM itself. In
the others  \SI{25}{\percent} of the cases, the $\overline{\AC}$ of
CS$^3$VM is slightly smaller than SVM.}

\rev{The second comparison considers only those three approaches that
  actually consider the unlabeled data, i.e., CS$^3$VM, IRCM, and
  WIRCM for all instances. As can be seen in Figure~\ref{ACtru3}, even
  though IRCM does not have an optimality guarantee, it has a better
  relative accuracy~$\widehat{\AC}$ than the hyperplane obtained from
  CS$^3$VM within the time limit. Consequently, as the hyperplane
  obtained from IRCM is used as a warm-start in WIRCM, it also has
  better accuracy.}
\begin{figure}
  \centering
  \includegraphics[width=0.495\textwidth]{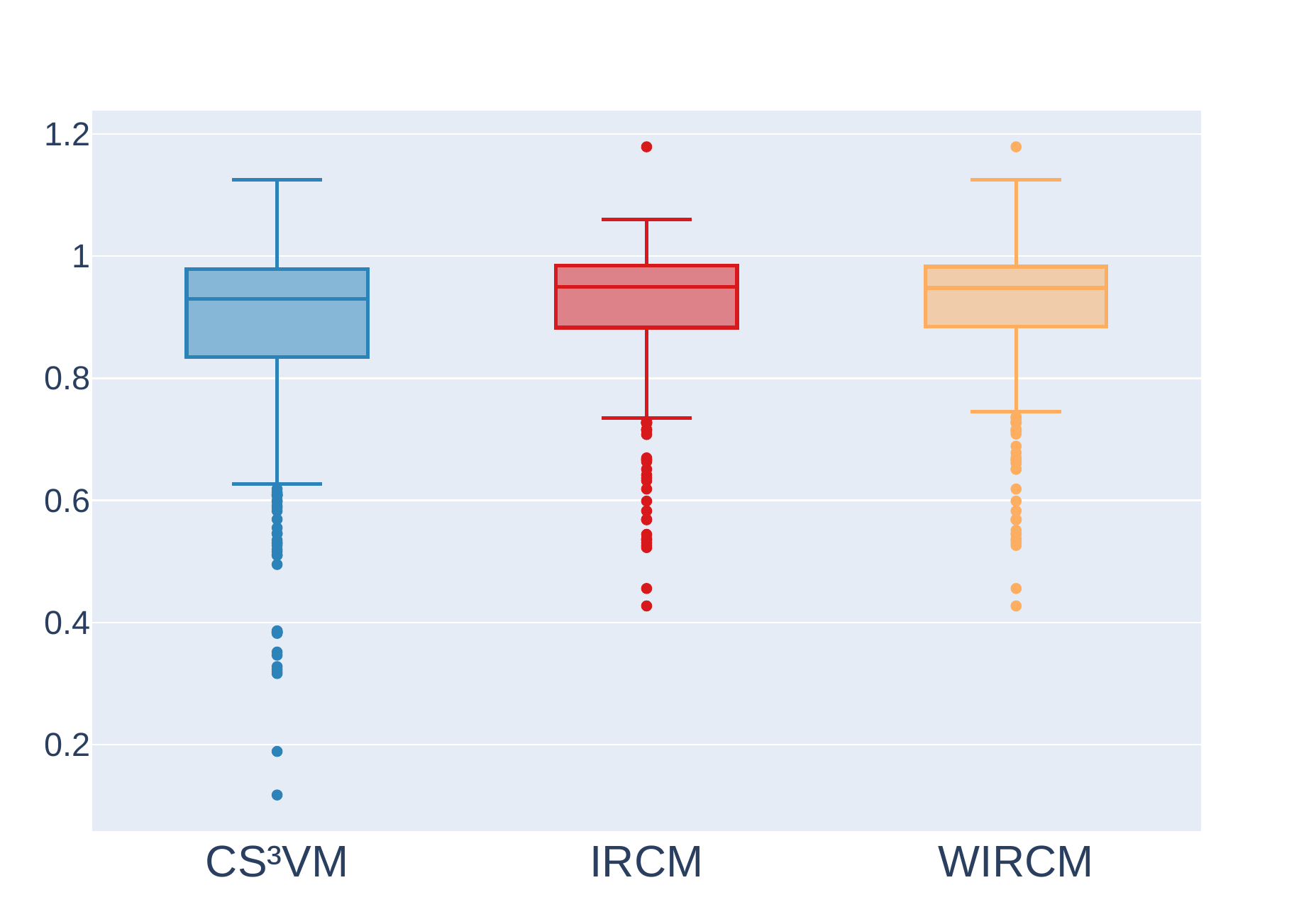}
  \includegraphics[width=0.495\textwidth]{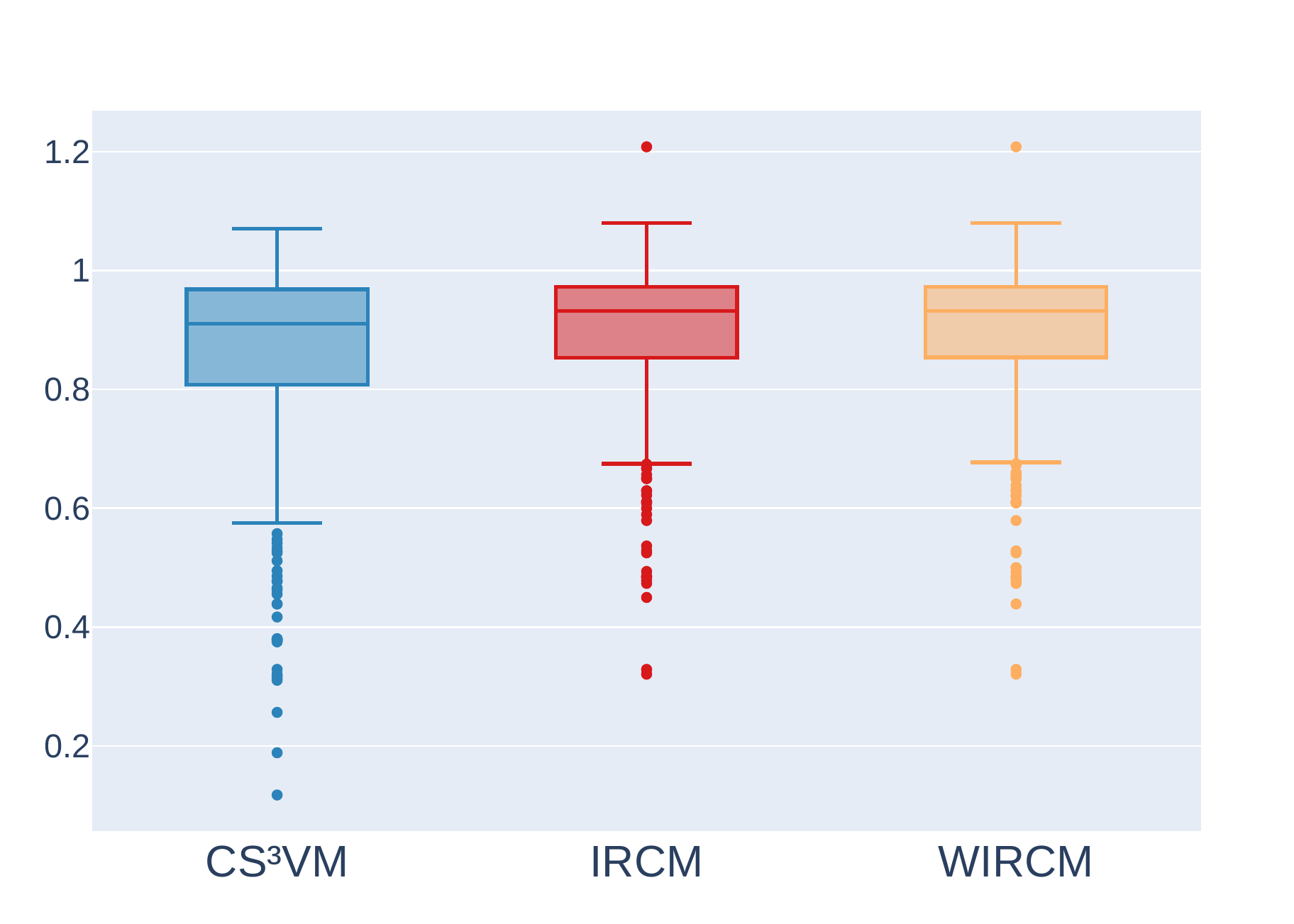}
  \caption{\rev{Relative accuracy $\widehat{\AC}$ w.r.t.\ the true
      hyperplane; see~\eqref{compartrue}.
      Left: Comparison for all data points.
      Right: Comparison only for unlabeled data points.}}
  \label{ACtru3}
\end{figure}
\rev{Figure \ref{ACSV3} shows that, in almost \SI{75}{\percent} of the  cases,
CS$^3$VM, the IRCM, and the WIRCM have $\overline{\AC}$ values larger
than zero. That is,
in general, our methods have greater accuracy than the SVM. Though, some
cases indicate worse $\overline{\AC}$ values for our methods than for
the SVM. This happens because for some instances, the methods (mainly
for CS$^3$VM; see also Figure~\ref{perfomancetime}) do not terminate
within the time limit. Hence, we expect that the number of negative
values will decrease if we would increase the time limit.}
\begin{figure}
  \centering
  \includegraphics[width=0.495\textwidth]{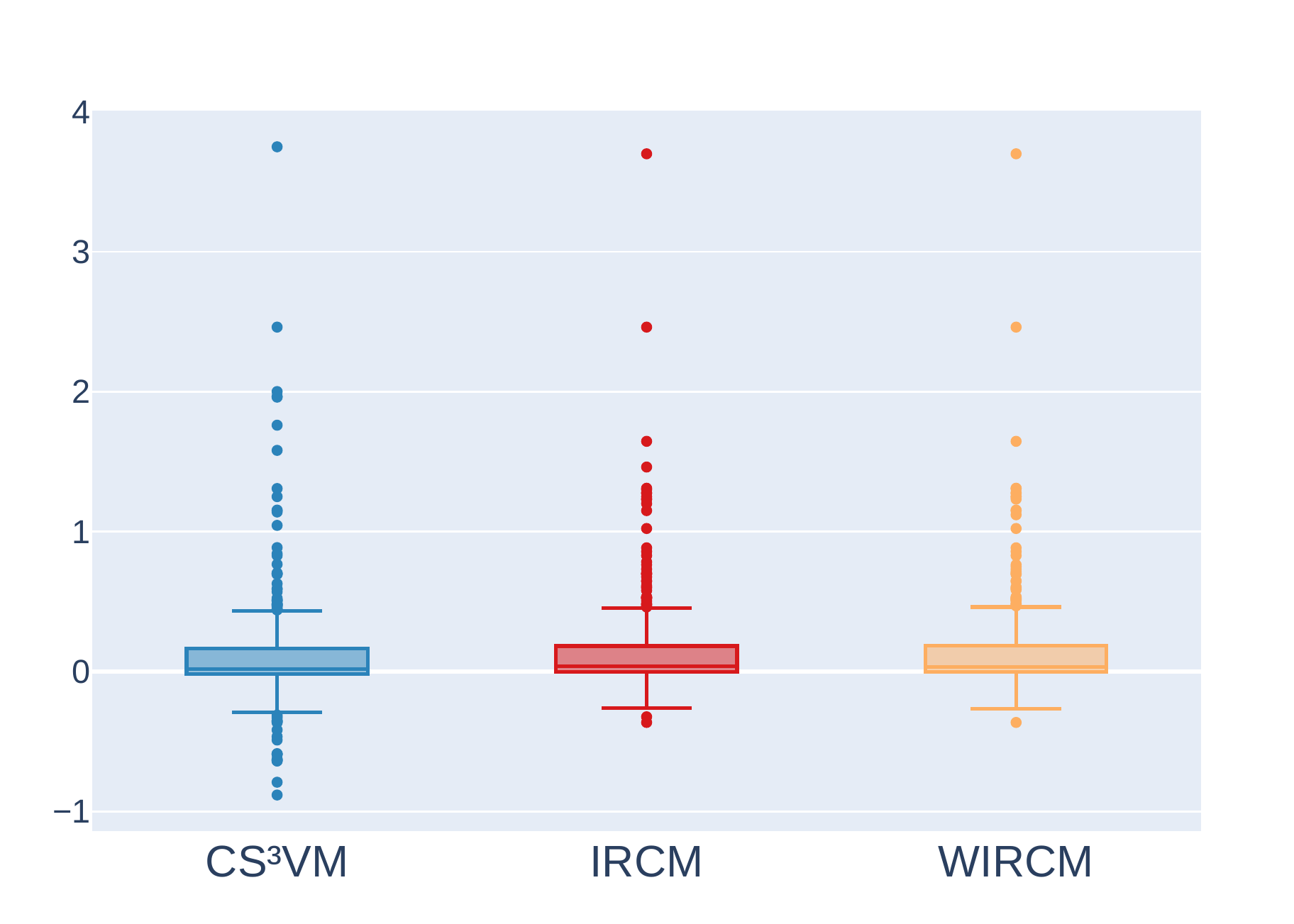}
  \includegraphics[width=0.495\textwidth]{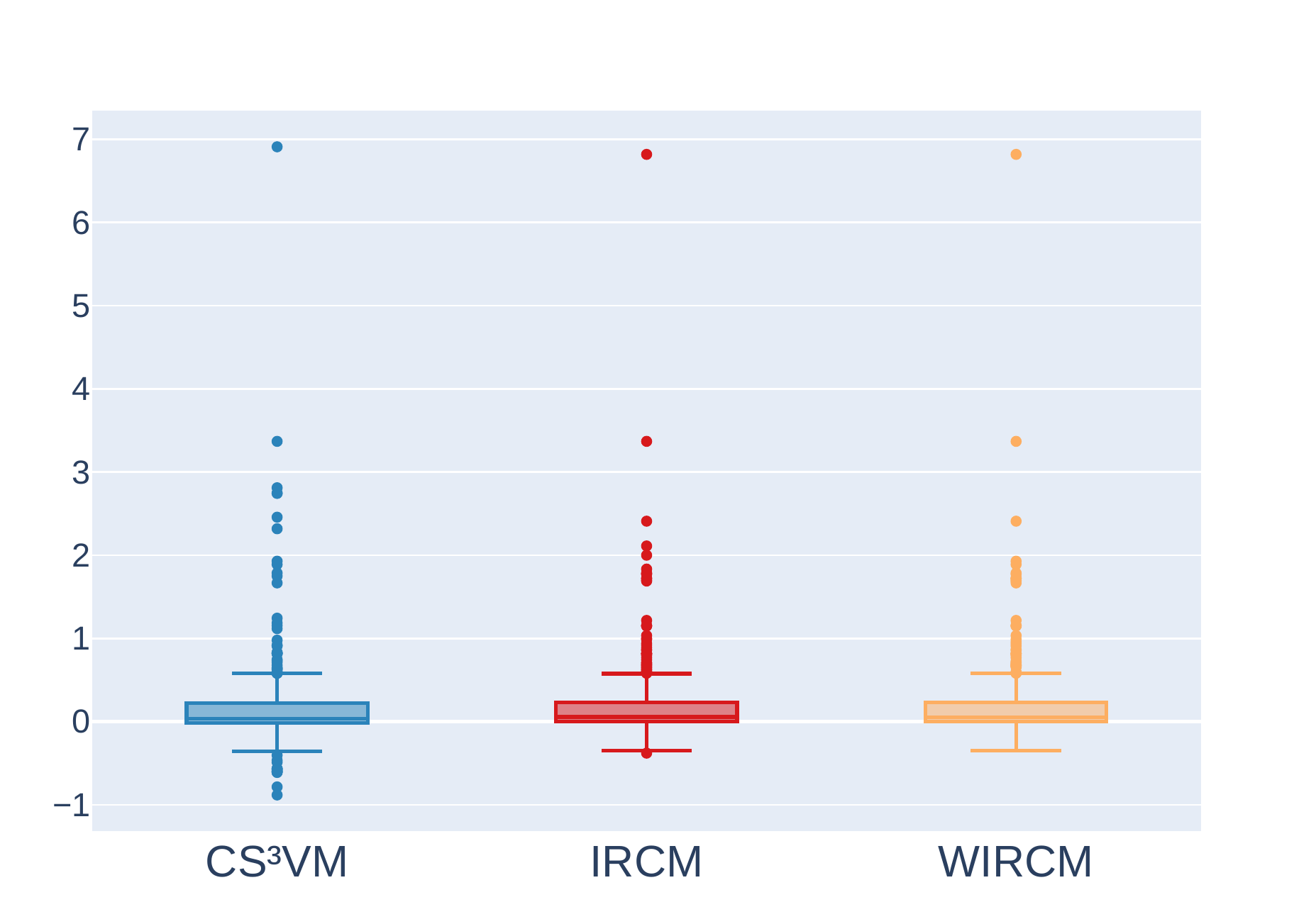}
  \caption{\rev{Accuracy values $\overline{\AC}$ w.r.t.\ the SVM; see
    \eqref{comparSVM} consider all instances.
    Left: Comparison for all data points.
    Right: Comparison only for unlabeled data points.}}
  \label{ACSV3}
\end{figure}

\subsubsection{Precision}
\label{precisionsec}

\rev{We again separate the comparisons as in Section
  \ref{secaccuracy}. Figure \ref{PRtruCSVM} shows that the SVM's
  relative precision $\widehat{\PR}$ is lower than the relative
  precision of CS$^3$VM. This means that CS$^3$VM re-produces the
  classification of the true hyperplane with higher precision than the
  original SVM. Hence, SVM has more false-positive results.
  This happens because the biased sample is more likely to have
  positively labeled data and due to having no information about the
  unlabeled data, the SVM ends up classifying points on the positive
  side. As can be seen in Figure~\ref{PRSVMCSVM}, CS$^3$VM has
  slightly higher $\overline{\PR}$ values than~0, which is the
  baseline here that refers to the SVM itself.
  This means, CS$^3$VM is slightly more precise than the SVM.}

\begin{figure}
  \centering
  \includegraphics[width=0.495\textwidth]{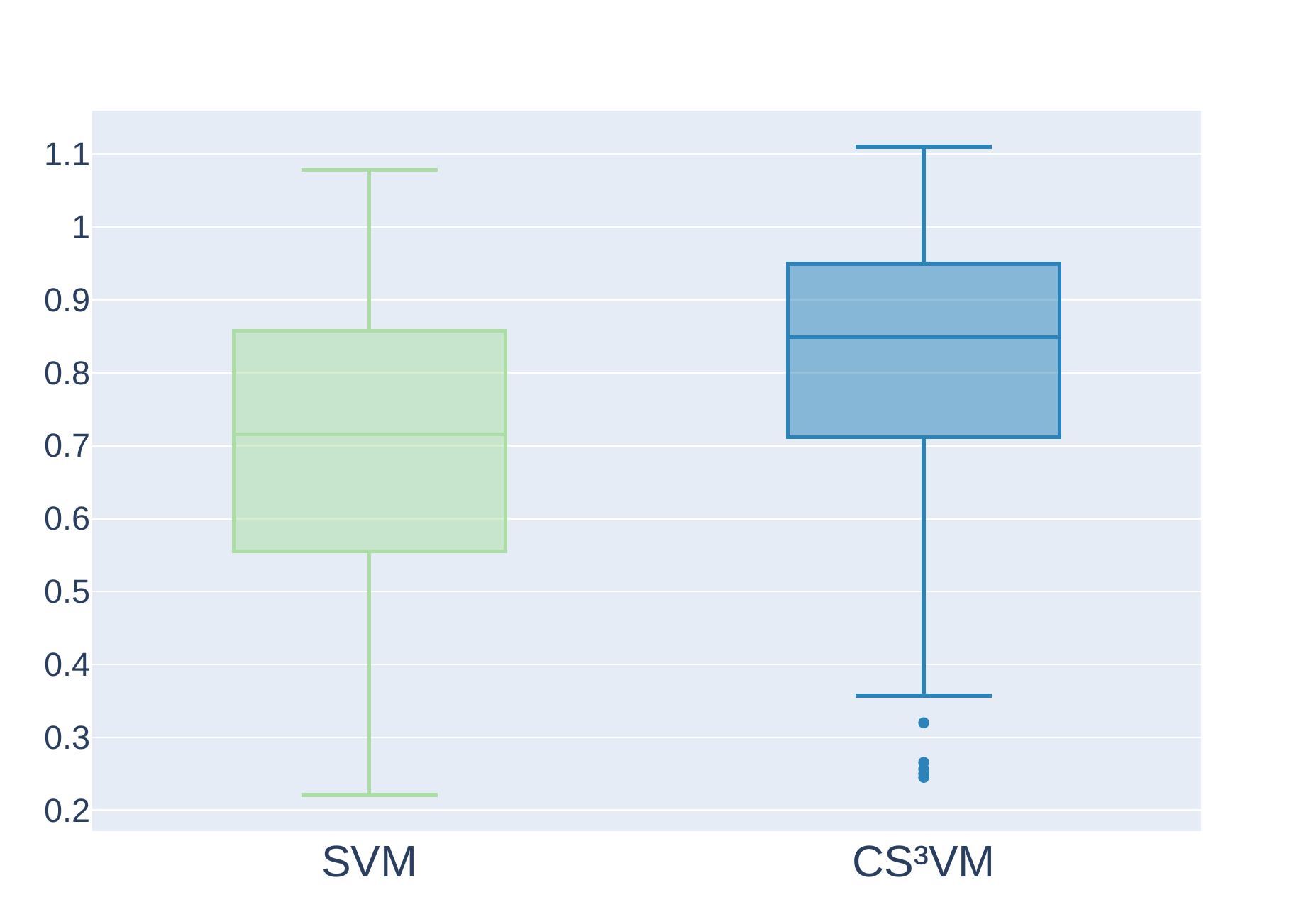}
  \includegraphics[width=0.495\textwidth]{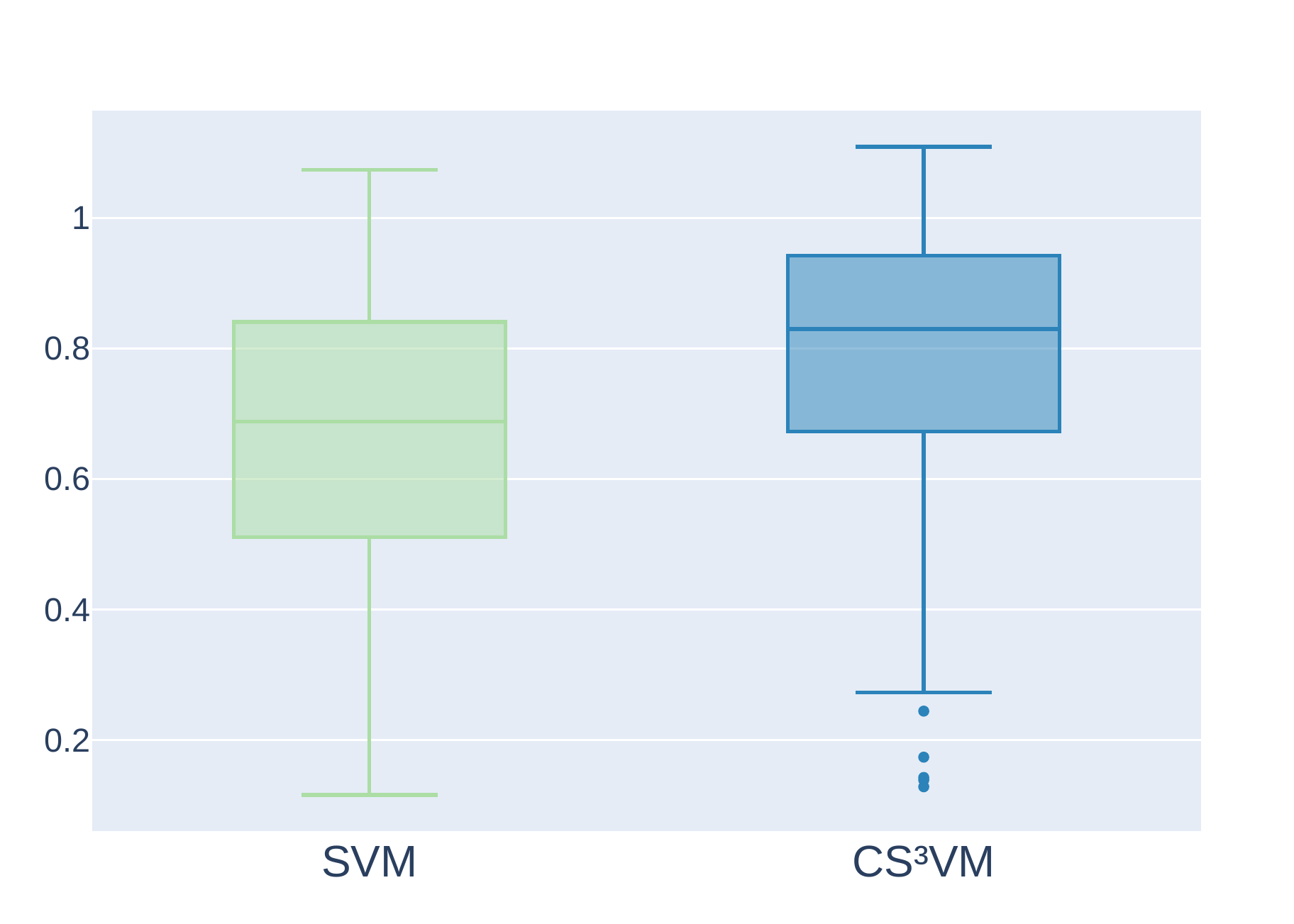}
  \caption{\rev{Relative precision $\widehat{\PR}$ w.r.t.\ the true
      hyperplane as; see \eqref{compartrue}.
      Only those instances are considered for which CS$^3$VM terminated.
      Left: Comparison for all data points.
      Right: Comparison only for unlabeled data points.}}
  \label{PRtruCSVM}
\end{figure}

\begin{figure}
  \centering
  \includegraphics[width=0.495\textwidth]{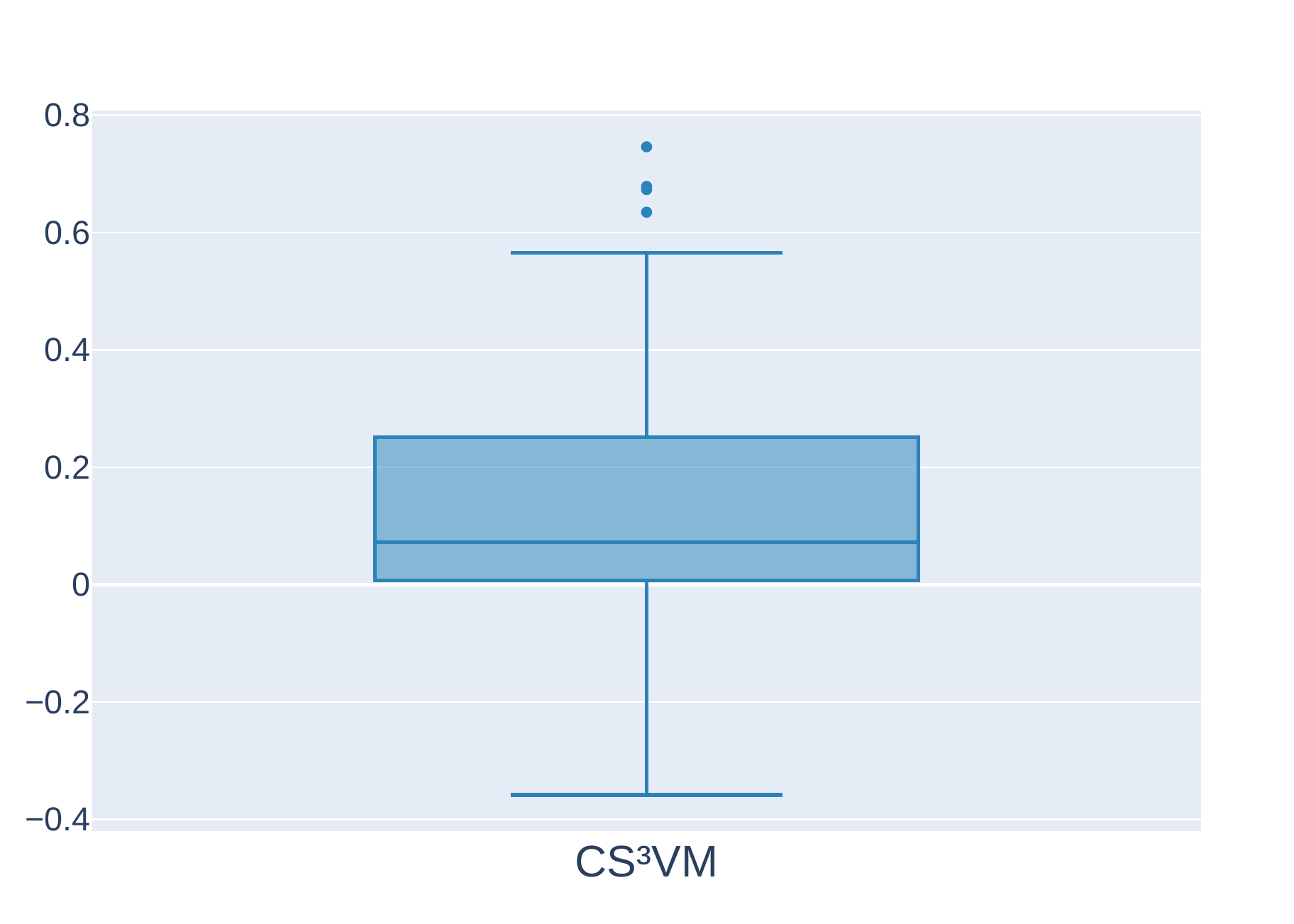}
  \includegraphics[width=0.495\textwidth]{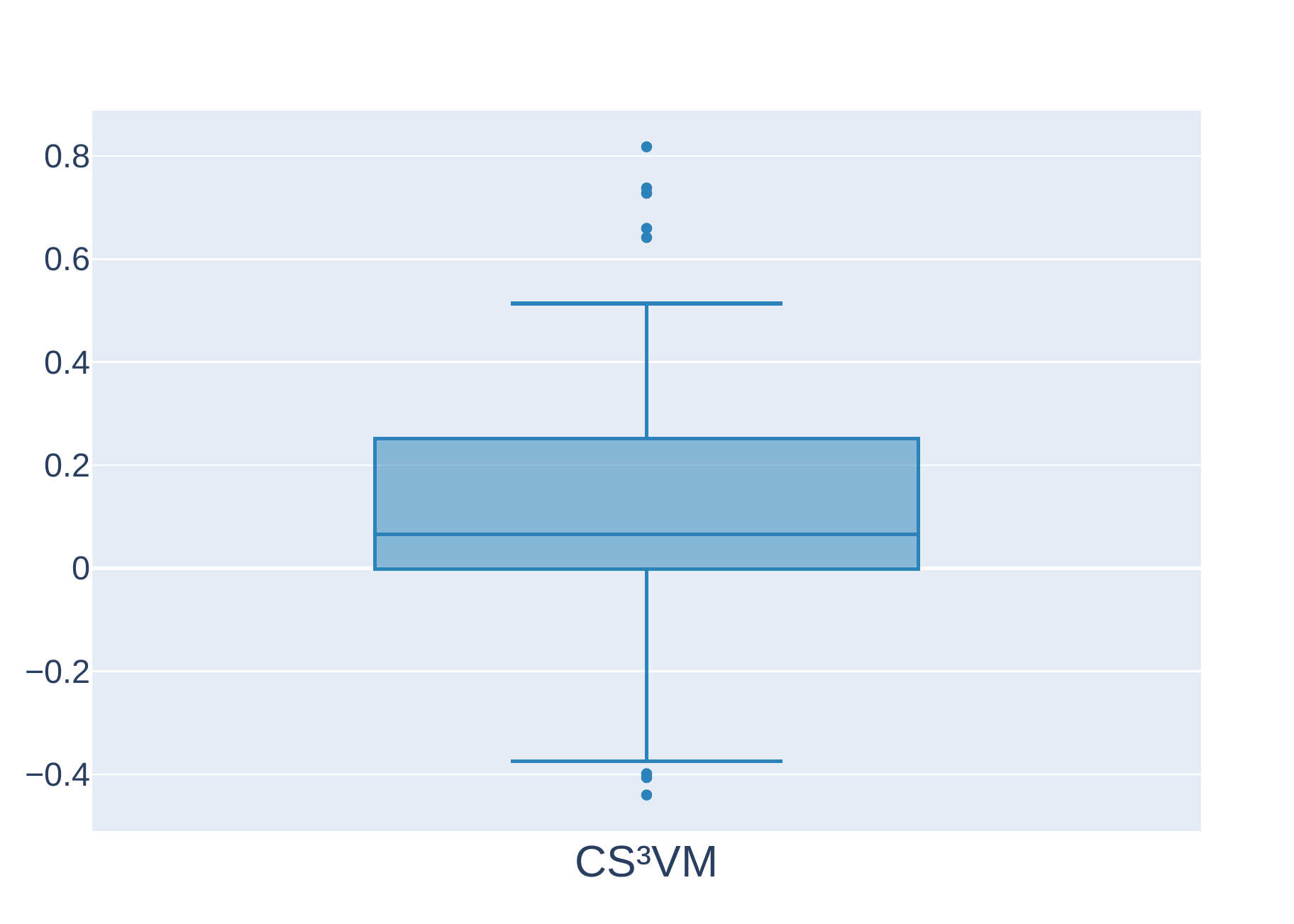}
  \caption{\rev{Precision values $\overline{\PR}$ w.r.t.\ the SVM; see
      \eqref{comparSVM}.
      Only those instances are considered for which CS$^3$VM terminated.
      Left: Comparison for all data points.
      Right: Comparison only for unlabeled data points.}}
  \label{PRSVMCSVM}
\end{figure}

\begin{figure}
  \centering
  \includegraphics[width=0.495\textwidth]{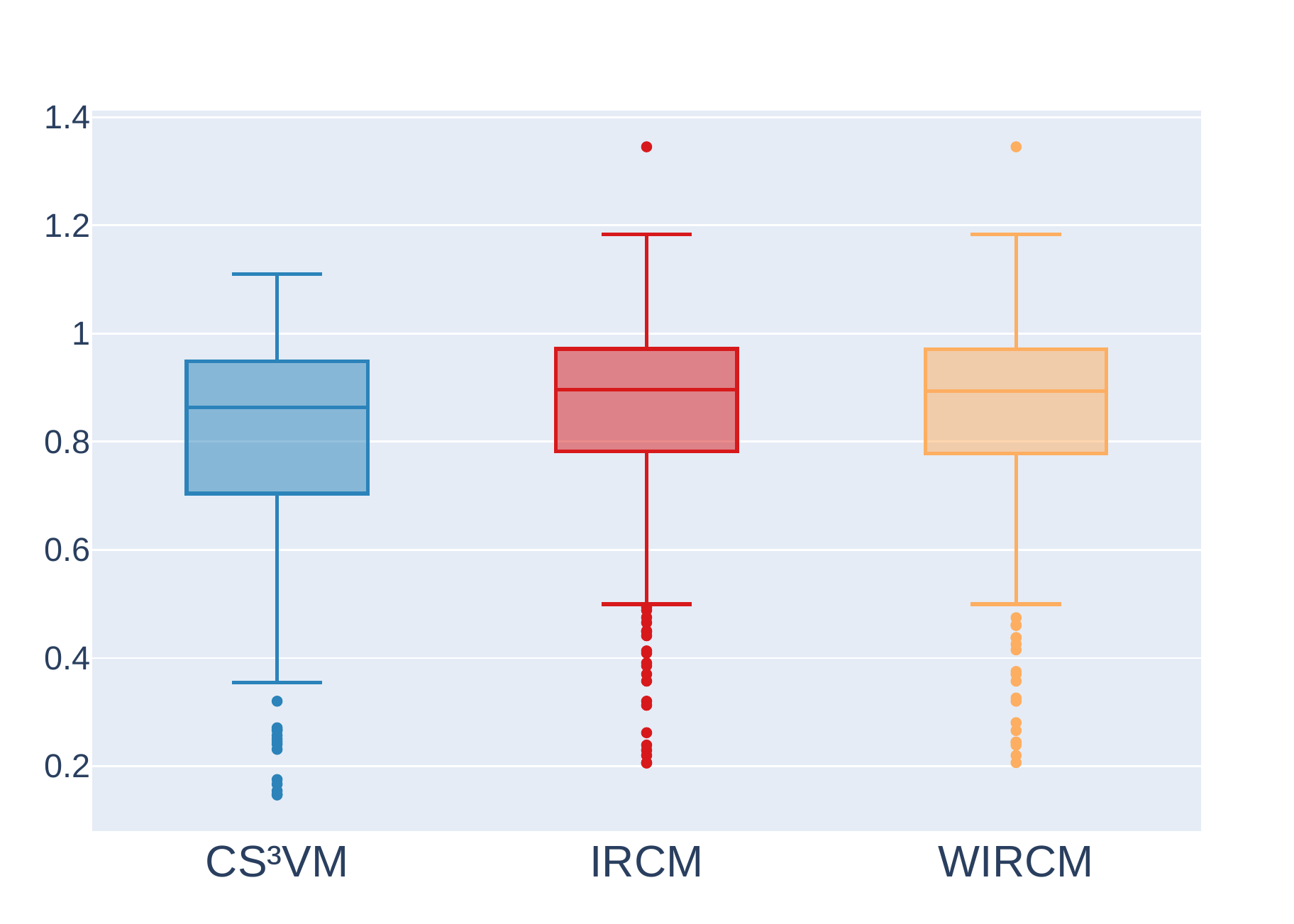}
  \includegraphics[width=0.495\textwidth]{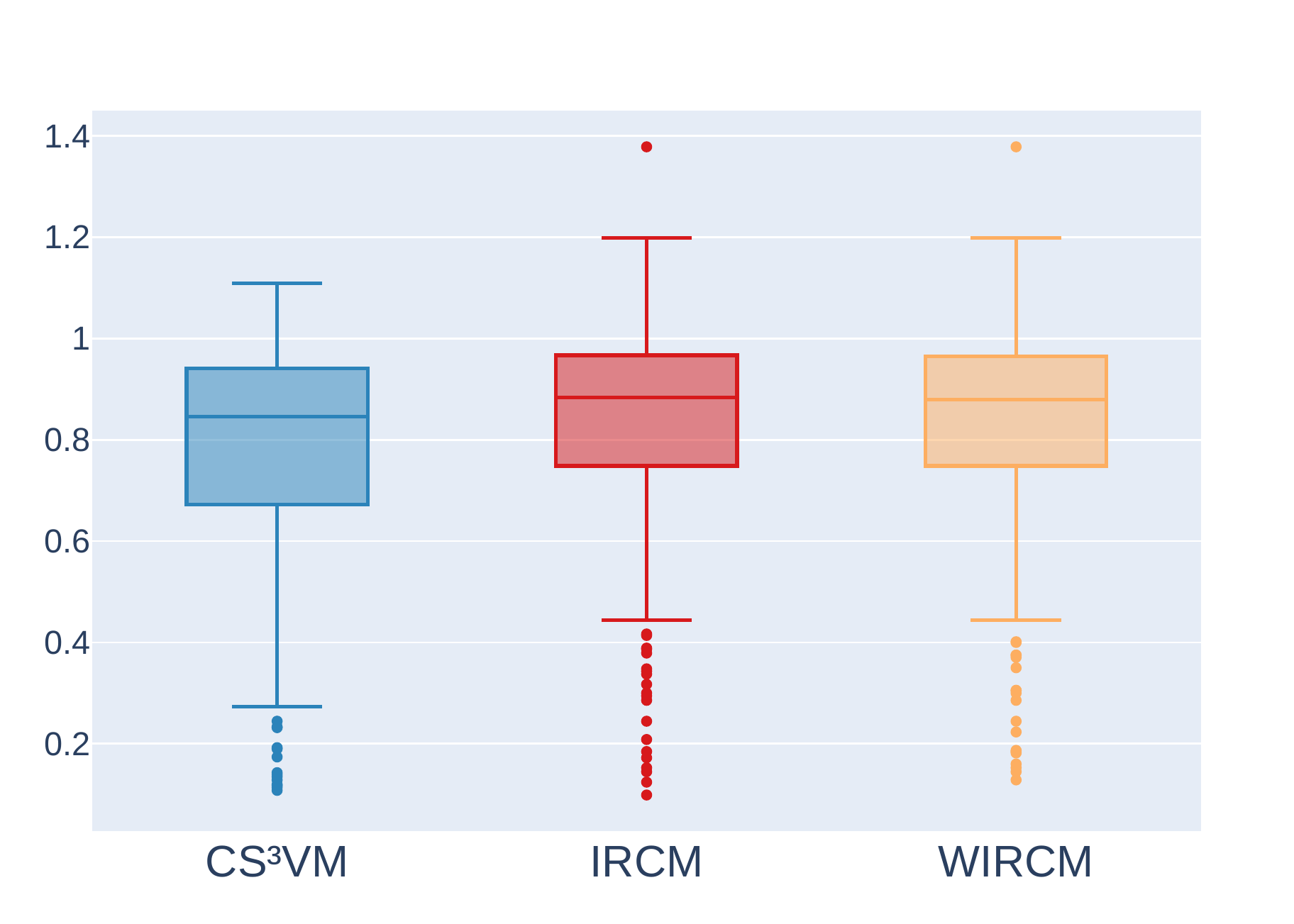}
  \caption{\rev{Relative precision $\widehat{\PR}$ w.r.t.\ the true
      hyperplane as; see \eqref{compartrue}.
      Left: Comparison for all data points.
      Right: Comparison only for unlabeled data points.}}
  \label{PRtru3}
\end{figure}

\begin{figure}
  \centering
  \includegraphics[width=0.495\textwidth]{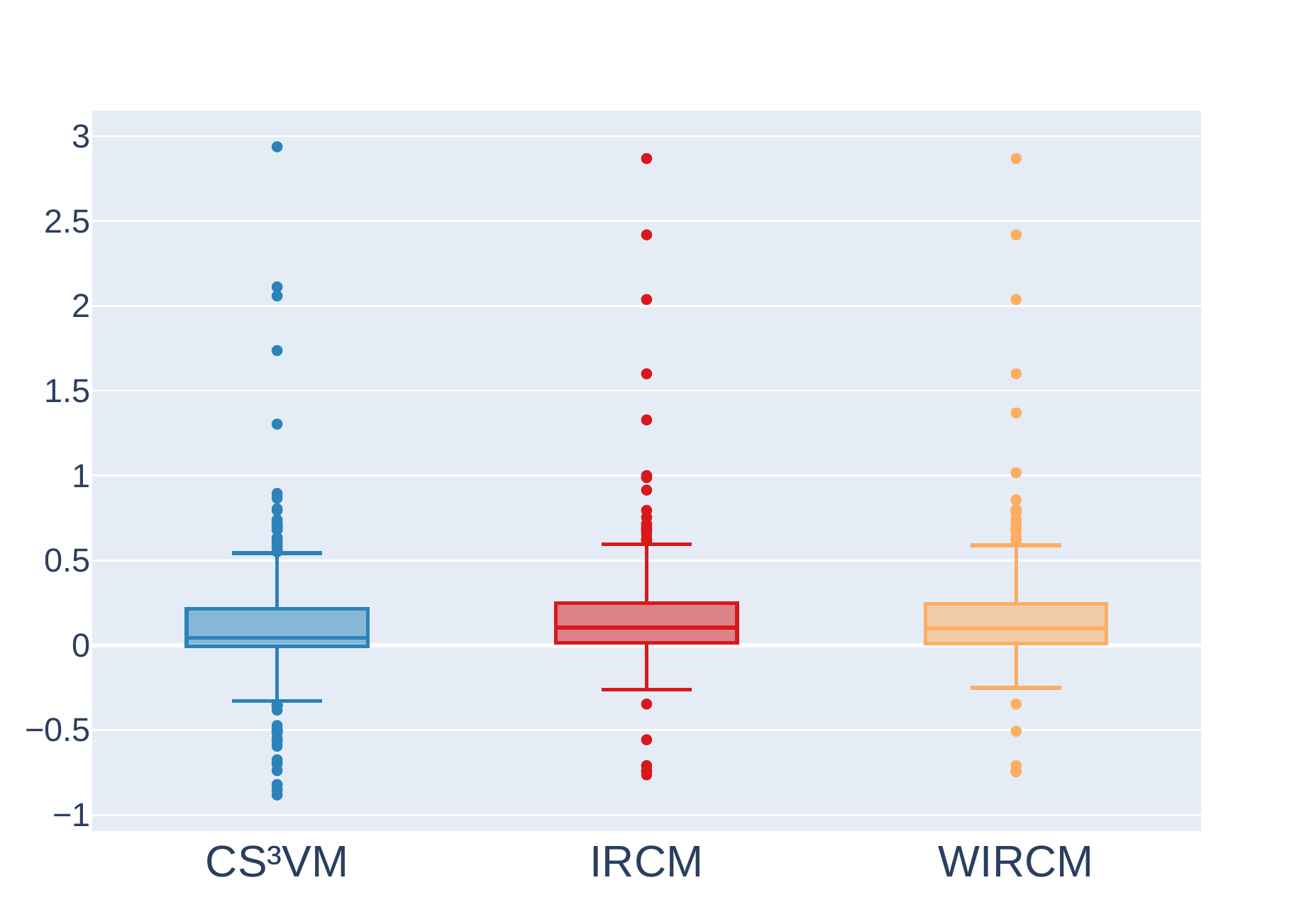}
  \includegraphics[width=0.495\textwidth]{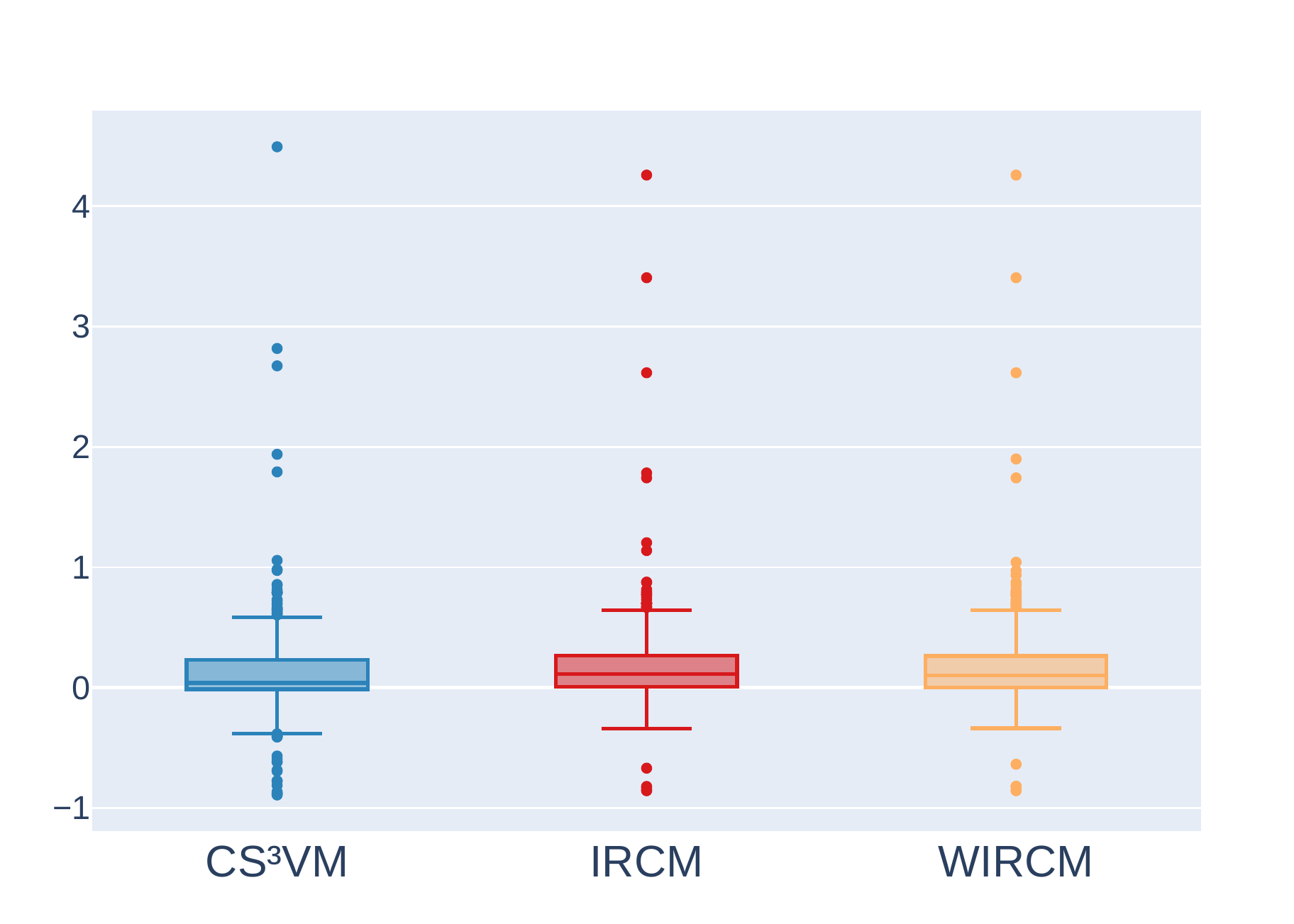}
  \caption{\rev{Precision values $\overline{\PR}$ w.r.t.\ the SVM; see
      \eqref{comparSVM}.
      Left: Comparison for all data points.
      Right: Comparison only for unlabeled data points.}}
  \label{PRSVM3}
\end{figure}

\rev{Figure \ref{PRtru3} shows that the $\widehat{\PR}$ values of the
  IRCM and the WIRCM are less spread than the ones of CS$^3$VM.
  The reason most likely is that the CS$^3$VM approach terminates on
  fewer instances than the IRCM and the WIRCM. As can be seen in
  Figure \ref{PRSVM3}, the IRCM and the WIRCM also have slightly
  higher $\overline{\PR}$ values than~0.
  This means that our methods are slightly more precise than the SVM.
  The negative outliers most likely are due to the same reason as
  those for the respective accuracy values.}


\section{Conclusion}
\label{sec:conclusion}

For many classification problems, it can be costly to obtain labels for
the entire population of interest.
However, aggregate information on how many points are in each class
can be available from external sources.
For this situation, we proposed a semi-supervised SVM that can be
modeled via a big-$M$-based MIQP formulation.
We also presented a rule for updating the big-$M$ in an iterative
re-clustering method and derived further computational techniques such
as tailored dimension reduction and warm-starting to reduce the
computational cost.

In case of simple random samples, our proposed semi-supervised methods
perform as good as the classic SVM approach. However, in many
applications, the available data is coming from non-probability
samples. Hence, there is the risk of obtaining biased samples. Our
numerical study shows that our approaches have better accuracy and
precision than the original SVM formulation in this setting.

The problem of considering a cardinality constraint is
computationally challenging. Our proposed clustering approach
significantly helps to decrease the run time and to find an objective
function value that is very close to the optimal value.
Besides that, the clustering approach maintains the same accuracy and
precision as the MIQP formulation.
Moreover, using the clustering approach as a warm-start and fixing some
unlabeled points on one side of the hyperplane helps to improve the
quality of the objective function value again.
Hence, the newly proposed methods lead to a significant improvement
compared to just solving the classic MIQP formulation using a standard
solver.

\rev{Despite these contributions, there is still room for improvement
  and future work.
  First, we  only considered the linear SVM kernel.
  For future work, the development of methods for other kernels, such
  as a Gaussian kernel, can be a valuable topic.
  Moreover, the use of other norms than the 2-norm could be analyzed
  as well and the formal hardness of the considered problem should be
  settled.
  Finally, the adaptation of our approaches for multiclass SVMs using
  a one-vs.-rest strategy may be another reasonable future work.}


\section*{Acknowledgements}

The authors thank the DFG for their support within
RTG~2126 \enquote{Algorithmic Optimization}.


\printbibliography

\appendix
\section{Detailed Information on the Instances}
\label{sec:deta-inform-inst}

\begin{center}
  \begin{longtable}{l c c c }
    \caption{Overview over the entire test set with the number of
      points ($N$) and the dimension ($d$)}
    \label{table1}\\
    \toprule
    ID & Instance  &$ N $ &$d$ \\
    \midrule
    1                         & prnn\_synth         & 250          & 2   \\
    $2^*$   &  analcatdata\_asbestos        &  73    &  3   \\
    $3^*$    &  lupus                        &  87    &  3  \\
    4     &  analcatdata\_boxing1         &  120   &  3  \\
    5       &  analcatdata\_boxing2         &  132   &  3   \\
    6        &  haberman     &  289   &  3     \\
    7        &  analcatdata\_happiness       &  60    &  3  \\
    $8^*$     &  analcatdata\_aids            &  50    &  4  \\
    9  &  analcatdata\_lawsuit         &  263   &  4  \\
    10     &  iris       &  147   &  4 \\
    11          &  hayes\_roth        &  93    &  4 \\
    12       &  balance\_scale               &  625   &  4 \\
    13       &  parity5     &  32    &  5  \\
    $14^*$        &  bupa                         &  341   &  5  \\
    15        &  irish                        &  470   &  5 \\
    16         &  phoneme       &  5349  &  5  \\
    17                        &  tae                          &  110   &  5  \\
    18                        &  new\_thyroid                 &  215   &  5  \\
    $19^*$                        &  analcatdata\_bankruptcy      &  50    &  6 \\
    $20^*$                        &  analcatdata\_creditscore     &  100   &  6  \\
    21                        &  mux6       &  64    &  6 \\
    22                        &  monk3                        &  357   &  6 \\
    23                        &  monk1                        &  432   &  6 \\
    24                        &  monk2                        &  432   &  6 \\
    25                        &  appendicitis                 &  106   &  7  \\
    26                        &  prnn\_crabs                  &  200   &  7  \\
    $27^*$ &  penguins                     &  333   &  7  \\
    28 &  postoperative\_patient\_data &  78    &  8  \\
    $29^*$                        &  biomed                       &  209   &  8  \\
    $30^*$                        &  pima                         &  768   &  8  \\
    $31^*$                        &  cars                         &  392   &  8  \\
    32                        &  analcatdata\_japansolvent    &  52    &  9  \\
    33                        &  glass2                       &  162   &  9  \\
    34                        &  breast\_cancer               &  272   &  9  \\
    35                        &  saheart                      &  462   &  9  \\
    36                        &  threeOf9                     &  512   &  9 \\
    37                        &  profb                        &  672   &  9  \\
    38                        &  breast\_w                    &  463   &  9 \\
    39                        &  tic\_tac\_toe                &  958   &  9  \\
    40                        &  xd6                          &  512   &  9  \\
    41                        &  cmc                          &  1425  &  9  \\
    $42$ &  analcatdata\_cyyoung9302     &  92    &  10\\
    $43$ &  analcatdata\_cyyoung8092     &  97    &  10 \\
    44 &  breast                       &  691   &  10 \\
    45                        &  flare                        &  315   &  10 \\
    46                        &  parity5+5                    &  1024  &  10\\
    $47$                        &  magic                        &  18905 &  10 \\
    48                        &  analcatdata\_fraud           &  42    &  11 \\
    $49$                        &  heart\_statlog               &  270   &  13 \\
    50                        & heart\_h                                                                                    & 293                          & 13  \\
    51 &  hungarian                             &  293   &  13   \\
    $52^*$ &  cleve           &  302   &  13   \\
    $53^*$ &  heart\_c         &  302   &  13  \\
    54 &  wine\_recognition           &  178   &  13   \\
    $55^*$ &  australian               &  690   &  14 \\
    $56^*$ &  adult   &  48790 &  14   \\
    $57^*$ &  schizo    &  340   &  14   \\
    $58^*$ &  buggyCrx    &  690   &  15   \\
    59 &  labor        &  57    &  16   \\
    60 &  house\_votes\_84           &  342   &  16   \\
    61 &  hepatitis             &  155   &  19   \\
    $62^*$ &  credit\_g                    &  1000  &  20   \\
    63 &  gametes\_e\_0.1H             &  1599  &  20   \\
    64 &  gametes\_e\_0.4H                &  1600  &  20   \\
    65 &  gametes\_e\_0.2H               &  1600  &  20   \\
    66 &  gametes\_h\_50 &  1592  &  20   \\
    67 &  gametes\_h\_75 &  1599  &  20   \\
    $68^*$ &  churn      &  5000  &  20   \\
    $69^*$ &  ring        &  7400  &  20   \\
    70 &  twonorm         &  7400  &  20   \\
    71 &  waveform\_21        &  5000  &  21   \\
    72 &  ann\_thyroid            &  7129  &  21   \\
    73 &  spect   &  228   &  22   \\
    74 &  horse\_colic            &  357   &  22   \\
    75 &  agaricus\_lepiota         &  8124  &  22   \\
    $76^*$ &  hypothyroid              &  3086  &  25   \\
    $77^*$ &  dis                  &  3711  &  29   \\
    $78^*$ &  allhypo      &  3709  &  29   \\
    $79^*$ &  allbp                                       &  3711  &  29   \\
    $80^*$ &  breast\_cancer\_wisconsin           &  569   &  30   \\
    81 &  backache &  180   &  32   \\
    82 &  ionosphere        &  351   &  34   \\
    83 &  chess        &  3196  &  36   \\
    84 &  waveform\_40    &  5000  &  40   \\
    85 &  connect\_4                   &  67557 &  42   \\
    86 &  spectf      &  267   &  44   \\
    $87^*$ &  tokyo1      &  959   &  44   \\
    88 &  molecular\_biology\_promoters       &  106   &  57   \\
    $89^*$ &  spambase            &  4210  &  57   \\
    90 &  sonar         &  208   &  60   \\
    91 &  splice        &  2903  &  60   \\
    92 &  coil2000       &  8380  &  85   \\
    $93^*$ &  Hill\_Valley\_without\_noise     &  1212  &  100  \\
    $94^*$ &  clean1        &  476   &  168  \\
    $95^*$ &  clean2                  &  6598  &  168  \\
    96 &  dna                         &  3002  &  180  \\
    97 &  gametes\_e\_1000atts       &  1600  &  1000 \\
    \bottomrule
  \end{longtable}
\end{center}



\section{Further Numerical Results}
\label{sec:furth-numer-results}

Besides the measures of accuracy and precision, we compare two further
measures in this section.
First, recall ($\RE$) measures the percentage of points with positive
label that are actually classified as positive.
It is formally given by
\begin{equation}
  \label{recall}
  \RE \define \frac{\TP}{\TP+\FN}.
\end{equation}
Note that for applications such as cancer diagnosis, it is relevant to
evaluate recall because it is more important to flag cancer rather
than to do not.
Also in cases of rare positive labels, recall is often the favored
metric.
Note that values close to 1 indicate a better classification here.

Second, we also compare the false positive rate ($\FPR$), which
measures the probability of points with negative labels being
classified as positive:
\begin{equation}
  \label{fpr}
  \FPR \define \frac{\FP}{\TN+\FP}.
\end{equation}
This quantity is important in some applications such as quality control,
where a false positive can cause more issues than a false negative.
Note that for $\FPR$, the lower the value, the better the
classification.

The main comparison in terms of recall and false positive rate is
w.r.t.\ the ``true hyperplane'', i.e., the solution of
Problem~\eqref{l2svm} on the complete data with all $N$ points and all
labels available.
The main question is how close the recall and false positive rate is
to the one of the true hyperplane.
Hence, we compute the ratios of the recall and false positive rate
according to
\begin{equation}
  \label{compartrue2}
  \widehat{\RE} \define \frac{\RE}{\RE_{\true}},
  \quad
  \widehat{\FPR} \define \frac{\FPR}{\FPR_{\true}},
\end{equation}
where $\RE_{\true}$ and $\FPR_{\true}$ are computed as
in~\eqref{recall} and~\eqref{fpr} for the true hyperplane.

As can be seen in Figure~\ref{REtru}, the SVM's relative recall is a
little bit larger than the one of the other methods. As in
Section~\ref{precisionsec}, this happens because the biased sample is
more likely to have positive labeled data and having no information about
the unlabeled data, the SVM ends up classifying points on the positive side.
\begin{figure}
  \centering
  \includegraphics[width=0.495\textwidth]{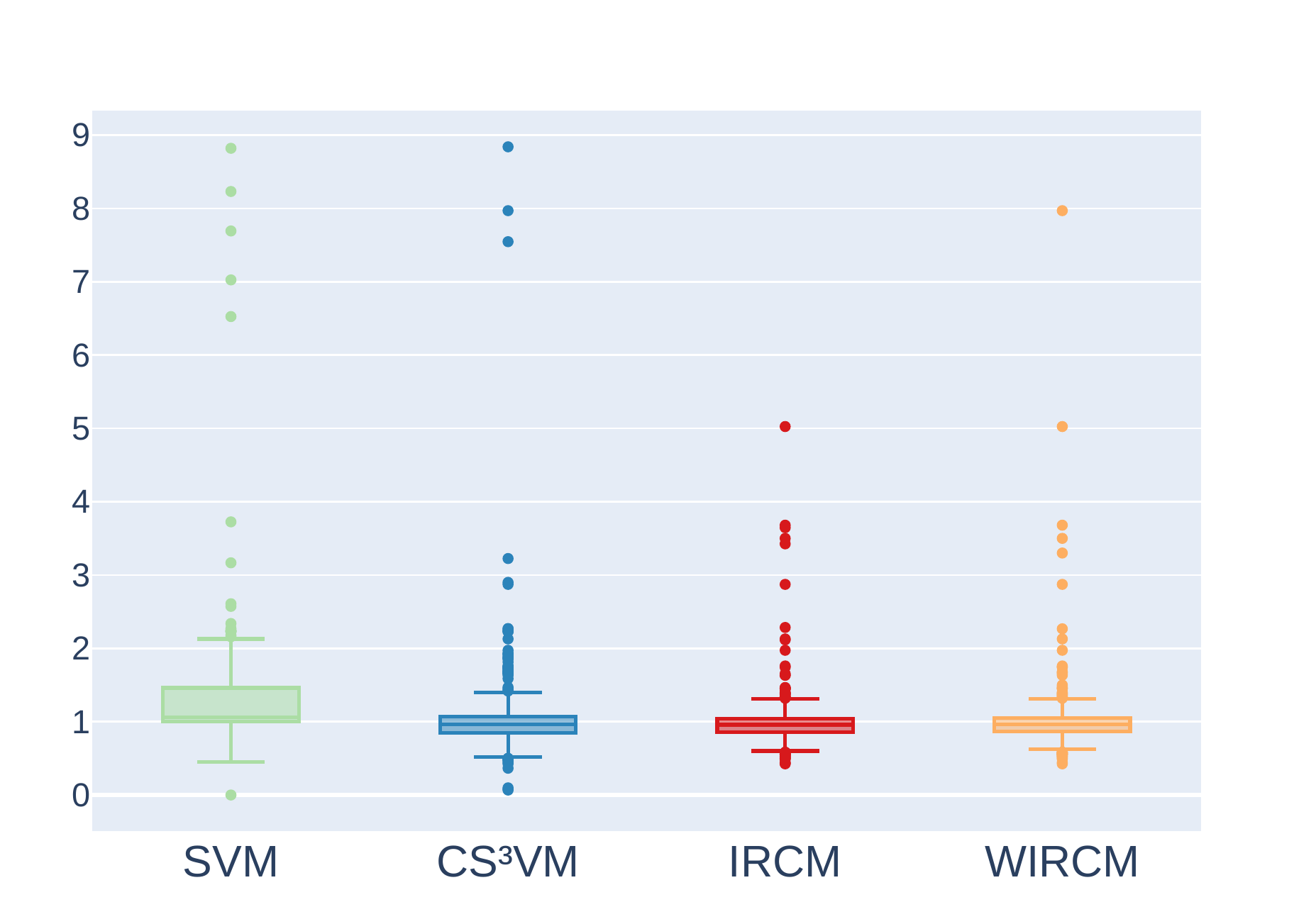}
  \includegraphics[width=0.495\textwidth]{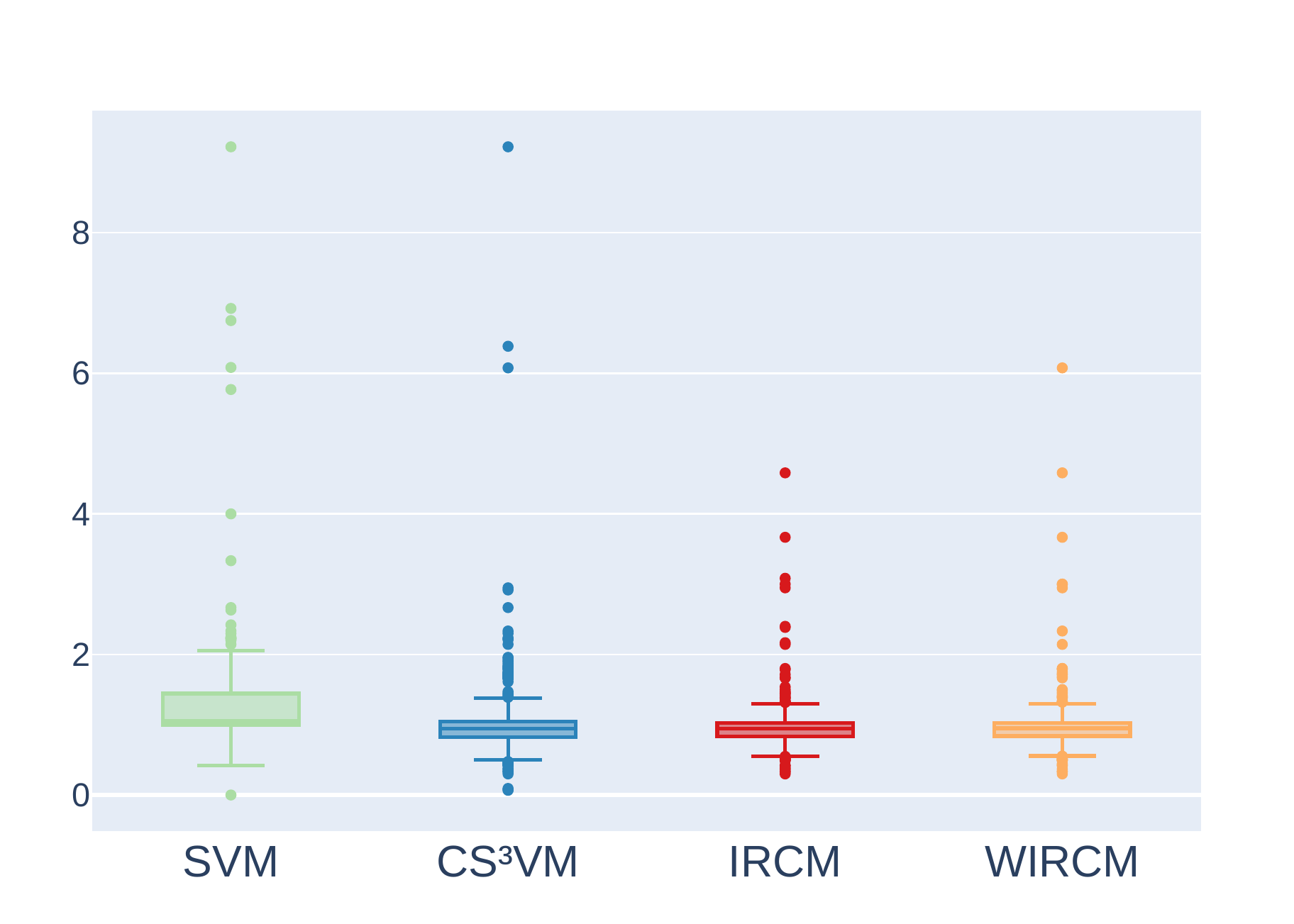}
  \caption{Relative recall $\widehat{\RE}$ w.r.t.\ the true
    hyperplane; see \eqref{compartrue2}.
    Left: Comparison for all data points.
    Right: Comparison only for unlabeled data points.}
  \label{REtru}
\end{figure}

Figure~\ref{FPRtru} shows that CS$^3$VM, the IRCM, and the WIRCM have
lower $\widehat{\FPR}$ values than the original SVM.
This means that the newly proposed methods have a lower false positive
rate than the original SVM.
The fact that CS$^3$VM terminates for less instances than the
IRCM explains why the IRCM has a lower relative false positive rate
than CS$^3$VM.
Finally, since the WIRCM uses the IRCM for warm-starting, the WIRCM
also has better relative false positive rates than CS$^3$VM.
\begin{figure}
  \centering
  \includegraphics[width=0.495\textwidth]{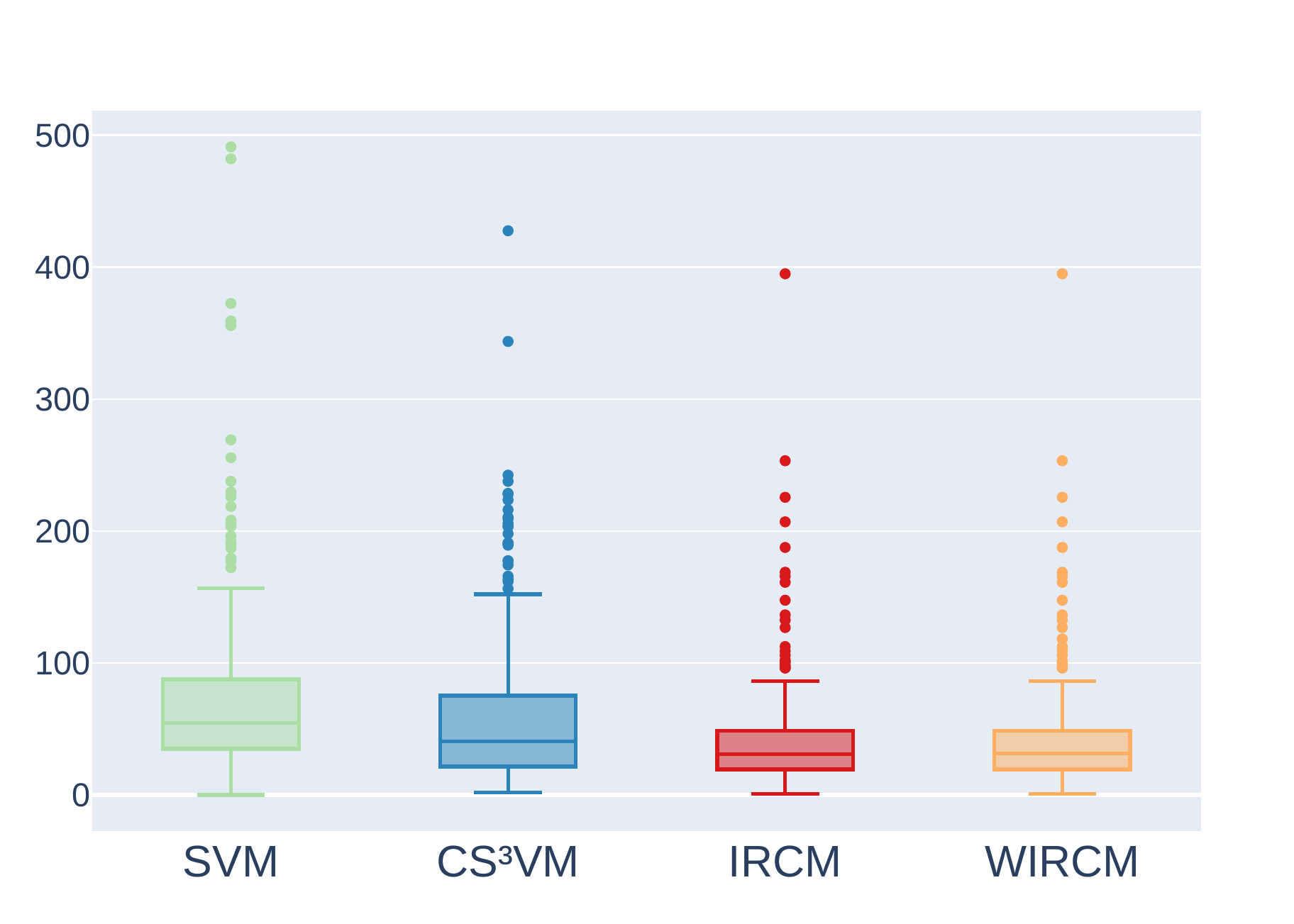}
  \includegraphics[width=0.495\textwidth]{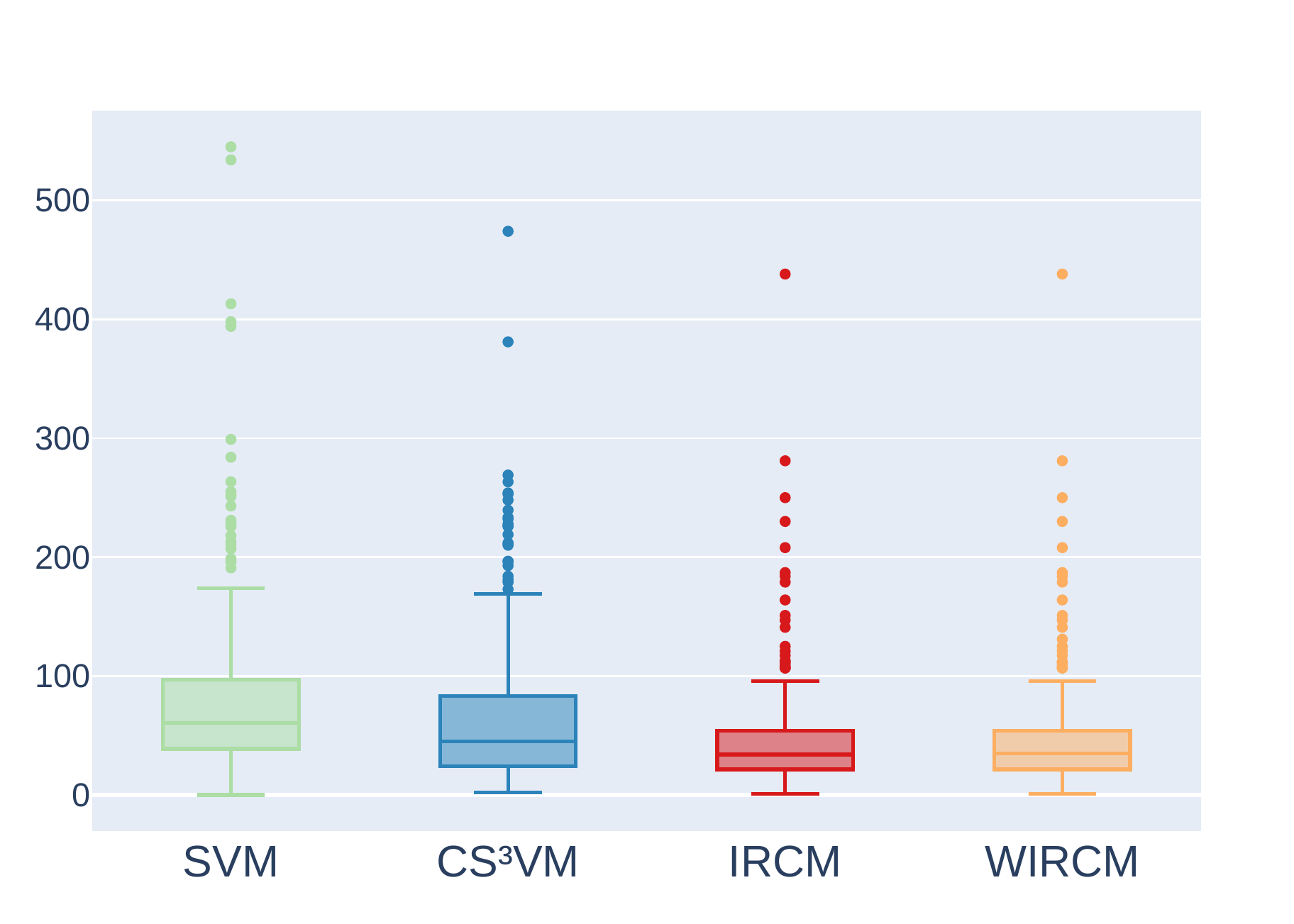}
  \caption{Relative false positive rate $\widehat{\FPR}$ w.r.t.\ the true
    hyperplane; see \eqref{compartrue2}.
    Left: Comparison for all data points.
    Right: Comparison only for unlabeled data points.}
  \label{FPRtru}
\end{figure}


\section{Numerical Results for Simple random samples}
\label{sec:num-results-simple-sample}

In Section~\ref{sec:numerical-results}, we focused our
computational study on non-representative, biased samples.
The common baseline scenario to check the performance of estimators is
to apply them on simple random samples. Hence, for completeness, we
also present the results under simple random sampling. That is, each
unit in the data set has the same probability $\pi_i = n/N$ to be
included into the sample of size $n$.
The instances are the same as described in
Section~\ref{subsection-test-sets}.
The computational setup follows the description in
Section~\ref{subsection-comp-setup}.
As before, the used evaluation criteria are
$\widehat{\AC},$ $\widehat{\PR}$ as in~\eqref{compartrue}
and $\widehat{\RE},$ $\widehat{\FPR}$ as in~\eqref{compartrue2}.

\begin{figure}
  \centering
  \includegraphics[width=0.495\textwidth]{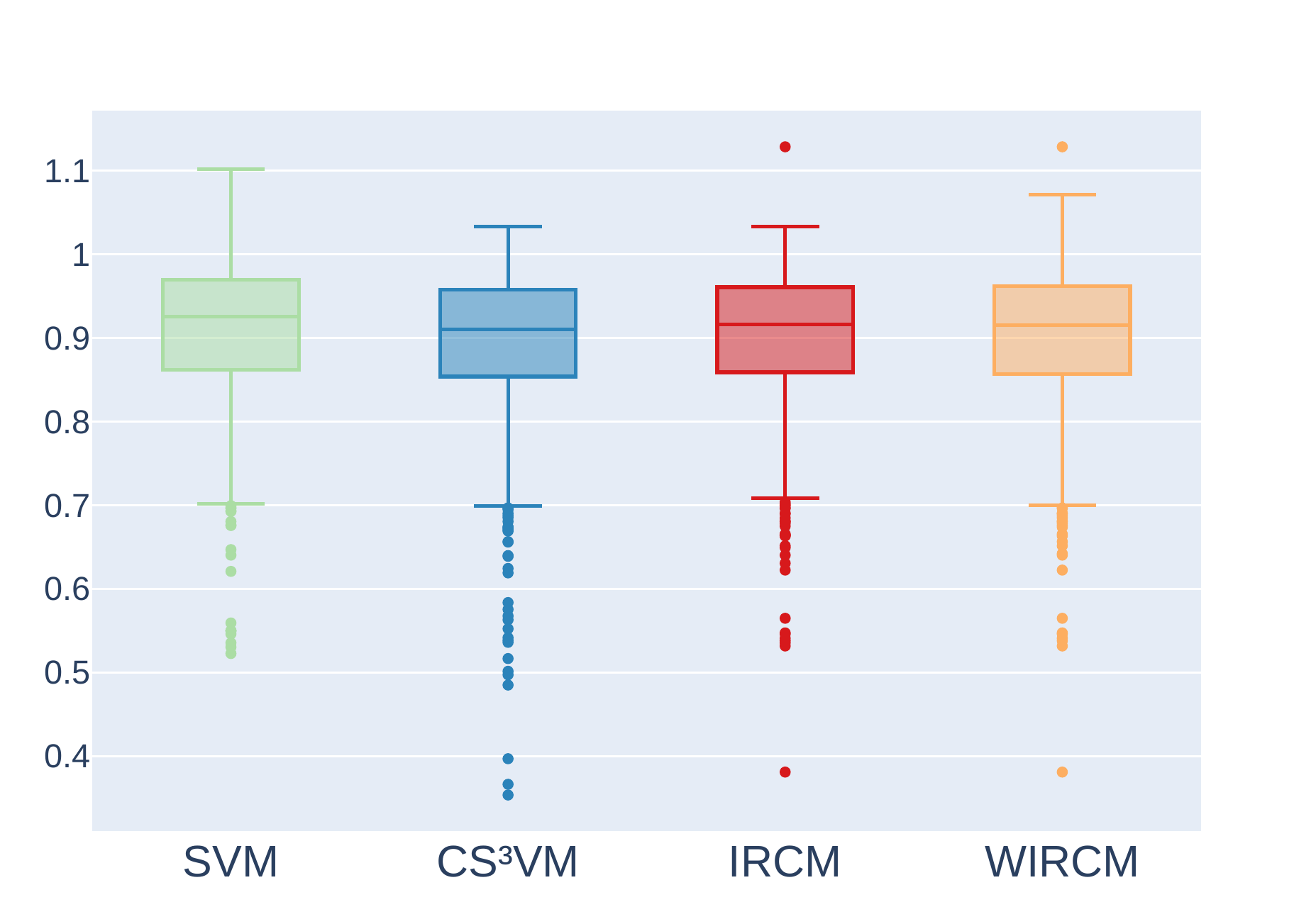}
  \includegraphics[width=0.495\textwidth]{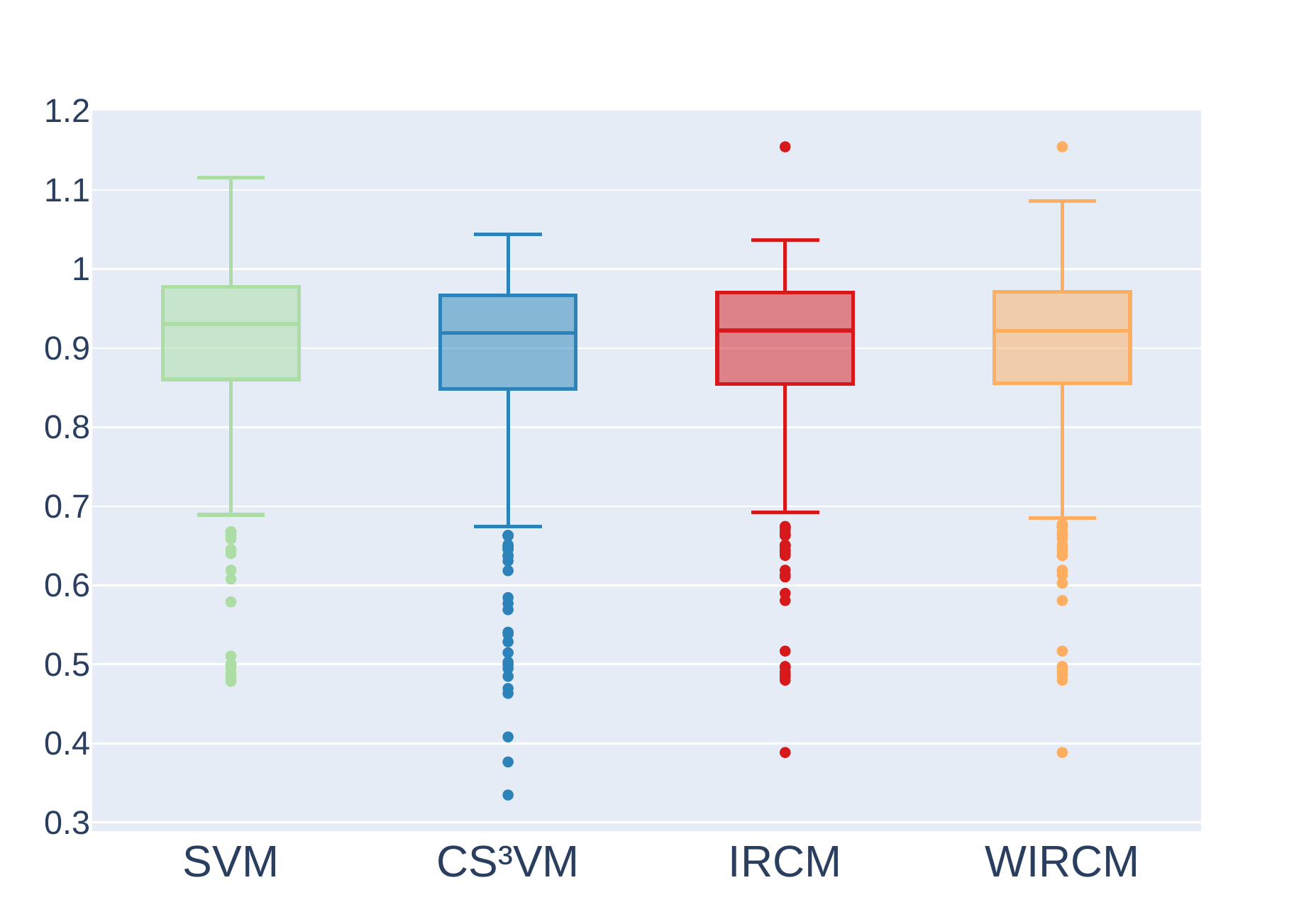}
  \caption{Relative accuracy $\widehat{\AC}$ w.r.t.\ the true
    hyperplane; see~\eqref{compartrue}, for the simple random samples.
    Left: Comparison for all data points.
    Right: Comparison only for unlabeled data points.}
  \label{ACtru2}
\end{figure}

Figure~\ref{ACtru2} and~\ref{PRtru2} show similar accuracy and
precision performance for all approaches. This is as expected, as the
sample is not biased and hence the cardinality constraint does not
contribute relevant additional information to the problem.
Therefore, the SVM does not tend to classify the points as positive as
it is the case for the biased samples.
The outliers, mainly present for CS$^3$VM, are due those instances
that are not solved within the time limit.
As can be seen in Figure~\ref{REtru2} and~\ref{FPRtru2}, recall and
false positive rate are also similar for all approaches.

Hence, for the simple random samples our approaches have almost the
same results as the SVM.
Note that for the biased samples, they outperformed the SVM.
Hence, in cases for which the type of sample is not known,
it is ``safe'' to use the newly proposed approaches for
classification.

\begin{figure}
  \centering
  \includegraphics[width=0.495\textwidth]{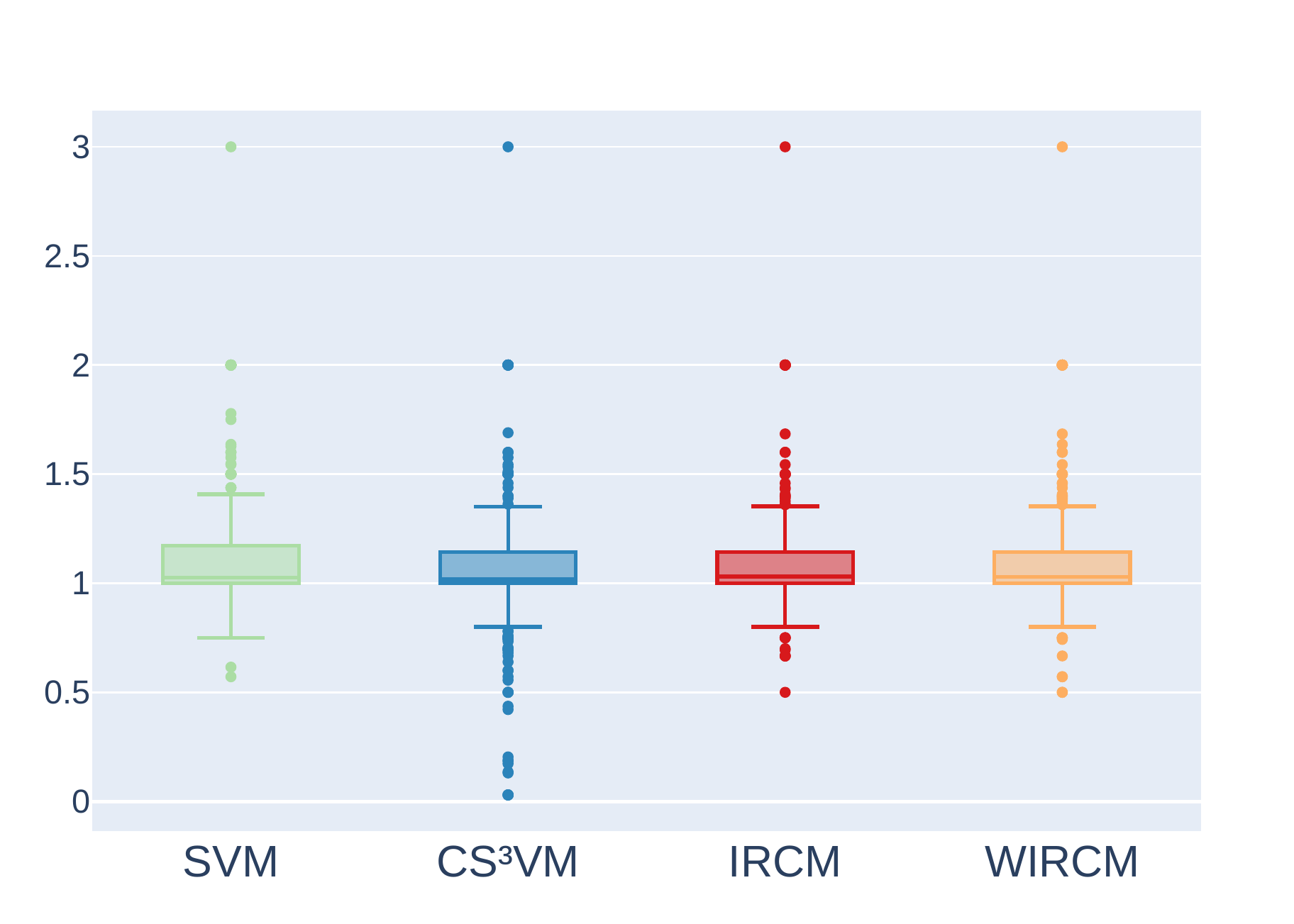}
  \includegraphics[width=0.495\textwidth]{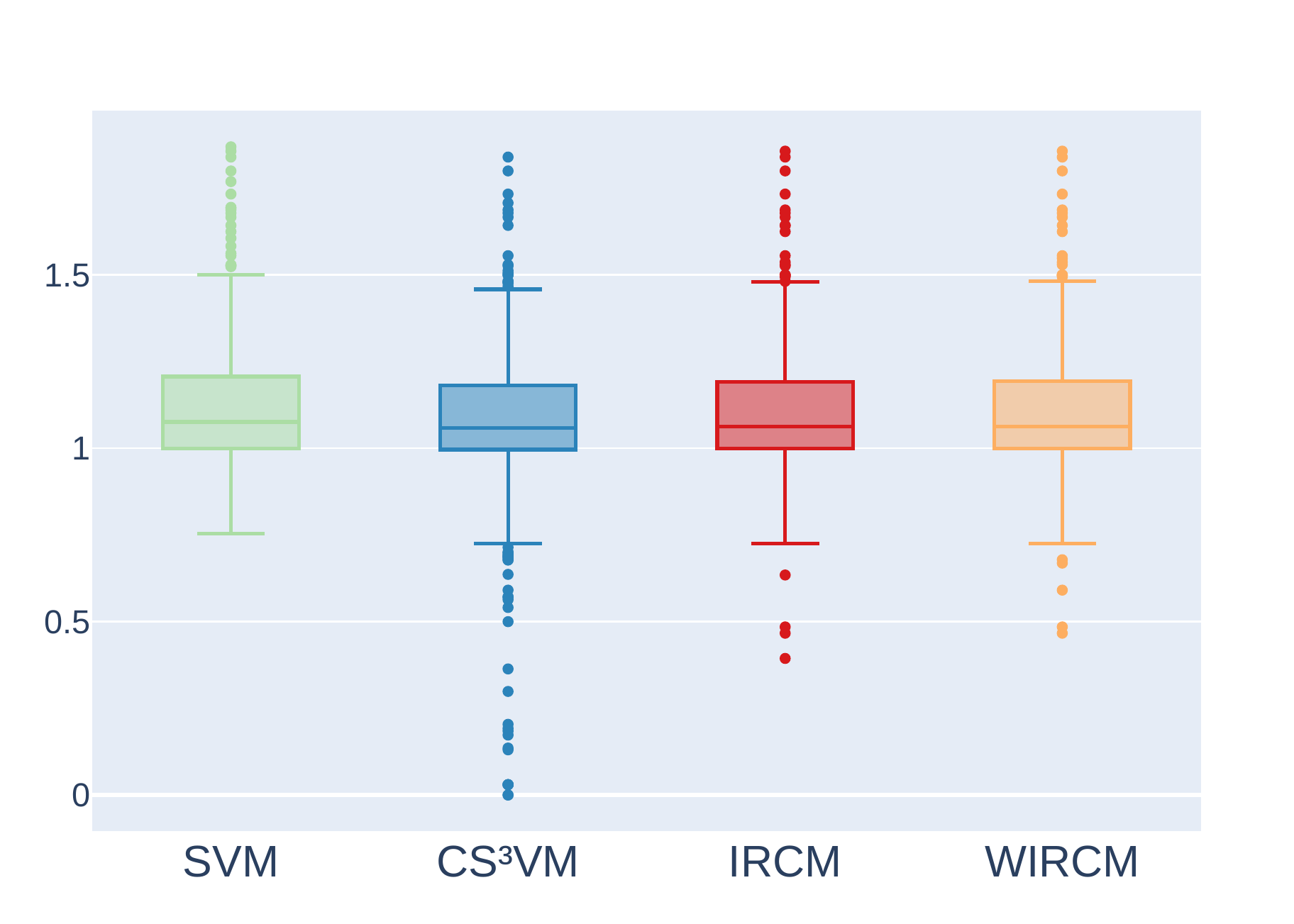}
  \caption{Relative precision $\widehat{\PR}$ w.r.t.\ the true
    hyperplane; see~\eqref{compartrue}, for the simple random samples.
    Left: Comparison for all data points.
    Right: Comparison only for unlabeled data points.}
  \label{PRtru2}
\end{figure}
\begin{figure}
  \centering
  \includegraphics[width=0.495\textwidth]{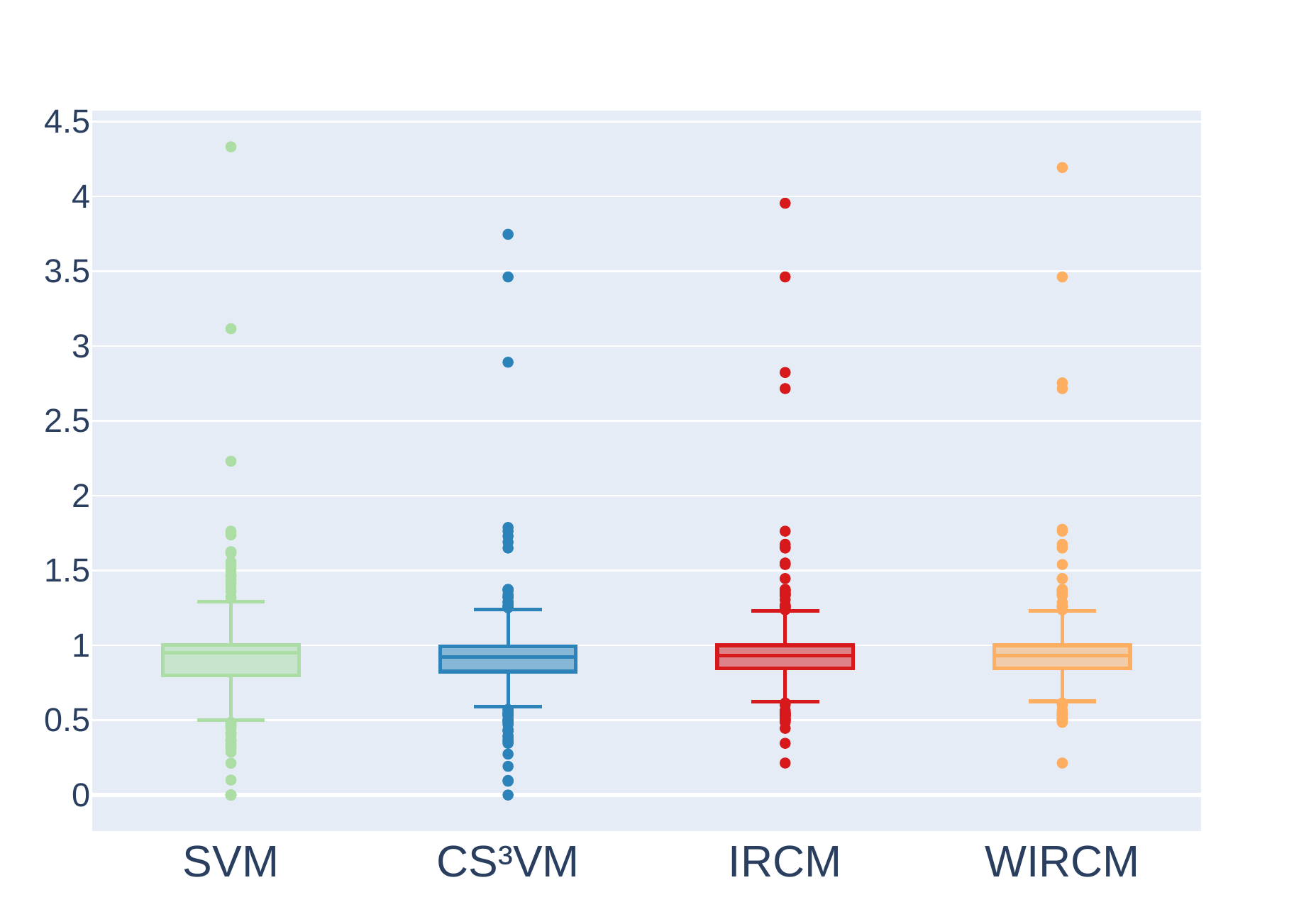}
  \includegraphics[width=0.495\textwidth]{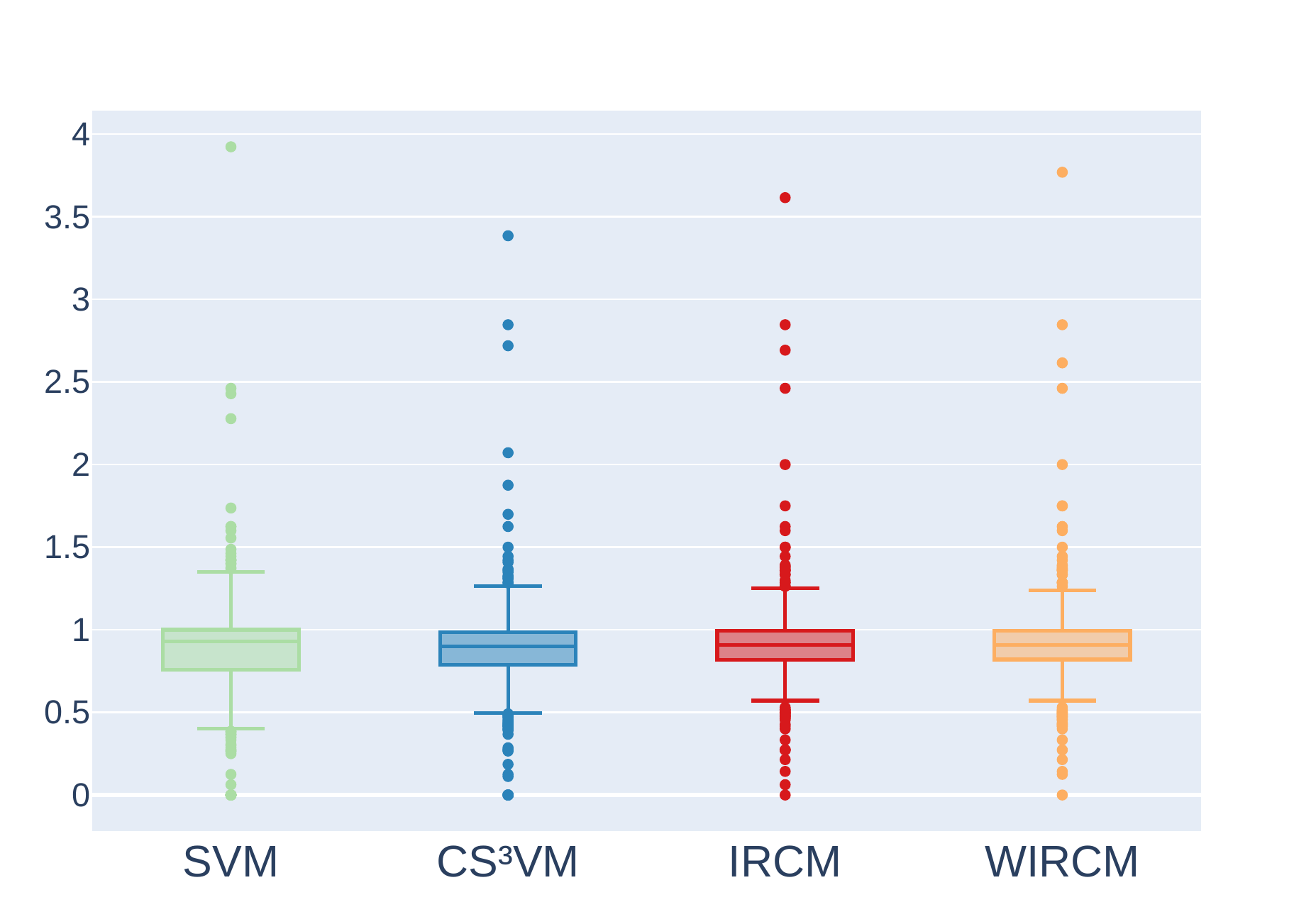}
  \caption{Relative recall $\widehat{\RE}$ w.r.t.\ the true hyperplane;
    see~\eqref{compartrue2}, for the simple random samples.
    Left: Comparison for all data points.
    Right: Comparison only for unlabeled data points.}
  \label{REtru2}
\end{figure}
\begin{figure}
  \centering
  \includegraphics[width=0.495\textwidth]{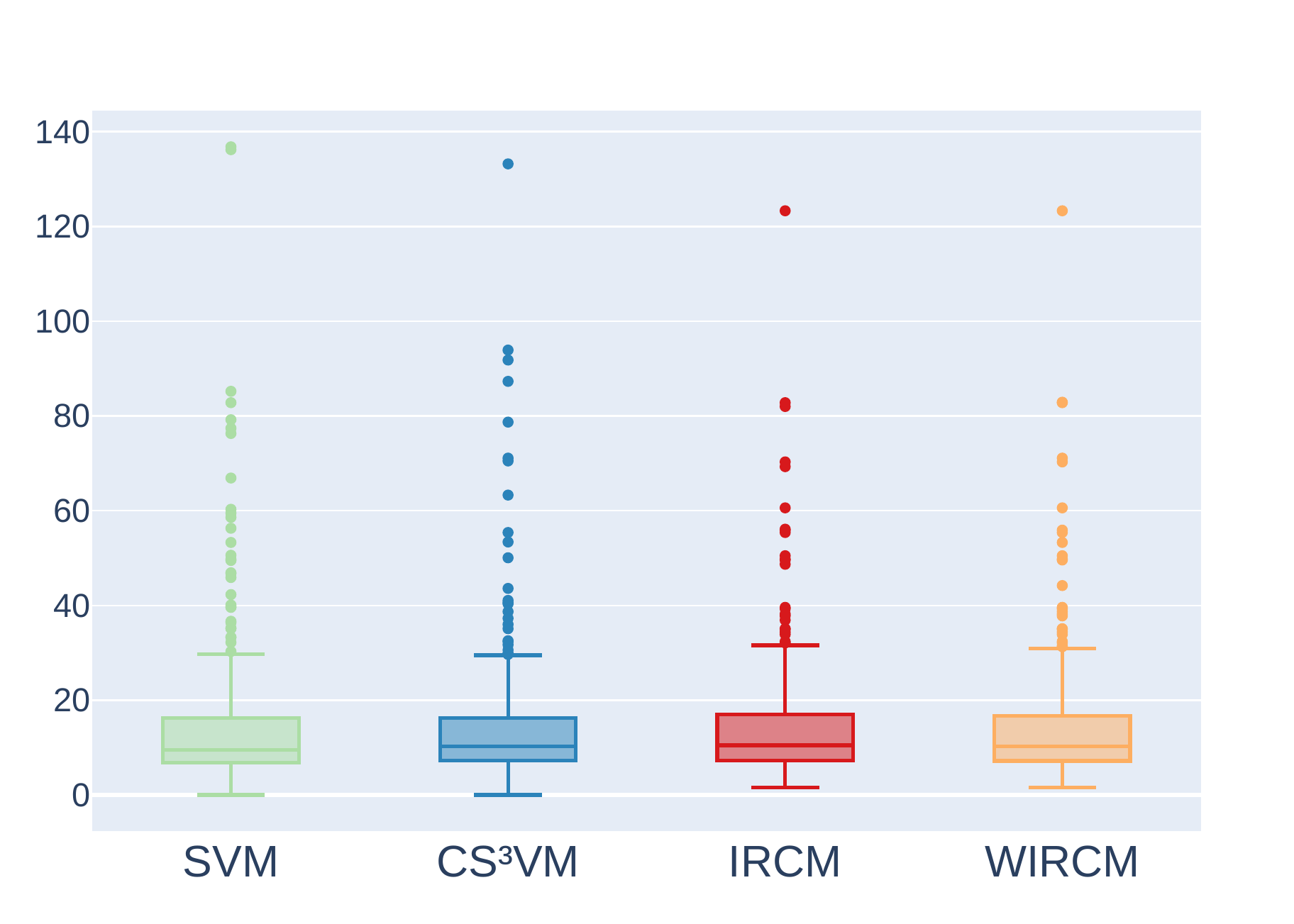}
  \includegraphics[width=0.495\textwidth]{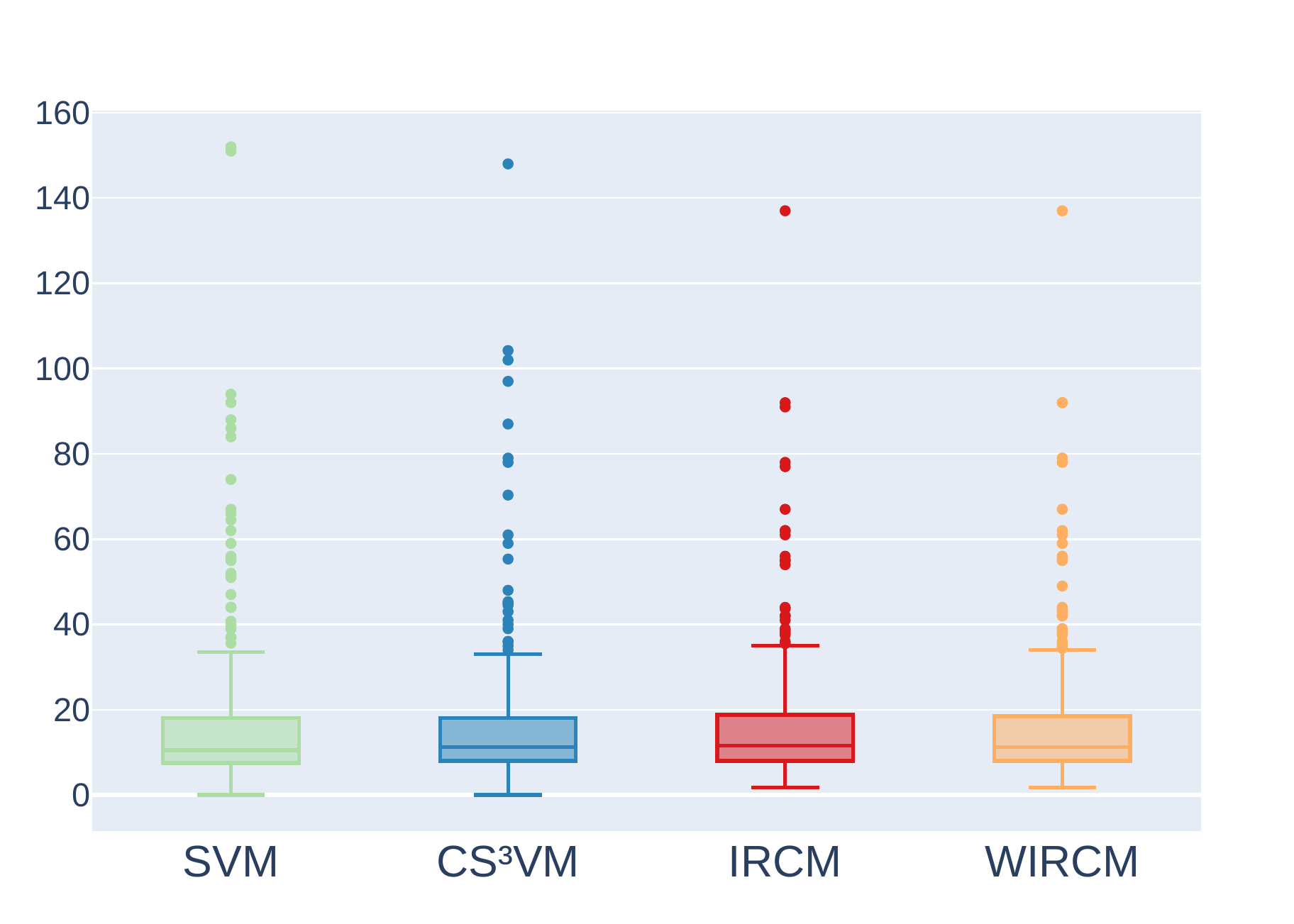}
  \caption{Relative false positive rate $\widehat{\FPR}$ w.r.t.\ the
    true hyperplane; see~\eqref{compartrue2}, for the simple random samples.
    Left: Comparison for all data points.
    Right: Comparison only for unlabeled data points.}
  \label{FPRtru2}
\end{figure}


\rev{\section{Choosing the Hyperparameters}}
\label{sec:sensibilty-parameters}

\rev{Each parameter of Algorithm~\ref{Second version} and~\ref{Scheme
    version} as well as in Problem~\eqref{l2svm} and~\eqref{equation2}
  can be chosen from a range. In Table~\ref{tab:ranges} we present
  plausible ranges for these parameters.}
\begin{table}
  \centering
  \caption{\rev{Plausible ranges for the hyperparameters}}
  \label{tab:ranges}
  \rev{
    \begin{tabular}{lll}
      \toprule
      Parameter & Plausible Range  & Current Choice \\
      \midrule
      $C_1$ & $\mathbb{R}_{\geq 0}$ & $1$ \\
      $C_2$ & $[0.5 C_1, 2C_1]$ & $1$ \\
      $k^1$ & $[2,m]$ & $10,20,50$ \\
      $k^+$ & $[k^1,m]$ & $50 $ \\
      $\hat{\Delta}^1$ & $[0.5, 0.9]$ & $0.8$ \\
      $\tilde{\Delta}$ & $[0.1 ,1-\hat{\Delta}^1]$ & $0.1$ \\
      $B_{\max}$ & $[1, m]$ & $0.2m,0.25m,0.35m,0.45m$ \\
      $\gamma$ & $[1.1, m/B_{\max}]$ & $1.2$ \\
      $T_{\max}$ & $[10, 100]s$& $40s$ \\
      \bottomrule
    \end{tabular}
  }
\end{table}

\rev{Clearly, $C_1 \in \mathbb{R}_{\geq 0}$ holds.
  However, the closer the value is to 1, the more equally important
  are maximizing the margin and minimizing the classification error
  for the labeled data. The range of $C_2$ is based on
  $C_1$ in order to indicate how much more important the unlabeled
  data is compared to the labeled data.
  Again, we choose $C_2 = 1$ so that both data have the
  same importance. Besides that, if $C_2$ is much bigger than $C_1$,
  our preliminary tests showed that this leads to focus on
  minimizing the classification error for the unlabeled data, which
  implies focusing on the binary variable and, hence, leads to larger
  run times.}

\rev{For choosing the other parameters, we consider the first $3$
  datasets presented in Table~\ref{table1} and varied the parameter
  choices in a preliminary numerical study.
  Based on the results, we now discuss how to choose the remaining
  parameters.
  The parameter~$k^1$ can be between $2$ and $m$ since we
  cluster $m$ unlabeled points.
  Note that, the smaller $k^1$, the less time per iteration is needed
  since we have fewer binary variables.
  However, more iterations may be needed to find the solution.
  On the other hand, the bigger $k^1$, the more time per iteration is
  required.
  We choose to start with a small value of~$k^1$ because in
  preliminary numerical tests, when the algorithm terminated, the
  number of clusters never exceeded $m/3$.
  Moreover, in our preliminary tests, if the algorithm exceeds
  $k^t=50$ for some iteration~$t$, it takes a lot of time to solve
  Problem~\eqref{equation3}.
  To decrease this time, we reduced the number of clusters,
  eliminating the ones being far from the hyperplane.
  This is the reason why we choose $k^+ = 50$.}

\rev{The parameter~$\hat{\Delta}^1$ indicates that clusters with a
  distance to the hyperplane greater than the
  $\hat{\Delta}^1$-quantile of all distances will be deactivated.
  It is between~$0.5$ and~$0.9$ because a smaller value than~$0.5$
  means removing points that are too close to the hyperplane.
  This implies that in next iterations many clusters can be
  reactivated.
  On the other hand, if it is larger than~$0.9$, it means that almost
  no clusters can be deactivated.
  We choose $0.8$ because in our preliminary numerical
  tests we noticed that with a smaller value, many clusters were
  activated again, which increased the required time per iteration.
  The range of $\tilde{\Delta}$ is justified by the fact that for
  all~$t$, the maximum value of $\hat{\Delta}^t$ is $1$.
  We chose $0.1$ because the higher the value we choose, the smaller the
  possibility to eliminate clusters becomes.
  If chosen smaller, $\hat{\Delta}^t$ and $\hat{\Delta}^{t+1}$
  would be very similar and some clusters would be deactivated and
  reactivated several times.}

\rev{Because we have $m$~unlabeled points, we can fix at most $m$
  unlabeled points, which justifies the range of $B_{\max}$ and the
  maximum value of $\gamma$.
  Since some points are not fixed on some side---they may be on the
  wrong side or it could take more than $T_{\max}$ to solve
  Problem~\eqref{equation5}---we try to fix at least more than
  \SI{10}{\percent} of $B_{\max}$ many unlabeled points.
  This is why the minimum value of $\gamma$ is 1.1.
  The maximum value of $T_{\max}$ is \SI{100}{\second} because,
  if chosen smaller, we observe that there is often not enough time to
  solve Problem~\eqref{equation5}.
  On the other hand, if it is larger, we observe that the time needed
  to solve the Algorithm~\ref{Scheme version} increases.}


\rev{\section{Run Times in Dependence of the Number of Data Points}}
\label{sec:performancepersize}

\rev{In this section, we complement
  Section~\ref{sec:numerical-results} by presenting the run times in
  dependence of the number of points in the data set in order to shed
  some light on the scalability of our approaches.
  To this end, we split the entire data set in three subsets.}

\begin{figure}
    \centering
    \includegraphics[width=0.7\textwidth]{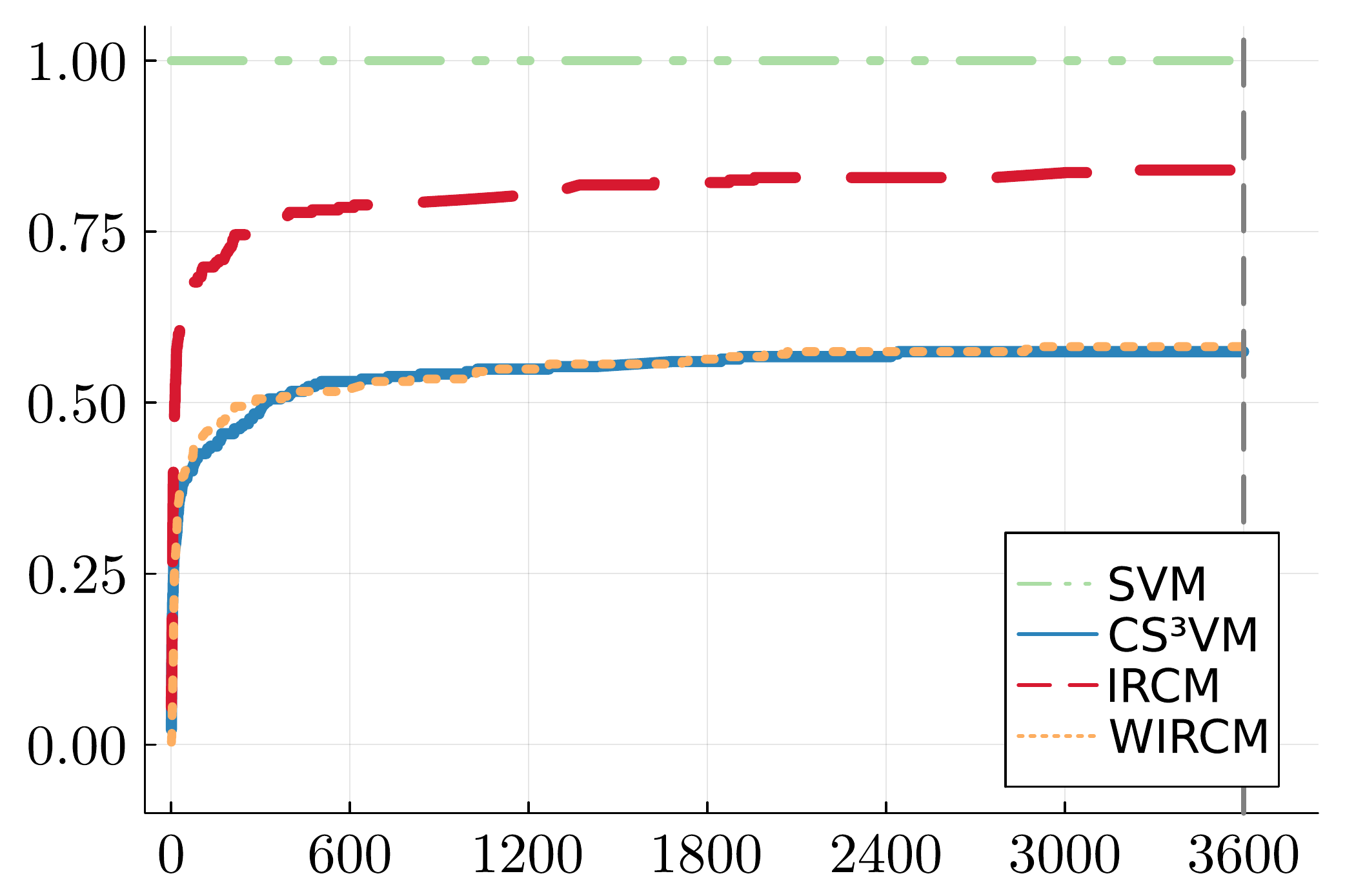}
    \includegraphics[width=0.7\textwidth]{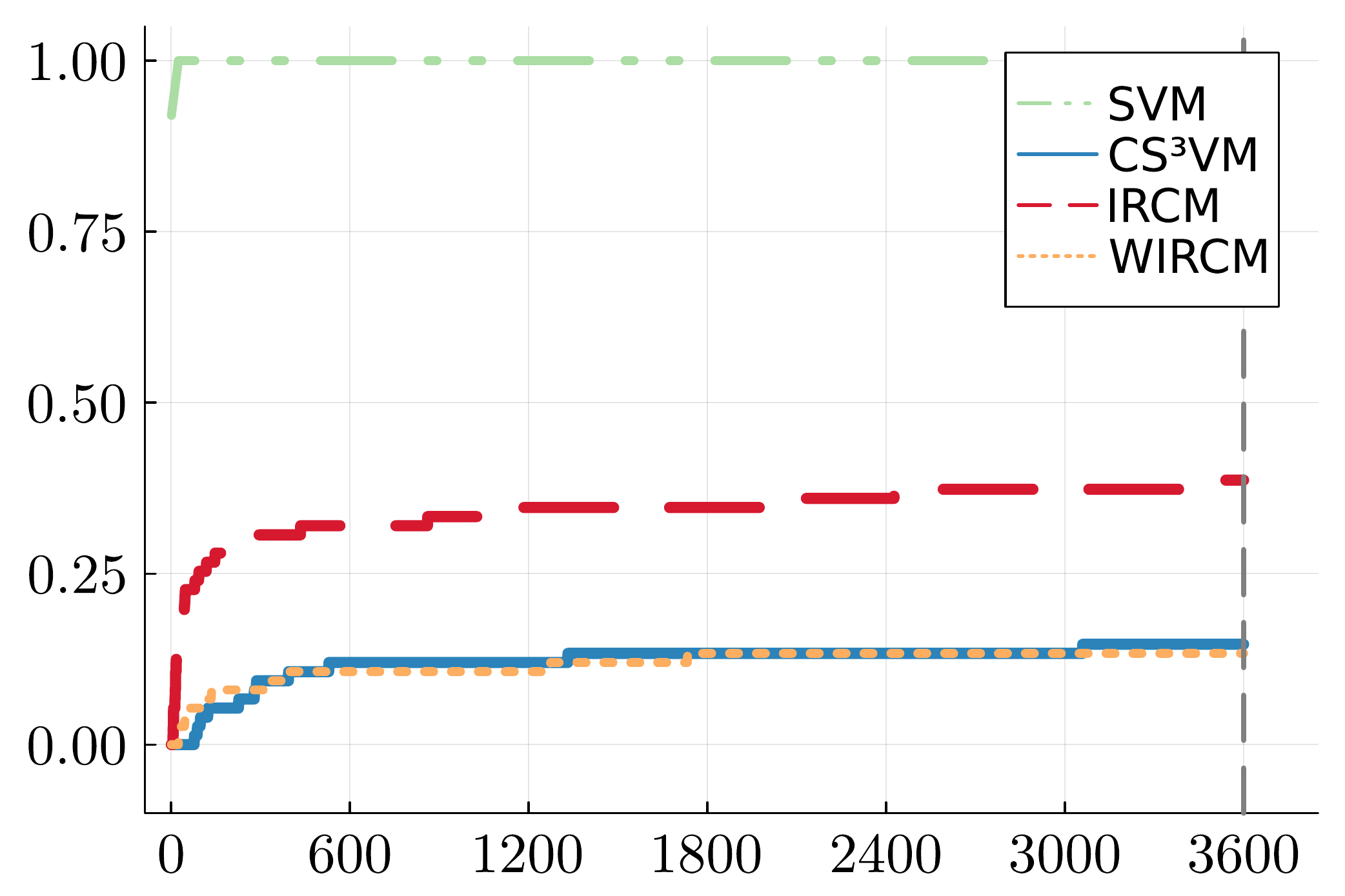}
    \includegraphics[width=0.7\textwidth]{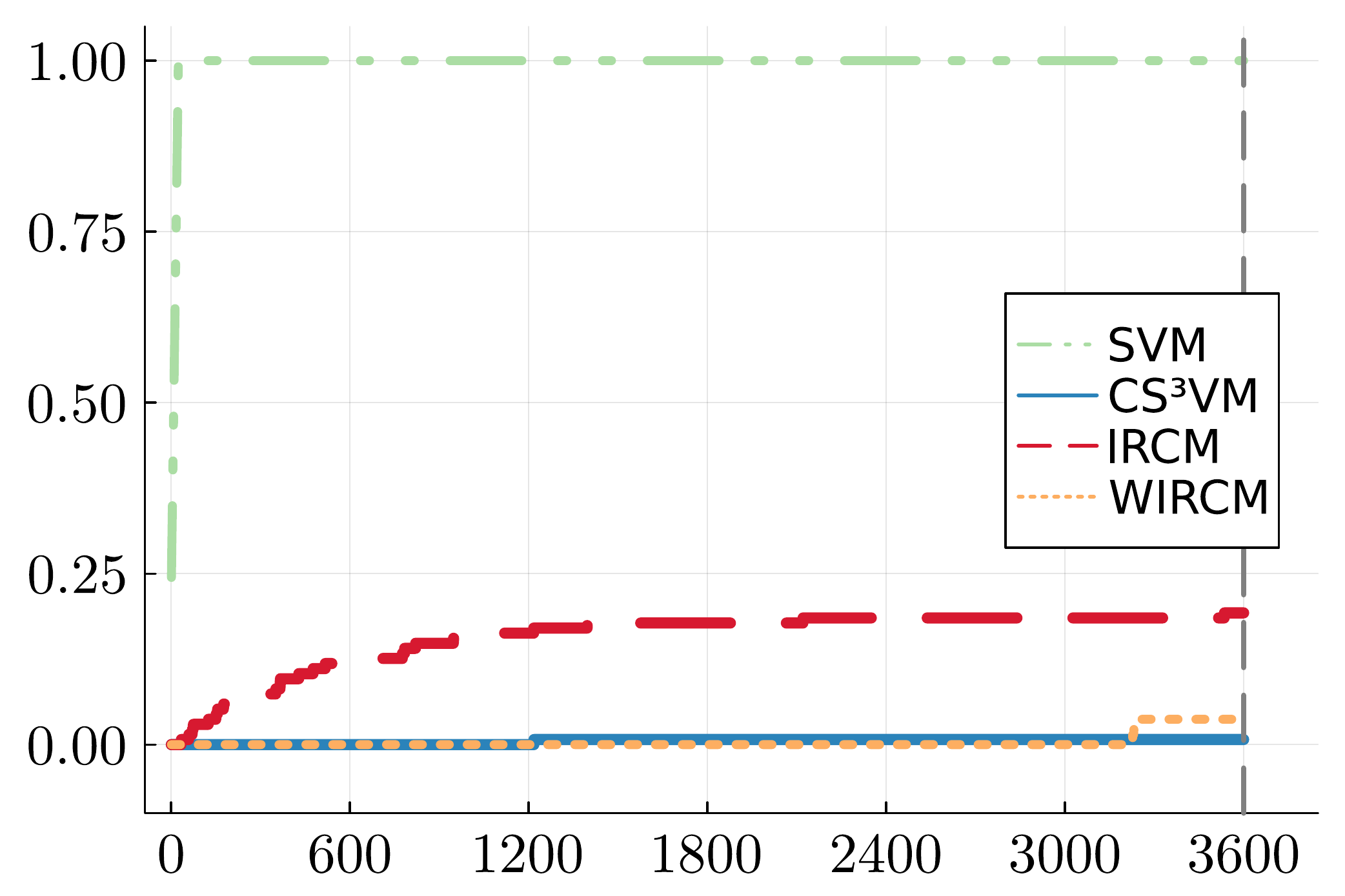}
    \caption{\rev{ECDFs for run time (in seconds).
        Top: Instances with $N \geq 500$.
        Middle: Instances with $N \in  (500,1500]$.
      Bottom: Instances with $N > 1500$.}}
    \label{fig:mainfig}
\end{figure}

\rev{The first subset only considers those 46 data sets with $N \leq
  500$.
  As can be seen in Figure~\ref{fig:mainfig} (top), IRCM
  solves more than \SI{75}{\percent} of the instances while CS$^3$VM
  and WIRCM solve more than \SI{50}{\percent}.
  The second subset contains 11 data sets with $N \in (500,1500]$.
  Figure~\ref{fig:mainfig} (middle) shows that for these test sets, IRCM solves
  about \SI{40}{\percent} of the instances while CS$^3$VM and WIRCM
  solve more than \SI{10}{\percent}.
  The last subset contains those 21 data sets with $N > 1500$.
  Figure~\ref{fig:mainfig} (bottom) shows that CS$^3$VM and WIRCM do not solve any of
  these instances and IRCM solves about \SI{20}{\percent}.
  As expected, the larger the number of points and, thus, the larger
  the number of binary variables, the more challenging it is to solve
  the instances.
  Besides that, SVM solves all instances, which is expected since it
  does not include any binary variables.}


\end{document}